\PassOptionsToPackage{unicode}{hyperref}
\PassOptionsToPackage{hyphens}{url}
\PassOptionsToPackage{dvipsnames,svgnames,x11names}{xcolor}
\documentclass[aap,preprint]{imsart}

\RequirePackage[T1]{fontenc}
\RequirePackage{amsthm,amsmath}
\RequirePackage[round]{natbib}
\RequirePackage[colorlinks,citecolor=blue,urlcolor=blue,linkcolor=blue]{hyperref}

\input{macro.tex}
\usepackage{booktabs}
\usepackage{enumitem}

\setlist[description]{font=\normalfont\itshape}
\usepackage[utf8]{inputenc}

\usepackage{algorithm}

\usepackage{algpseudocode}
\algrenewcommand\algorithmicif{\textsf{if}}
\algrenewcommand\algorithmicthen{\textsf{then}}\algrenewcommand\algorithmicend{\textsf{end}}\algrenewcommand\algorithmicelse{\textsf{else}}
\makeatletter
\newcounter{ALG@uniq}
\let\OldALG@step\ALG@step
\renewcommand{\ALG@step}{\OldALG@step\refstepcounter{ALG@uniq}}
\makeatother
\ifPDFTeX\else  
\fi
\IfFileExists{upquote.sty}{\usepackage{upquote}}{}
\IfFileExists{microtype.sty}{
  \usepackage[]{microtype}
  \UseMicrotypeSet[protrusion]{basicmath} 
}{}

\usepackage{xcolor,csquotes}
\IfFileExists{xurl.sty}{\usepackage{xurl}}{} 
\hypersetup{
  pdftitle={Sticky CIR process with potential: invariant measure 
       and exact sampling},
  pdfauthor={TS},
  colorlinks=true,
  linkcolor={blue},
  filecolor={Maroon},
  citecolor={Blue},
  urlcolor={Blue},
  pdfcreator={}}

\usepackage{longtable,booktabs,array}
\usepackage{calc} 

\IfFileExists{footnotehyper.sty}{\usepackage{footnotehyper}}{\usepackage{footnote}}
\makesavenoteenv{longtable}
\usepackage{graphicx}

\usepackage{amsthm}
\usepackage[capitalize]{cleveref}

\startlocaldefs
\numberwithin{equation}{section}
\theoremstyle{plain}
\newtheorem{theorem}{Theorem}[section]
\newtheorem{lemma}[theorem]{Lemma}
\newtheorem{corollary}[theorem]{Corollary}

\newtheorem{proposition}[theorem]{Proposition}

\theoremstyle{definition}
\newtheorem{definition}[theorem]{Definition}
\newtheorem{assumption}[theorem]{Assumption}
\newtheorem{remark}[theorem]{Remark}
\endlocaldefs

\crefname{remark}{Remark}{Remarks}
\Crefname{remark}{Remark}{Remarks}

\Crefname{corollary}{Corollary}{Corollaries}

\Crefname{proposition}{Proposition}{Propositions}

\Crefname{assumption}{Assumption}{Assumptions}
\Crefname{appendix}{}{}

\AtBeginDocument{\renewcommand*{\proofname}{Proof}}


\makeatletter
\@ifpackageloaded{caption}{}{\usepackage{caption}}
\@ifpackageloaded{subcaption}{}{\usepackage{subcaption}}
\makeatother

\usepackage{bookmark}
\begin{document}

\begin{frontmatter}

\title{Sticky CIR process with potential: invariant measure and exact sampling}
\runtitle{Sticky CIR: invariant measure and sampling}

\begin{aug}
\author{\fnms{Tony}~\snm{Shardlow}\ead[label=e1]{t.shardlow@bath.ac.uk}}
\runauthor{T. Shardlow}
\address{Department of Mathematical Sciences,
University of Bath,
Bath BA2 7AY,
United Kingdom.
\printead{e1}}
\end{aug}

\begin{abstract}
We study the sticky Cox--Ingersoll--Ross (CIR) process in one dimension,
a diffusion on $[0,\infty)$ with a sticky boundary condition at the origin,
arising as the marginal process in a sparse Bayesian inference framework
based on Hadamard--Langevin dynamics. For the parameter range $\delta\in(1,2)$,
in which the origin is accessible but not absorbing, we prove well-posedness
of the process and uniqueness of its invariant measure, which is a mixture
of a point mass at zero and a weighted gamma-type density on the interior.
We derive an explicit Green's function in terms of
confluent hypergeometric functions, and use this to construct an exact
sampler for the invariant measure in the zero-potential case. For a
non-trivial potential $G$, we establish existence and uniqueness of the
tilted invariant measure via a Girsanov change of measure, and develop
two sampling algorithms: a Metropolis--Hastings corrected sampler that
targets the invariant measure exactly, and a cheaper, biased unadjusted Langevin
algorithm (ULA) for a boundary-clamped variant, for which we prove a first-order
expansion of the stationary bias with an explicit constant, its leading term of
order $h\abs{\log h } $, independent of $\delta$.
Numerical experiments confirm the predicted behaviour: the Metropolis--Hastings sampler achieves the target invariant measure at all step sizes, while the ULA bias follows the proven first-order law, including its constant.
\end{abstract}



\end{frontmatter}


\section{Introduction}
We establish well-posedness and develop sampling methods for the invariant measure of the following sticky CIR process. 
For potential function $G\colon [0,\infty) \to \mathbb{R}$, stickiness parameter $\mu>0$, dissipation parameter $\lambda>0$, CIR parameter $\delta\in(1,2)$,  inverse temperature $\beta>0$, and one-dimensional Brownian motion $W_t $, consider
\begin{gather}
\begin{split}
  \label{eq-CIR-potential}
  du_t &= \left[\frac{\delta-1}{\beta\, u_t} - \lambda\, u_t - G'(u_t)\right] dt + \sqrt{\frac{2}{\beta}}\,dW_t , \\
  \mathbf{1}_{\{u_t=0\}}\,dt&=\frac{\exp(-\beta G(0))}{\mu}d\ell_t^0(u),
\end{split}
\end{gather}
with initial data $u_0=x>0$. It turns out that solutions $u_t$ are always non-negative and that, for the parameter range $\delta\in (1,2)$, the origin $u=0$ is accessible but not absorbing. The second equation, known as the \enquote{sojourn condition}, describes the stickiness at the boundary $u=0$. The sojourn condition is given in terms of diffusion local time $\ell_t^0$ (see \cref{diffusion-local-time}).

\subsection{Motivation}\label{sec:motivation}

The motivation for this work is the sparse sampling methodology of \citet{Cheltsov2025-no}, who use \cref{eq-CIR} with $\delta=2$ to sample distributions with an $\norm{\cdot}_1$-regularised log-density. For $\delta=2$, the origin $u=0$ is inaccessible, so sample components are small but never exactly zero; the $\norm{\cdot}_1$ prior encourages sparsity, but the dynamics cannot deliver it. Reducing $\delta$ into the range $(1,2)$ weakens the repulsive drift $(\delta-1)/(\beta u)$ enough that $u_t$ reaches the boundary in finite time; the sticky condition there prevents absorption. The resulting invariant measure of \cref{eq-CIR-potential} carries a point mass at zero, delivering genuine sparsity, mixed with a continuous density on $(0,\infty)$, exactly the structure of the spike-and-slab prior \citep{Mitchell1988,GeorgeMcCulloch1993}.

To place this in a Bayesian inverse-problem context, consider recovering an unknown $\vec X \in \mathbb{R}^d$ from observations $\vec y = \vec O(\vec X)$ with likelihood density $\rho(\vec y - \vec O(\vec X))$. Following \citet{Cheltsov2025-no}, write $\vec X = \vec u \odot \vec v$ as a Hadamard product with $(\vec u, \vec v) \in (\mathbb{R}^+)^d \times \mathbb{R}^d$, and place a product prior $\pi_0$ on $(\vec u,\vec v)$: each $u$-component drawn from the invariant measure of \cref{eq-CIR-potential} with $G=0$ (giving sparse $\vec u$, hence sparse $\vec X$), and each $v$-component Gaussian. Writing the likelihood as $\rho(\vec y - \vec O(\vec u \odot \vec v)) = \exp(-\beta G(\vec u \odot \vec v))$ for a potential $G$, the Bayesian posterior is the Gibbsian reweighting of the prior:
\[
  \frac{d\pi}{d\pi_0}(\vec u,\vec v)
   \propto  \exp\bigl(-\beta G(\vec u \odot \vec v)\bigr).
\]

The central observation behind our approach is that this same Gibbsian reweighting also describes the invariant measure of the one-dimensional \cref{eq-CIR-potential} (\cref{thm-inv-potential}). In one dimension, sampling the posterior therefore reduces to simulating the SDE: the $G\equiv 0$ sticky CIR provides exact draws from the prior, and the potential $G$ enters as a Radon--Nikodym factor handled either by Metropolis--Hastings correction or by an unadjusted Langevin step. The present paper develops this one-dimensional case in detail; the full $d$-dimensional sampler, which must couple the $\vec u$ and $\vec v$ components, is left for future work.

Beyond this sampling motivation, the paper contributes to a sparsely populated area: the \emph{numerical} approximation of sticky SDEs.  Sticky diffusions are classical objects analytically \citep{RogersWilliams2000, BorodinSalminen2002}, but discretisation schemes for them and their convergence theory remain scarce, the recent work of \citet{Sharma2025} being one of the few dedicated treatments.  The novelty of our contribution is twofold: the exact, resolvent-based samplers for the invariant measure (\cref{alg:sticky-cir,alg:mcmc}), and the sharp first-order law for the bias of the unadjusted scheme (\cref{thm:sharp-rate}); the well-posedness and the invariant-measure characterisation adapt the established sticky-diffusion theory \citep{Peskir2022,PeskirRoodman2023}.  This bias law also places the sticky CIR within the error theory for unadjusted discretisations: for a smooth potential on $\mathbb{R}^d$, the Euler--Maruyama discretisation of the overdamped Langevin diffusion has a first-order stationary bias $\pi_h(f)-\pi(f)=C(f)\,h+O(h^2)$ with a smooth leading coefficient \citep{TalayTubaro1990}. The associated unadjusted Langevin algorithm (ULA) is geometrically ergodic with an $O(h)$ bias under suitable convexity \citep{RobertsTweedie1996,DurmusMoulines2017}.  The sticky CIR falls outside this $O(h)$ law on two counts---the diffusion degenerates at an \emph{accessible} boundary, and its invariant measure carries an atom alongside a cusped interior density---so the stationary bias is instead of exact order $h\abs{\log h}$, strictly larger than $O(h)$, with the classical rate recovered precisely when the boundary is inactive.

\subsection{Outline of the paper}
The remainder of this paper is organised as follows.
\cref{sec:background} collects background material on one-dimensional
diffusions, including the scale function, speed measure, and diffusion local
time, and reviews the well-posedness of the instantaneously reflecting CIR
process (\cref{thm-wp}), which serves as the foundation for
the sticky construction.
\cref{sticky-cir-process} introduces the sticky CIR process with
zero potential: well-posedness is proved via an It\^o--McKean time-change
(\cref{thm-sticky-wp}), the unique invariant measure is
characterised as a mixture of a point mass and a weighted gamma-type density
(\cref{thm-ergodicity}), and the Wentzell boundary condition on the
transition density is identified (\cref{thm-pdf0}).
\cref{sec-sampling} derives  the exact transition distribution at
exponential times (\cref{thm:transition-dist}) and an exact sampler for the
zero-potential invariant measure (\cref{alg:sticky-cir}).
\cref{sec:potential} extends the analysis to a non-trivial potential $G$.
Existence and uniqueness of the tilted invariant measure $\pi$ are established
via a Girsanov change of measure (\cref{thm-exist-potential,thm-inv-potential}),
and two samplers are developed: a Metropolis--Hastings scheme targeting $\pi$
exactly (\cref{alg:mcmc}) and a cheaper, biased unadjusted Langevin scheme
(\cref{alg:ula}).  The analysis of the ULA bias is  a consistency--stability argument, which provides a first-order expansion pinning
down the exact rate and its leading constant (\cref{thm:sharp-rate}); the
uniform-in-$h$ geometric ergodicity underlying these bounds is established for
non-convex potentials with bounded gradient
(\cref{lem:contraction-nonconvex}). 
Numerical experiments in \cref{sec:experiments}
confirm the theoretical predictions and compare the two algorithms across a
grid of potentials, stickiness parameters, and step sizes.  
Appendix~\ref{app:ula-bias} houses many of the proofs and supporting lemmas.
\section{Background}\label{sec:background}

We collect standard results on one-dimensional diffusions that
underpin the analysis, including the well-posedness of the reflecting CIR process in \cref{thm-wp}; they are recalled to fix notation
and to make the paper self-contained.  

\subsection{Scale function, speed measure, and local time}\label{review-regular-diffusion-processes}

\begin{definition}[regular
diffusion]\label{def-regular-diffusion}

A \emph{regular diffusion} on an interval \(I \subseteq \mathbb{R}\),
in the sense of Itô--McKean, is a continuous strong Markov process
\(X = (X_t)_{t \geq 0}\) where every point of \(I\) is \emph{accessible} from every other point with positive probability. In terms of  $\tau_y=\inf\{t>0\colon X_t=y\}$, this means, for all \(x\) in the interior of $I$ and all \(y  \in I\), that
  \(
  \mathbb{P}(\tau_y<\infty \mid X_0=x) > 0
  \). See \citep[Definition~45.2]{RogersWilliams2000} or \cite[Chapter~15.1]{karlintaylor1981} or \cite[Chapter VII, before Proposition~3.1]{RevuzYor1999}.
\end{definition}
Every regular diffusion on \(I\) is associated to a strictly increasing function \(s\colon I \to \mathbb{R}\) (the \emph{scale
  function}) and a Radon measure \(m\) on \(I\) with full support (the
  \emph{speed measure}). Explicit definitions for $s$ and $m$ are given in terms of hitting probabilities and expected occupation times in \cite[Section~4.1]{ItoMcKean1965} or \cite[Proposition~3.2 and Theorem~3.6]{RevuzYor1999}. Using $s$ and $m$, the generator \(\mathcal{L}\) of a regular diffusion \(X\) is
\begin{equation}\label{eq:generator-scale-speed}
  \mathcal{L}f = \frac{d}{dm}\!\left(\frac{df}{ds}\right),
\end{equation}
for \(f\) in the domain of \(\mathcal{L}\), where $df/ds$ denotes the derivative of $f$ with respect to the scale function (i.e., $\lim_{y\to x}(f(y)-f(x))/(s(y)-s(x))$), and $d/dm$ denotes the Radon--Nikodym derivative with respect to the speed measure.   This operator is self-adjoint in $L^2(m)$ and the Lebesgue adjoint operator is \begin{equation}
  \label{eq-adjoint-generator}
    (\mathcal{L}^*f)(x) = \frac{d}{dx}\pp{\frac{1}{s'(x)}\frac{d}{dx}\frac{f(x)}{m'(x)}},
  \end{equation}
  where $m'(x)$ is the Lebesgue density of the speed measure $m$ and $s'(x)$ is the derivative of the scale function.     In the case of a smooth drift $b\colon I\to \mathbb R$ and smooth non-zero diffusion $\sigma\colon I\to \mathbb R$, the diffusion with generator $\mathcal L=\frac 12 \sigma(x)^2\frac{\partial^2}{\partial x^2}+b(x)\frac{\partial}{\partial x}$ has, for $x$ in the interior of $I$,  \begin{equation}
    \label{eq-scale-speed-sde}
    s(x) = \int^x \exp\left(-2\int^y \frac{b(z)}{\sigma^2(z)}\,dz\right) dy,  \quad m'(x)=\frac{2}{\sigma^2(x)s'(x)},
  \end{equation}
  where the scale function is only defined up to an additive constant. See
 \cite[Chapter~VII, Theorem~3.12 and Exercise~3.20]{RevuzYor1999} or \cite[Section~4.2]{ItoMcKean1965}.

Intuitively, the scale function $s$ captures the drift of the process so that  the transformed process $s(X_t)$ is a local martingale on $s(I)$. As a consequence, by optional stopping, the hitting times $\tau_a,\tau_b$ to levels $a<x<b$ satisfy
\[
  \mathbb{P}(\tau_a < \tau_b \mid X_0 = x) = \frac{s(b) - s(x)}{s(b) - s(a)},
\]
so $s$ acts as a coordinate in which exit probabilities are linear in position. The speed measure $m$ captures how long the process lingers in each region; areas with high density with respect to $m$ are visited more frequently and for longer durations.  In particular, when the speed measure has an atom at the boundary, that point is said to be  \emph{slowly reflecting} in \cite[Chapter VII, Definition 3.11]{RevuzYor1999} or \emph{sticky}, where the process sojourns or temporarily remains (in contrast to absorbing boundary points where the process remains forever). The speed measure also yields the invariant measure: if $m$ is finite, then the invariant measure is $m$ normalised to be a probability measure.

\begin{remark}
  Some authors (e.g., \cite[Chapter~15]{karlintaylor1981} or \cite[Theorem 52.1]{RogersWilliams2000}) define the speed measure with an additional factor of $1/2$, so $m'(x) = 1/(\sigma^2(x)s'(x))$. This changes some formulas by a factor of $2$ but does not affect the underlying theory.
\end{remark}

The sojourn condition is expressed in terms of diffusion local time \cite[Theorem 49.1]{RogersWilliams2000}.
It is the natural local time for diffusions and is defined via the speed measure.
\begin{definition}[diffusion local time]\label[definition]{diffusion-local-time}
  For a diffusion $X$ with speed measure $m$, the \emph{diffusion local time} $\ell_t^a$ at level $a$ is defined by
\[
\ell^a_t(X) = \lim_{\varepsilon \downarrow 0} \frac{1}{m((a-\varepsilon, a+\varepsilon))} \int_0^t \mathbf{1}_{\{|X_s - a| < \varepsilon\}}\,ds.
\]
The following occupation-time formula holds: 
for measurable $f \geq 0$,
\begin{equation}\label{otf}
\int_0^t f(X_s)\,ds = \int f(a)\,\ell^a_t(X)\,m(da).
\end{equation}
\end{definition}

When $m$ has an atom and the process sojourns at the boundary, this definition applies naturally. If $m(\{0\}) = 1/\mu$, for some $\mu>0$, then taking the limit gives
\[
\ell^0_t(X) = \mu \int_0^t \mathbf{1}_{\{X_s = 0\}}\,ds.
\]
This is often written $\mathbf{1}_{\{X_t = 0\}}\,dt = \frac{1}{\mu}\,d\ell^0_t(X)$ and known as the \enquote{sojourn condition}. It says the length of the sojourn at the boundary is determined by the diffusion local time and stickiness parameter $\mu$; the limit \(\mu \to 0\) gives absorption and infinite time at the boundary, and
\(\mu \to \infty\) gives instantaneous reflection and zero time at the boundary. 

\begin{remark}
  Diffusion local time $\ell_t^a$ differs from the  local time $L_t^a$ defined via the Tanaka formula~\cite[Chapter VI, Corollary 1.2]{RevuzYor1999}. In that case, the occupation-time formula \cite[Chapter VI, Corollary 1.6]{RevuzYor1999} holds with respect to the quadratic variation: for a continuous semimartingale $X$ and measurable $f\geq 0$,
$$
\int_0^t f(X_s)\,d\langle X\rangle_s=\int_{\mathbb R} f(a) L_t^a(X)\,da.
$$
In the interior, the two local times are related by $\ell^a_t = s'(a) L^a_t / 2$, where $s$ is the scale function.
\end{remark}

\subsection{Reflecting CIR
process}\label{review-reflecting-cir-process}

The Cox--Ingersoll--Ross (CIR) process  \citep{CoxIngersollRoss1985} is defined by
\begin{equation}\label{eq-CIR}{
      du_t=\left[\frac{\delta-1}{\beta \,u_t} - \lambda\, u_t \right]\,dt + \sqrt{\frac{2}{\beta}}\,dW_t ,\qquad u_0=x>0,\\
}\end{equation} with parameters \(\lambda,\beta>0\) and
\(\delta\in (1,2)\) and a standard  Brownian motion \(W_t \).
By \cref{eq-scale-speed-sde} with $b(x) = (\delta-1)/(\beta x) - \lambda x$ and $\sigma^2 = 2/\beta$, the scale and speed densities for \cref{eq-CIR} are
\begin{equation}\label{eq:scale-speed-cir}
  s'(x) = x^{1-\delta}\exp(\lambda\beta x^2/2), \qquad m'(x) = \beta x^{\delta-1}\exp(-\lambda\beta x^2/2), \qquad x \in (0,\infty).
\end{equation}
Near the origin, $s'(x) \sim x^{1-\delta}$ is integrable if and only if $\delta < 2$, in which case $s(0+) > -\infty$ and $\int_0^\varepsilon (s(x)-s(0+))\,m(dx) < \infty$. Feller's test \citep[Lemma~6.1]{karlintaylor1981} then makes the boundary at $0$ accessible, which is essential for the sticky case that we focus on in this paper.  For $\delta \ge 2$, the boundary is inaccessible.  The lower endpoint $\delta > 1$ ensures $m'(x) \to 0$ as $x \to 0$, so the interior part of the invariant measure vanishes near the boundary.  We restrict to $\delta \in (1,2)$ throughout.

We show this equation has a
well-defined solution \(u_t\ge 0\).

\begin{theorem}[well-posedness of reflecting CIR process]\label{thm-wp}
There exists a unique strong solution \(u_t\) of \cref{eq-CIR}
for \(\delta\in(1,2)\). The solution \(u_t\) is a non-negative regular diffusion
on $[0,\infty)$, with the scale and speed densities
of \eqref{eq:scale-speed-cir}; the boundary \(u=0\) is accessible and reflected
instantaneously, equivalently the speed measure has no atom there, $m(\{0\}) = 0$.
\end{theorem}
\begin{proof}
By It\^o's formula, $Y_t = u_t^2$ satisfies
\begin{equation}\label{eq-CIR2}
dY_t = \frac{2\delta}{\beta}\,dt - 2\lambda Y_t\,dt + \sqrt{\frac{8 Y_t}{\beta}}\,dW_t.
\end{equation}
That is, $Y_t$ is a mean-reverting squared Bessel process of dimension $\delta$.  By the classical theory of squared Bessel processes with $\delta\in(1,2)$ \citep[Section~3]{GoingJaeschkeYor2003}, $Y_t$ admits a unique non-negative strong solution and is a regular diffusion on $[0,\infty)$ whose boundary at $0$ is reached and reflected instantaneously: $\int_0^t \mathbf{1}_{\{Y_s = 0\}}\,ds = 0$ a.s.  Setting $u_t = \sqrt{Y_t}$, an increasing homeomorphism of $[0,\infty)$, yields a unique non-negative strong solution of \cref{eq-CIR} that inherits the regularity of $Y_t$, with $\int_0^t \mathbf{1}_{\{u_s=0\}}\,ds = 0$, so the occupation-time formula \cref{otf} forces $m(\{0\}) = 0$.
\end{proof}

\section{Sticky CIR process (no potential)}\label{sticky-cir-process}

We turn to the CIR process with a sticky boundary condition (\cref{eq-CIR-potential} with $G=0$), 
\begin{equation}\label{eq-CIR4}{
\begin{aligned}
      du_t&=\left[\frac{\delta-1}{\beta \,u_t} - \lambda\, u_t \right]\mathbf{1}_{\{u_t>0\}}\,dt + \sqrt{\frac{2}{\beta}}\,dW_t ,\\
\mathbf{1}_{\{u_t=0\}}\,dt&=\frac{1}{\mu}d\ell_t^0(u),
\end{aligned}
}\end{equation} 
subject to the initial condition $u_0=x>0$. The second line is the sojourn condition where $\ell_t^0$ is the diffusion
local time at $0$ and $\mu>0$ controls the degree of stickiness (see \cref{diffusion-local-time}).

\begin{theorem}[well-posedness of the sticky CIR
process]\label{thm-sticky-wp}

The sticky CIR process defined by \cref{eq-CIR4} has a unique
weak solution \(u_t\) for \(\delta\in(1,2)\) and  \(\mu> 0\). The solution \(u_t\) is a regular
diffusion process with a Feller transition semigroup
$P_t\varphi(x)\coloneqq\mathbb{E}[\varphi(u_t)\mid u_0=x]$: it is a strongly continuous
contraction semigroup on $C_0([0,\infty))$, the continuous functions on $[0,\infty)$
vanishing at infinity, and $x\mapsto P_t\varphi(x)$ is continuous for every bounded
continuous $\varphi$.

\end{theorem}

\begin{proof}
Sticky Feller (square-root) diffusions are constructed and their boundary behaviour classified by \citet{PeskirRoodman2023}; we give a self-contained construction tailored to the CIR drift. The It\^o--McKean time change of \citet[Section~2]{Peskir2022} uses the base process only through its being a regular diffusion with an instantaneously reflecting, regular boundary at $0$ carrying a finite boundary local time---never through the specific Bessel drift. These properties hold for the reflecting CIR process $\tilde u_t$, including its mean-reversion $-\lambda u$ and singular $1/u$ drift, by \cref{thm-wp}; the construction therefore transfers, with his reflecting Bessel process replaced by $\tilde u_t$. A weak solution is obtained as $u_t = \tilde u_{T_t}$, where $T_t = A_t^{-1}$ and $A_t = t + \mu^{-1}\ell_t^0(\tilde u)$. Conversely, any weak solution recovers $\tilde u$ via the inverse time change $\tilde u_t = u_{A_t}$, with $A_t = \inf\{s : T_s > t\}$ and $T_t = \int_0^t \mathbf{1}_{\{u_s>0\}}\,ds$; since $\ell^0(\tilde u)$ is locally finite, the time change is defined for all $t \ge 0$, and strong uniqueness of $\tilde u$ (\cref{thm-wp}) gives weak uniqueness of $u$. The mean-reversion drift $-\lambda \tilde u_t$ enters only through the drift of $\tilde u$ and passes through the time change unchanged. Regularity of $u_t$ in the sense of \cref{def-regular-diffusion} follows from the general theory of time-changed strong Markov processes \citep{Volkonskii1958}: $\tilde u$ is regular by \cref{thm-wp}, and time change by a continuous strictly-increasing additive functional preserves regularity. The boundary at $0$ being regular, the classification of regular one-dimensional diffusions \citep[\S4]{ItoMcKean1965} makes $P_t$ a strongly continuous contraction semigroup on $C_0([0,\infty))$, with $x\mapsto P_t\varphi(x)$ continuous for every bounded continuous $\varphi$.
\end{proof}

\subsection{Invariant measure and
ergodicity}\label{invariant-measure-and-ergodicity}

The next theorem gives ergodicity of the sticky CIR process, and presents the invariant measure as the sum of an atom at the boundary and a continuous part on the interior.  





\begin{theorem}[ergodicity of the sticky CIR
process]\label{thm-ergodicity}
Let \(u_t\) be the sticky CIR process defined by \cref{eq-CIR4}
with parameters \(\lambda, \beta, \mu > 0\) and \(\delta \in (1,2)\). 
\begin{enumerate}
\def\labelenumi{\arabic{enumi}.}
\item There exists a unique invariant
  probability measure \(\pi_0\) on \([0,\infty)\) given by \[
  \pi_0(dx) = \frac{1}{Z} \left[ \frac{1}{\mu} \delta_0(dx) + \beta x^{\delta-1} e^{-\lambda\beta x^2/2} \, dx \right]
  \] where $Z = 1/{\mu} + \int_0^\infty \beta x^{\delta-1}
e^{-\lambda\beta x^2/2}\,dx<\infty$.
\item For any initial point \(x \in [0,\infty)\) and
  any bounded measurable function \(\phi\colon [0,\infty)\to \mathbb{R}\), \[
  \lim_{t \to \infty} P_t\phi(x) = \int_{[0,\infty)} \phi(y) \, \pi_0(dy).
  \]
\end{enumerate}
\end{theorem}
\begin{proof}
  The sticky SDE \eqref{eq-CIR4} shares the interior dynamics of the reflecting CIR process, so the interior scale and speed densities are given by \cref{eq:scale-speed-cir}.
  The sticky boundary condition adds a point mass at $0$ with mass
  $m(\{0\}) = 1/\mu$ as discussed in \cref{diffusion-local-time},
  giving the full speed measure $m(dx) = m'(x)\,dx + \mu^{-1}\delta_0(dx)$.
   Since $m$ is the speed measure of a regular diffusion (\cref{thm-sticky-wp}), it is an invariant
  $\sigma$-finite measure \citep[Chapter~V]{RevuzYor1999}.  As $\delta>1$,
  \[
    m([0,\infty)) = \frac{1}{\mu}
    + \int_0^\infty \beta x^{\delta-1} e^{-\lambda\beta x^2/2}\,dx < \infty,
  \]
  so $\pi_0 = m/m([0,\infty))$ is a well-defined invariant probability measure.

  The process is conservative, irreducible with respect to $m$, and admits a
  finite invariant measure. Hence, \citet[Theorem~10.0.1]{MeynTweedie1993} applies, and we conclude that $\pi_0$ is the unique invariant measure and $P_t\phi(x)\to\pi_0(\phi)$
  for all bounded measurable $\phi$.
  \end{proof}


\subsection{Transition density and boundary
condition}\label{transition-density-and-boundary-condition}

We identify the Fokker--Planck equation for the transition density, including the boundary condition at $u=0$.

\begin{theorem}[transition density]\label{thm-pdf0}
  On $(0,\infty)$, the Lebesgue density $f$ of the CIR process defined by \cref{eq-CIR4} solves
  \[
  \partial_t f = -\partial_u\!\left(\left(\tfrac{\delta-1}{\beta u}
    -\lambda u\right)f\right) + \tfrac{1}{\beta}\partial_{uu}f.
  \]
  If $p$ denotes the density with respect to the speed measure $m$,
  then at $u=0$ the sticky (Wentzell) boundary condition holds
  \begin{equation}
    \label{eq-wentzell}
    \partial_s p(t,0^+) = \tfrac{1}{\mu}\,\partial_t p(t,0),
  \end{equation}
  where $\partial_s = (1/s')\partial_u$ is the scale derivative.
  \end{theorem}
  
  \begin{proof}
  The Lebesgue-density $f(t,u)$ of the process on $(0,\infty)$ satisfies
  the Fokker--Planck equation $\partial_t f = \mathcal{L}^* f$, where $\mathcal L^*$ is the
  $L^2(du)$-adjoint of the generator. Substituting the coefficients
  $b(u) = (\delta-1)/(\beta u) - \lambda u$ and $\sigma^2 = 2/\beta$ into $\mathcal{L}^* f = -\partial_u(bf) + \frac{1}{2}\partial_{uu}(\sigma^2 f)$
  gives the stated equation.
  
  Let $p_0(t) = \mathbb{P}(u_t = 0)$ denote the mass at the sticky
  point. Conservation of probability gives
  \[
    p_0(t) + \int_0^\infty f(t,u)\,du = 1.
  \]
  Differentiating in $t$ and substituting $\partial_t f = \mathcal{L}^*f$:
  \[
    \partial_t p_0(t)
    = -\int_0^\infty \mathcal{L}^* f\,du
    = \lim_{u\to 0^+} J(t,u),
  \]
  where $J(t,u)$ is the probability flux, and we used that the flux decays at infinity.  
  From~\cref{eq-adjoint-generator}, the Lebesgue-adjoint generator is
  $\mathcal{L}^* f = \frac{d}{du}\!\bigl(\frac{1}{s'}\frac{d}{du}\frac{f}{m'}\bigr)$,
  so the flux
  \begin{equation}\label{eq-flux}
    J(t,u) = \frac{1}{s'(u)}\,\frac{d}{du}\frac{f(t,u)}{m'(u)}
             = \frac{1}{s'(u)}\,\partial_u p(t,u)
             = \partial_s p(t,u),
  \end{equation}
  where $p = f/m'$ is the speed measure density and we used $\partial_s = (1/s')\partial_u$.
    The sticky boundary contributes a point mass $m(\{0\}) = 1/\mu$, so writing the boundary mass in terms of the speed measure density gives
  $p_0(t) = \mathbb{P}(u_t=0) = {p(t,0)}/{\mu}$.
  Substituting into the flux-balance equation $\partial_t p_0(t) =
  J(t,0^+)$, we find
  \(
    \tfrac{1}{\mu}\,\partial_t p(t,0) = \partial_s p(t,0^+),
  \)
  which rearranges to the Wentzell condition~\cref{eq-wentzell}.
  \end{proof}

\section{Sampling the invariant measure}\label{sec-sampling}
Our sampling algorithms are based on an explicit form for the transition distribution for exponentially distributed times. This requires use of the resolvent.
\begin{definition}[resolvent]\label[definition]{def:resolvent}
Let $X$ be a regular diffusion with generator $\mathcal L$ and speed measure $m$.  For $\alpha>0$, the \emph{resolvent} is the operator $R_\alpha\coloneqq\alpha(\alpha-\mathcal L)^{-1}$, which averages a function over an independent $\operatorname{Exp}(\alpha)$ time:
\begin{equation}\label{eq:resolvent-op}
  R_\alpha f(x)\;=\;\mathbb E\bigl[f(X_{T_\alpha})\mid X_0=x\bigr]\;=\;\alpha\!\int_0^\infty G_\alpha(x,y)\,f(y)\,m(dy),
\end{equation}
with \emph{resolvent kernel} $G_\alpha$.
\end{definition}

We apply this to the reflecting CIR $\tilde u$ of \cref{thm-wp}.  Analytically, its kernel $G_\alpha$ is  the \emph{Green's function} of the resolvent equation: as a function of $x$ for fixed $y$, $f(x) = G_\alpha(x,y)$ satisfies
\begin{equation}\label{eq:resolvent-ode}
  (\alpha - \mathcal L)\,f = \delta_y,
\end{equation}
where $\mathcal L$ is the generator of $\tilde u$. This is interpreted in the speed-measure sense, so $\int \phi(x)(\alpha-\mathcal L)f(x)\,m(dx)=\phi(y)$ for any test function $\phi$.
Writing $p(t;x,y)$ for the transition density of $\tilde u$ with respect to $m$, the resolvent kernel $G_\alpha$ is the Laplace transform of the transition density,
\begin{equation}\label{eq:green-laplace}
  G_\alpha(x,y) = \int_0^\infty e^{-\alpha t}\,p(t;x,y)\,dt.
\end{equation}
Alternatively, if $T_\alpha\sim\operatorname{Exp}(\alpha)$ is independent of $\tilde u$, then
\begin{equation}\label{eq:green-resolvent}
  G_\alpha(x,y)\,m(dy) = \frac{1}{\alpha}\,\mathbb{P}(\tilde u_{T_\alpha}\in dy\mid \tilde u_0=x),
\end{equation}
so $\alpha\,G_\alpha(x,y)\,m(dy)$ is the transition distribution at the
random time $T_\alpha$.

\subsection{Calculating the resolvent kernel \texorpdfstring{$G_\alpha$}{G\_alpha}}\label{green-function}

 The distributions after exponential  times are computable and can be used to construct efficient sampling algorithms for the invariant measure of the sticky CIR process, as we will see in
 \cref{sampler-gpotential,sampler-gzero}.

Away from $x=y$, \cref{eq:resolvent-ode} reduces to the
homogeneous ODE $(\alpha-\mathcal L)f=0$, which is Kummer's equation
\begin{equation}\label{eq:kummer}
  f'' + \left(\frac{\delta-1}{x} - \lambda\beta x\right)f' - \alpha\beta\,f = 0.
\end{equation}
Setting $z = \lambda\beta x^2/2$, a direct calculation transforms \cref{eq:kummer} into $  z\,w'' + (b - z)\,w' - a\,w = 0$, for
$a=\alpha/(2\lambda)$, $b=\delta/2$. This is Kummer's confluent hypergeometric equation, whose linearly independent solutions are the confluent hypergeometric functions of the first kind, $M(a,b,z) = {}_1F_1(a;b;z)$ (regular at $z=0$, with $M(a,b,0)=1$), and of the second kind, $U(a,b,z)$ (decaying at $z=\infty$, with $U(a,b,z)\sim z^{-a}$); we use the notation of \citet[Chapter~13]{Olver2010}. Therefore, \cref{eq:kummer}  has solutions $M(a,b,z_x)$ and $U(a,b,z_x)$, where $z_x=\lambda\beta x^2/2$.

The Green's function is constructed by matching at $x=y$: the decaying solution
$U$ is used for the larger argument and the regular solution $M$ for the
smaller, 
\begin{equation}\label{eq:green-formula}
  G_\alpha(x,y) = \frac{f_0(x\wedge y)\,U(a,b,z_{x\vee y})}{|\mathcal{W}|},
\end{equation}
where $f_0(x) = M(a,b,z_x) + c\,U(a,b,z_x)$ with $c$ determined by the boundary condition and $|\mathcal{W}|$ is the Wronskian normalisation that ensures the jump condition at $x=y$ is satisfied. 
See \citep{RogersWilliams2000,BorodinSalminen2002} for the general theory of Green's functions for one-dimensional diffusions.

The next lemma collects three boundary quantities that recur throughout:  the scale derivatives of $f_0$ and $f_\infty$ at $0^+$, used in \cref{lem:sticky-resolvent} to impose the Wentzell sticky condition; the Wronskian normalisation $|\mathcal{W}|$ from \cref{eq:green-formula}; and the boundary value $G_\alpha(0,0) = U_0/|\mathcal{W}|$, which determines the leave probability $\Phi = 1/G_\alpha(0,0)$ of \cref{cor:leave-zero}. Each is given in closed form via the gamma function.
\begin{lemma}[boundary quantities of the reflecting resolvent]\label[lemma]{lemma-green-zero}
  \begin{enumerate}[label=\textup{(\roman*)}]
    \item  The scale derivatives of $f_0(x) = M(a,b,z_x)$ and
  $f_\infty(x) = U(a,b,z_x)$ at the boundary satisfy
  \[
    \frac{df_0}{ds}(0^+) = 0,
    \qquad
    \frac{df_\infty}{ds}(0^+) = -|\mathcal{W}|.
  \]
  \item
  \(
    G_\alpha(0,0) = {U_0}/{|\mathcal{W}|},
    \)
    where $U_0 = U(a,b,0) = \frac{\Gamma(1-b)}{\Gamma(1+a-b)}$ and
    \begin{equation}\label{wronskian}
      |\mathcal{W}| = \lambda\beta\,\frac{\Gamma(b)}{\Gamma(a)}\,
      \Bigl(\tfrac{\lambda\beta}{2}\Bigr)^{-b}.
    \end{equation}
  \end{enumerate}

\end{lemma}

\begin{proof}
(i) The scale density is $s'(x) = x^{1-\delta}e^{\lambda\beta x^2/2}$.
  Since $M(a,b,z) = 1 + \frac{a}{b}z + O(z^2)$, we have
  $f_0'(x) \sim \frac{a}{b}\lambda\beta x$ as $x\downarrow 0$, so
  $f_0'(x)/s'(x) \sim \frac{a\lambda\beta}{b}x^\delta \to 0$,
  confirming $(df_0/ds)(0^+) = 0$.

  The Wronskian satisfies the identity $W_z[M,U] \coloneqq MU' - M'U = -\frac{\Gamma(b)}{\Gamma(a)}z^{-b}e^z$
  \cite[\S13.2.34]{Olver2010}. Since $dz_x/dx = \lambda\beta x$, this gives
  \[
    f_0(x) f_\infty'(x) - f_\infty(x) f_0'(x)
    = W_z[M,U]\big|_{z_x} \,  \lambda\beta x
    = -\frac{\Gamma(b)}{\Gamma(a)}
      \Bigl(\frac{\lambda\beta x^2}{2}\Bigr)^{\!-b}
      e^{\lambda\beta x^2/2}\,\lambda\beta x.
  \]
  Dividing by $s'(x)$, the exponentials cancel and the net power of $x$
  vanishes (since $x^{-2b}\,x\,x^{\delta-1} = 1$), giving the constant
  \[
    \frac{f_0 f_\infty' - f_\infty f_0'}{s'(x)}
    = -\frac{\Gamma(b)}{\Gamma(a)}\,\lambda\beta
      \left(\frac{\lambda\beta}{2}\right)^{-b}
    \eqqcolon -|\mathcal{W}|.
  \]
  By Abel's theorem applied to the ODE~\eqref{eq:kummer} with coefficient
  $p(x) = (\delta-1)/x - \lambda\beta x$, the Wronskian satisfies
  $W[f_0,f_\infty](x) = C\exp(-\int p\,dx)$.  Since
  $\exp(-\int p\,dx) = x^{1-\delta}e^{\lambda\beta x^2/2} = s'(x)$
  by~\cref{eq:scale-speed-cir}, the ratio $W[f_0,f_\infty](x)/s'(x)$ is
  constant in $x$, and the calculation above identifies this constant
  as $-|\mathcal{W}|$.
  Taking $x\to 0^+$ and using $f_0(0)=1$, $(df_0/ds)(0^+)=0$, the Wronskian identity gives
  \[
    \frac{df_\infty}{ds}(0^+) = -|\mathcal{W}|.
  \]

  (ii) Note that $M(a,b,0)=1$  and $U(a,b,0)=\Gamma(1-b)/\Gamma(1+a-b)\eqqcolon U_0$ from  \cite[(13.2.2) and (13.2.18)]{Olver2010} (under the condition $b\in(0,1)$, equivalently $\delta\in(0,2)$). From  $G_\alpha(x,y) = f_0(x\wedge y)\,f_\infty(x\vee y)/|\mathcal{W}|$, we have  that
  $G_\alpha(0,0) = U_0/|\mathcal{W}|$. \qedhere
\end{proof}

\subsection{Transition distributions for the sticky CIR process}
\label{transition-sticky-cir}

We derive the transition distribution $\mathbb{P}(u_{T_\alpha}\in\cdot\mid u_0=x)$
for $x\geq 0$, where $T_\alpha\sim\mathrm{Exp}(\alpha)$ is
independent of $u$. The key step is to
impose the Wentzell sticky boundary condition on the resolvent.


\begin{lemma}[resolvent boundary condition]\label[lemma]{lem:resolvent-bc}
The resolvent kernel $G^{\rm sticky}_\alpha(x,y)$, viewed as a function
of its first argument at fixed $y \in (0,\infty)$, solves
$(\alpha-\mathcal{L})G^{\rm sticky}_\alpha(\cdot,y)=\delta_y$ on $(0,\infty)$
in the speed-measure sense, with boundary condition
\begin{equation}\label{eq:wentzell}
  \frac{dG^{\rm sticky}_\alpha(\cdot,y)}{ds}(0^+)
  = \frac{\alpha}{\mu}\,G^{\rm sticky}_\alpha(0,y).
\end{equation}
By the symmetry $G^{\rm sticky}_\alpha(x,y) = G^{\rm sticky}_\alpha(y,x)$, the same condition holds in the second argument.
\end{lemma}

\begin{proof}
  By standard properties of Green's functions,  the function $x\mapsto G^{\rm sticky}_\alpha(x,y)$ solves
the  resolvent ODE $(\alpha-\mathcal L)f=\delta_y$. Only the boundary behaviour at $0$ is
modified by stickiness, and that is what we verify here.

Fix $x\in(0,\infty)$. Recall that $G_\alpha^{\rm sticky}(x,y)$ is given by the Laplace transform $\int_0^\infty e^{-\alpha t} p(t;x,y)\,dt$ by~\cref{eq:green-laplace} (applied to the sticky process, with $m$ now carrying the atom $\mu^{-1}\delta_0$).  We apply the Laplace transform 
to both sides of the time-domain Wentzell condition~\cref{eq-wentzell},
\(
  \partial_s p(t;x,0^+) = \tfrac{1}{\mu}\,\partial_t p(t;x,0).
\) 
On the right, integration by parts in $t$ gives
\begin{align*}
  \int_0^\infty e^{-\alpha t}\,\partial_t p(t;x,0)\,dt
  &= \bigl[e^{-\alpha t}p(t;x,0)\bigr]_0^\infty
  + \alpha\int_0^\infty e^{-\alpha t}\,p(t;x,0)\,dt\\
  &= \alpha\,G^{\rm sticky}_\alpha(x,0),
\end{align*}
where the boundary term at $\infty$ vanishes and the term at $0$ equals
$-p(0^+;x,0)=0$ (since $x>0$ and the initial law has no mass at $0$). On the left, commuting $\partial_s$ with the $t$-integral in the definition of $G^{\rm sticky}_\alpha$ (justified as $p$ and its scale derivative are continuous up to the boundary with at most polynomial growth in $t$; \cref{lem:sticky-resolvent}) gives
\[
  \int_0^\infty e^{-\alpha t}\,\partial_s p(t;x,0^+)\,dt
   =  \frac{d\,G^{\rm sticky}_\alpha(x,\cdot)}{ds}(0^+).
\]
Equating the two transforms gives the second-argument form of~\cref{eq:wentzell}:
\[
  \frac{d\,G^{\rm sticky}_\alpha(x,\cdot)}{ds}(0^+)
   =  \frac{\alpha}{\mu}\,G^{\rm sticky}_\alpha(x,0).
\]
Symmetry of the resolvent kernel with respect to the speed measure
 transfers this to the first argument
and gives~\cref{eq:wentzell}. 
\end{proof}
\begin{lemma}[sticky resolvent kernel]\label[lemma]{lem:sticky-resolvent}
The fundamental solutions of $(\alpha-\mathcal{L})f=0$ satisfying the
boundary condition~\eqref{eq:wentzell} are $f_\infty(x) = U(a,b,z_x)$
and
\begin{equation}\label{eq:f0sticky}
  f_0^{\rm sticky}(x) = M(a,b,z_x) + c_\mu\,U(a,b,z_x),
\end{equation}
where
\begin{equation}\label{eq:cmu}
  c_\mu = \frac{-\alpha}{\mu\,|\mathcal{W}| + \alpha U_0}.
\end{equation}
The sticky resolvent kernel is
\[
  G^{\rm sticky}_\alpha(x,y) =
  \frac{f_0^{\rm sticky}(x\wedge y)\,f_\infty(x\vee y)}{|\mathcal{W}|},
\]
with the same normalisation constant $|\mathcal{W}|$ as the reflecting
resolvent kernel.
\end{lemma}
\begin{proof}
The solution $f_\infty = U(a,b,z_x)$ decays at $\infty$ and is unchanged from the reflecting case. The general solution of 
$(\alpha-\mathcal{L})f=0$ on $(0,\infty)$ regular at $0$ is a linear
combination of $M$ and $U$, so we write $f_0^{\rm sticky}$ as
in \cref{eq:f0sticky}. From \cref{lemma-green-zero},
$(dM/ds)(0^+)=0$ and $(dU/ds)(0^+)=-|\mathcal{W}|$, so substituting
into \cref{eq:wentzell} gives
\[ 
  -c_\mu\,|\mathcal{W}| = \frac{\alpha}{\mu}(1 + c_\mu U_0),
\]
and solving yields \cref{eq:cmu}. The normalisation constant is unchanged
since adding $c_\mu f_\infty$ to $f_0$ contributes
$c_\mu(f_\infty f_\infty' - f_\infty f_\infty')=0$ to the Wronskian.
The resolvent equation $(\alpha-\mathcal L)G^{\rm sticky}_\alpha(\cdot,y)=\delta_y$
then holds on $(0,\infty)$ as in the reflecting case \cref{eq:green-formula}, and~\cref{eq:wentzell}
holds by the choice of $c_\mu$.
\end{proof}
We recast the transition distribution in a form that can be exploited in the sampling algorithms: \cref{thm:transition-dist} for initial data $x>0$ and \cref{cor:leave-zero} when starting from the boundary $x=0$.
\begin{theorem}[transition distribution]\label{thm:transition-dist}
  For $x>0$, the transition distribution $P_x(u_{T_\alpha}\in\cdot)$
  admits the resolvent representation
  \begin{equation}\label{eq:trans}
    \mathbb{P}(u_{T_\alpha}\in A\mid u_0=x)
    = \frac{\alpha}{\mu}\,G^{\rm sticky}_\alpha(x,0)\,\mathbf{1}_{0\in A}
    + \alpha\int_{A\cap(0,\infty)} G^{\rm sticky}_\alpha(x,y)\,m'(y)\,dy,
  \end{equation}
  equivalent to the three-component mixture
  \begin{equation}\label{eq:trans-mixture}
    \mathbb P(u_{T_\alpha}\in A\mid u_0=x)
    = w_0(x)\,\mathbf{1}_{0\in A}
    + w_<(x)\,\nu_<(A)
    + w_>(x)\,\nu_>(A),
  \end{equation}
  where $\nu_<$ and $\nu_>$ are the probability measures with Lebesgue
  densities
  \[
    \frac{d\nu_<(y)}{dy}  \propto  f_0^{\rm sticky}(y)\,m'(y)\,\mathbf{1}_{(0,x)}(y),
    \qquad
    \frac{d\nu_>(y)}{dy}  \propto  U(a,b,z_y)\,m'(y)\,\mathbf{1}_{(x,\infty)}(y),
  \]
  and the weights are
  \begin{align}
    w_0(x) &= -c_\mu\,U(a,b,z_x), \label{eq:w0}\\
    w_<(x) &= \frac{\alpha\,U(a,b,z_x)}{|\mathcal{W}|}
               \int_0^x f_0^{\rm sticky}(y)\,m'(y)\,dy, \label{eq:wlt}\\
    w_>(x) &= \frac{\alpha\,f_0^{\rm sticky}(x)}{|\mathcal{W}|}
               \int_x^\infty U(a,b,z_y)\,m'(y)\,dy, \label{eq:wgt}
  \end{align}
  satisfying $w_0(x)+w_<(x)+w_>(x)=1$.
\end{theorem}

\begin{proof}
  The representation~\cref{eq:trans} follows from~\cref{eq:green-resolvent}
  and the decomposition of the speed measure,
  $m(dy)=\mu^{-1}\delta_0(dy)+m'(y)\,dy$.

  For~\cref{eq:trans-mixture}, evaluate~\cref{eq:trans} on the atom first.
  By~\cref{lem:sticky-resolvent} with $y=0$,
  $G^{\rm sticky}_\alpha(x,0) = f_0^{\rm sticky}(0)\,U(a,b,z_x)/|\mathcal{W}|$.
  Evaluating $f_0^{\rm sticky}(0) = M(a,b,0)+c_\mu U(a,b,0) = 1+c_\mu U_0$
  and substituting $c_\mu=-\alpha/(\mu|\mathcal{W}|+\alpha U_0)$
  from~\cref{eq:cmu} gives
  \begin{equation}\label{eq:f0sticky-zero}   
    f_0^{\rm sticky}(0)
    = 1 - \frac{\alpha U_0}{\mu|\mathcal{W}|+\alpha U_0}
    = \frac{\mu|\mathcal{W}|}{\mu|\mathcal{W}|+\alpha U_0}.
  \end{equation}

  The atom weight $w_0(x) = (\alpha/\mu)\,G^{\rm sticky}_\alpha(x,0)$ then
  becomes
  \[
    w_0(x)
    = \frac{\alpha}{\mu}
      \,\frac{\mu|\mathcal{W}|}{\mu|\mathcal{W}|+\alpha U_0}
      \,\frac{U(a,b,z_x)}{|\mathcal{W}|}
    = -c_\mu\,U(a,b,z_x) > 0,
  \]
  where the inequality holds since $c_\mu<0$ and $U(a,b,z_x)>0$.

  To verify $w_0+w_<+w_>=1$, recall that
  \[
    G_\alpha(x,y) = \frac{M(a,b,z_{x\wedge y})\,U(a,b,z_{x\vee y})}{|\mathcal{W}|}
  \]
  is the reflecting resolvent kernel (the case $c_\mu=0$), and use the decomposition
  \[
    G^{\rm sticky}_\alpha(x,y)
    = G_\alpha(x,y)
    + \frac{c_\mu\,U(a,b,z_x)\,U(a,b,z_y)}{|\mathcal{W}|},
  \]
  which follows from $f^{\rm sticky}_0 = M + c_\mu U$. Since
  $w_<+w_> = \alpha\int_0^\infty G^{\rm sticky}_\alpha(x,y)\,m'(y)\,dy$,
  the decomposition gives
  \[
    w_<+w_>
    = \alpha\int_0^\infty G_\alpha(x,y)\,m'(y)\,dy
    + \frac{\alpha c_\mu\,U(a,b,z_x)}{|\mathcal{W}|}\int_0^\infty U(a,b,z_y)\,m'(y)\,dy.
  \]
  The first term equals $1$ since $\alpha G_\alpha(x,\cdot)\,m'(\cdot)$
  is a probability density on $(0,\infty)$ (the boundary mass vanishes for
  the reflecting process; see \cref{eq:green-resolvent}).
  The second integral equals $|\mathcal{W}|/\alpha$ by the same argument at
  $x=0$, where $G_\alpha(0,y)=U(a,b,z_y)/|\mathcal{W}|$; that is,
  \begin{equation}\label{eq:U-integral}
    \int_0^\infty U(a,b,z_y)\,m'(y)\,dy = |\mathcal{W}|/\alpha,
  \end{equation}
  so the second term equals $c_\mu\,U(a,b,z_x)=-w_0$.
  Hence, $w_0+w_<+w_>=1$.
\end{proof} 

  \begin{corollary}[leave probability and exit distribution]
  \label[corollary]{cor:leave-zero}
  Let $T_\alpha\sim\operatorname{Exp}(\alpha)$ be independent of the sticky process $u$, with $u_0=0$.
  The probability that $u_{T_\alpha}\neq 0$  is
  \[
    p_{\mathrm{leave}}
    = \frac{\mu}{\mu+\Phi},
    \qquad
    \Phi = \frac{1}{G_\alpha(0,0)},
  \]
  where $G_\alpha$ is the reflecting resolvent kernel of \cref{green-function}.
  Conditional on $u_{T_\alpha}\neq 0$, the exit position is drawn from
  the density
  \begin{equation}\label{eq:exit-density}
    \nu_\alpha(dy) = Z_\nu^{-1}\,U(a,b,z_y)\,m'(y)\,dy,
    \qquad Z_\nu = |\mathcal{W}|/\alpha.
  \end{equation}
  \end{corollary}
  
  \begin{proof}
       By the time-change construction $u_t = \tilde{u}_{T_t}$, the sticky process
remains at $0$ until the accumulated local time $\ell^0_{T_\alpha}(\tilde{u})$
exceeds an independent budget $\theta \sim \operatorname{Exp}(\mu)$.
To determine the distribution of $\ell^0_{T_\alpha}(\tilde{u})$, let
$\tau_\ell = \inf\{t : \ell^0_t(\tilde{u}) > \ell\}$ be the inverse local time
of the reflecting process $\tilde{u}$ at $0$; this is a subordinator with
Laplace exponent $\kappa(\alpha)$, meaning
$\mathbb{E}[e^{-\alpha \tau_\ell}] = e^{-\ell\,\kappa(\alpha)}$.
By the standard theory of one-dimensional diffusions
\citep[Section~6.2]{ItoMcKean1965}, the Laplace exponent equals
$\kappa(\alpha) = 1/G_\alpha(0,0) \eqqcolon \Phi$, where $G_\alpha(0,0) = U_0/|\mathcal W|$
is computed in \cref{lemma-green-zero}.
Since $T_\alpha \sim \operatorname{Exp}(\alpha)$ is independent of $\tilde{u}$,
\[
  \mathbb{P}\pp{\ell^0_{T_\alpha}(\tilde{u}) > \ell}
  = \mathbb{P}\pp{\tau_\ell < T_\alpha}
  = \mathbb{E}[e^{-\alpha \tau_\ell}]
  = e^{-\ell\,\Phi},
\]
so $\ell^0_{T_\alpha}(\tilde{u}) \sim \operatorname{Exp}(\Phi)$, independently
of $\theta$. The leave probability then follows from
$\mathbb{P}\pp{E_1 > E_2} = \mu/(\mu+\Phi)$ for independent
exponentials $E_1 \sim \operatorname{Exp}(\Phi)$, $E_2 \sim
\operatorname{Exp}(\mu)$.

  Setting $x=0$ in \cref{thm:transition-dist}, the exit density is
  $\nu_\alpha(dy)\propto G^{\rm sticky}_\alpha(0,y)\,m'(y)\,dy \propto
  U(a,b,z_y)\,m'(y)\,dy$, since $f_0^{\rm sticky}(0)$ is a constant
  independent of $y$. The normalising constant
  $Z_\nu = |\mathcal{W}|/\alpha$
  follows from~\cref{eq:U-integral}.
  \end{proof}

\subsection{Sticky CIR sampler (no potential)}\label{sampler-gzero}

We construct an exact sampler for the invariant measure $\pi_0$ of the sticky CIR process (with zero potential $G=0$).  The chain samples the sticky CIR process exactly at independent exponential times: given $u_k = x$, draw $T_\alpha\sim\operatorname{Exp}(\alpha)$ and set $u_{k+1}\sim\mathbb{P}(u_{T_\alpha}\in\cdot\mid u_0=x)$.  Since the continuous-time process is ergodic (\cref{thm-ergodicity}), the chain returns $\pi_0$ as its invariant distribution.  The closed-form three-component mixture~\cref{eq:trans-mixture} of \cref{thm:transition-dist} reduces each draw to sampling from three explicit components: an atom at the boundary, a density on $(0,x)$, and a density on $(x,\infty)$, each sampled in turn below.

Two integrals are needed for each starting point $x$:
\[
  I_>(x) = \int_x^\infty U(a,b,z_y)\,m'(y)\,dy, \qquad
  I_<(x) = \int_0^x f_0^{\rm sticky}(y)\,m'(y)\,dy.
\]
Since $f_0^{\rm sticky}(y) = M(a,b,z_y) + c_\mu\,U(a,b,z_y)$, the second
integral splits as $I_<(x) = I_<^M(x) + c_\mu\,I_<^U(x)$, where
\[
  I_<^M(x) = \int_0^x M(a,b,z_y)\,m'(y)\,dy, \qquad
  I_<^U(x) = \int_0^x U(a,b,z_y)\,m'(y)\,dy.
\]
In \cref{alg:sticky-cir}, both $I_<^U$ and $I_>$ are precomputed on a truncated grid $[0,y_{\max}]$ by the trapezium rule; the upper limit $y_{\max}$ is chosen so that the tail $\int_{y_{\max}}^\infty U(a,b,z_y)\,m'(y)\,dy$ is negligible, which is justified by the rapid decay of $U(a,b,z_y)$ as $y\to\infty$.
The remaining integral $I_<^M$ is recovered without further quadrature using
the reflecting resolvent identity (i.e., \cref{thm:transition-dist} with $\mu\to \infty$ and $c_\mu=0$):
\begin{equation}\label{eq:resolvent-identity}
  \frac{\alpha}{|\mathcal{W}|}\bigl[M(a,b,z_x)\,I_>(x) +
  U(a,b,z_x)\,I_<^M(x)\bigr] = 1,
\end{equation}
which gives $I_<^M(x) = \bigl(|\mathcal{W}|/\alpha -
M(a,b,z_x)\,I_>(x)\bigr)/U(a,b,z_x)$.  

From $u_k = x > 0$, the next state is drawn from the three-component mixture~\cref{eq:trans-mixture} with weights~\cref{eq:w0,eq:wgt}:
\begin{description}
  \item[Atom at $0$] (weight $w_0(x)$): set $u_{k+1}=0$.

  \item[Above $x$] (weight $w_>(x)$): density $U(a,b,z_y)\,m'(y)$ on $(x,\infty)$, sampled by CDF inversion against the precomputed $I_>$.

  \item[Below $x$] (weight $w_<(x)$): density $f_0^{\rm sticky}(y)\,m'(y)$ on $(0,x)$, sampled by CDF inversion against $I_<$ (recovered from $I_>$ and $I_<^U$ via \cref{eq:resolvent-identity}).
\end{description}

From $u_k = 0$, the leave probability $p_{\rm leave}$ and exit
distribution $\nu_\alpha$ are given by \cref{cor:leave-zero}: with
probability $p_{\rm leave}$ the process leaves $0$ and the landing
position is drawn from $\nu_\alpha$ (see \cref{rem:exit-sampler} below), and with probability
$1-p_{\rm leave}$ it remains at $0$.

\begin{remark}[sampling the exit density]\label[remark]{rem:exit-sampler}
To draw $y\sim\nu_\alpha$, we use rejection sampling. First, substitute $w = z_y = \lambda\beta y^2/2$, so the speed measure density becomes $m'(y)\,dy \propto w^{b-1}e^{-w}\,dw$ (using $y^{\delta-2}\propto w^{b-1}$ with $b=\delta/2$), which is proportional to the $\mathrm{Gamma}(b,1)$ density. Therefore, we propose $y=\sqrt{2w/(\lambda\beta)}$ for $w\sim\mathrm{Gamma}(b,1)$. The remaining factor $U(a,b,w)$ in \cref{eq:exit-density} is bounded by $U_0=U(a,b,0)$ since $U(a,b,\cdot)$ is strictly decreasing on $[0,\infty)$. Hence, we accept with probability $U(a,b,w)/U_0$.
\end{remark}
\begin{remark}
  The functions $M(a,b,z)$ and $U(a,b,z)$ are the confluent hypergeometric
  functions of the first and second kind, available in Python as
  \texttt{scipy.special.hyp1f1} and \texttt{scipy.special.hyperu}
  respectively.  The cost per step is dominated by the offline precomputation of $I_<^U$ and $I_>$ on a grid of starting points $x$.  The precomputation cost is $O(N)$ evaluations
  of $M$ and $U$ on the grid of $N$ points, done once per
  $(\alpha,\lambda,\beta,\delta)$.  The per-step cost is $O(\log N)$ for the
  binary-search CDF inversion, plus two evaluations of $M$ and $U$ at the
  current state $x$.  The sampler reuses the same grid for all values of
  $\mu$: only the scalar $c_\mu$ and $p_{\rm leave}$ change.
\end{remark}

\begin{algorithm}
  \caption{Exact sampler for the sticky CIR (no potential, $G=0$)}
  \label[algorithm]{alg:sticky-cir}
  \begin{algorithmic}[1]
  \Statex \textsc{Precomputation}
  \State Set $a=\alpha/(2\lambda)$, $b=\delta/2$.
         Compute $U_0$ and $|\mathcal{W}|$ via \cref{lemma-green-zero}.
  \State On a grid $y_j$ in the truncated domain $[0,y_{\max}]$, compute quadrature approximations to the integrals
         $I_<^U(y_j) = \int_0^{y_j} U(a,b,z_y)\,m'(y)\,dy$ and $I_>(y_j) = \int_{y_j}^{y_{\rm max}} U(a,b,z_y)\,m'(y)\,dy$.
  \State Set $c_\mu = -\alpha/(\alpha U_0 + \mu|\mathcal{W}|)$
         and $p_{\rm leave} = \mu|\mathcal{W}|/(\alpha U_0 + \mu|\mathcal{W}|)$.
         \Statex    
  \Statex \textsc{Case 1: } $u_k = x > 0$ (interior).
  \State Evaluate $U_x = U(a,b,z_x)$, $M_x = M(a,b,z_x)$,
         $f_{0,x} = M_x + c_\mu\,U_x$.
  \State Interpolate $I_<^U(x)$ and $I_>(x)$ from the grid; compute
         $I_<^M(x) = (|\mathcal{W}|/\alpha - M_x\,I_>(x))/U_x$
         and $I_<(x) = I_<^M(x) + c_\mu\,I_<^U(x)$.
  \State Compute weights $w_0 = -c_\mu\,U_x$,
         $w_< = \alpha U_x\,I_<(x)/|\mathcal{W}|$,
         $w_> = \alpha f_{0,x}\,I_>(x)/|\mathcal{W}|$.
  \State Draw $V\sim\mathrm{Uniform}(0,1)$.
  \If{$V < w_0$}
    \State $u_{k+1} = 0$.
  \ElsIf{$V < w_0 + w_<$}
    \State Draw $u_{k+1}$ from $f_0^{\rm sticky}(y)\,m'(y)$ on $(0,x)$
           by CDF inversion.
  \Else
    \State Draw $u_{k+1}$ from $U(a,b,z_y)\,m'(y)$ on $(x,\infty)$
           by CDF inversion.
  \EndIf
  \Statex
  \Statex \textsc{Case 2:} $u_k = 0$ (boundary).
  \State Draw $B\sim\mathrm{Bernoulli}(p_{\rm leave})$.
  \If{$B = 0$}
    \State $u_{k+1} = 0$.
  \Else
    \State Propose $w\sim\operatorname{Gamma}(b,1)$; set
           $y = \sqrt{2w/(\lambda\beta)}$; accept with probability
           $U(a,b,w)/U_0$.  If rejected, repeat.  Set $u_{k+1} = y$.
  \EndIf
  \end{algorithmic}
\end{algorithm}

\section{Extension to a potential function}
\label{sec:potential}
We are ready to treat the full SDE~\eqref{eq-CIR-potential}, which includes a non-trivial potential $G$. We assume $G \in C^1([0,\infty))$ with $G$ bounded below, and that $G'$ satisfies a linear growth bound: $|G'(u)| \leq C(1 + u)$ for some $C > 0$. Under this assumption, two results carry over from the $G = 0$ case. 
The first is well-posedness: the linear-growth assumption on $G'$ is exactly what is needed for a Girsanov change of measure against the sticky CIR process of \cref{thm-sticky-wp}.

\begin{theorem}[well-posedness]\label{thm-exist-potential}
Under the above assumptions on $G$, there exists a unique weak solution to \cref{eq-CIR-potential} with sticky boundary at $0$.  As for the $G=0$ process (\cref{thm-sticky-wp}), $u_t$ is a regular diffusion on $[0,\infty)$ with a Feller transition semigroup $P_t\varphi(x)\coloneqq\mathbb{E}[\varphi(u_t)\mid u_0=x]$.
\end{theorem}

\begin{proof}
Let $\tilde u_t$ be the sticky CIR process of \cref{thm-sticky-wp} on $(\Omega,\mathcal{F},(\mathcal{F}_t),\mathbb{P})$ with driving Brownian motion $W_t$, and define the exponential local martingale
\[
  M_t = \exp\!\left(-\sqrt{\tfrac{\beta}{2}} \int_0^t G'(\tilde u_s)\,dW_s - \tfrac{\beta}{4}\int_0^t |G'(\tilde u_s)|^2\,ds\right).
\]
The linear-growth bound on $G'(u)$ reduces Novikov's condition to a uniform exponential second-moment bound for $\tilde u$ on $[0,T]$. Taking $V(x)=e^{\gamma x^2}$ with $\gamma\in(0,\lambda\beta/2)$ as a Lyapunov function gives $\mathcal{L}V \le -\eta x^2 V + K$ (cf.~\cref{lem:lyapunov}), and Khasminskii's lemma \citep[Lemma~1.2]{MeynTweedie1993} supplies the required bound. Hence, $(M_t)_{t\in[0,T]}$ is a true martingale, and Girsanov's theorem makes $W_t^{\mathbb{Q}} = W_t + \sqrt{\beta/2}\int_0^t G'(\tilde u_s)\,ds$ a Brownian motion under $d\mathbb{Q} = M_T\,d\mathbb{P}|_{\mathcal{F}_T}$. Under $\mathbb{Q}$, $\tilde u$ satisfies \cref{eq-CIR-potential}; the sticky boundary condition, being a sample-path property, is preserved. Uniqueness in law follows by reversing the transformation.  The tilt adds only the locally bounded interior drift $-G'$ (as $G\in C^1$) and leaves the diffusion coefficient and the sticky boundary at $0$ unchanged, so $u_t$ inherits from $\tilde u$ the structure of a regular diffusion on $[0,\infty)$ with a Feller transition semigroup, by the classification of regular one-dimensional diffusions used in \cref{thm-sticky-wp}.
\end{proof}

The second result identifies the invariant measure: scale and speed-measure calculations from \cref{eq-scale-speed-sde} show that the potential reweights the $G=0$ invariant measure by the Gibbsian factor $e^{-\beta G}$, both on the interior and at the sticky atom.

\begin{theorem}[invariant measure]\label{thm-inv-potential}
  Suppose that $G \in C^1([0,\infty))$ and $e^{-\beta G}$ is integrable
  with respect to $\pi_0$. Then, \cref{eq-CIR-potential}
  has a unique invariant probability measure $\pi$, given by the Gibbsian
  reweighting of $\pi_0$:
  \[
    \frac{d\pi}{d\pi_0}(u) = \frac{1}{Z}\,e^{-\beta G(u)},
    \qquad
    Z = \int_{[0,\infty)} e^{-\beta G(u)}\,\pi_0(du).
  \]
\end{theorem}

\begin{proof}
  A regular diffusion with finite speed measure $m$ has unique invariant
  probability measure $m/m([0,\infty))$ (\cref{thm-ergodicity}), so it
  suffices to compute $m$ and show it has finite total mass proportional
  to $e^{-\beta G}\pi_0$.

  The drift $b(u) = (\delta-1)/(\beta u) - \lambda u - G'(u)$ equals $H'(u)$
  for $H(u) = \frac{\delta-1}{\beta}\log u - \lambda u^2/2 - G(u)$, so
  by~\cref{eq-scale-speed-sde} the scale derivative is
  \[
    s'(u) = e^{-\beta H(u)}
          = u^{1-\delta}\,e^{\lambda\beta u^2/2}\,e^{\beta G(u)}
          = s_0'(u)\,e^{\beta G(u)},
  \]
  where $s_0'$ is the scale derivative of the base ($G \equiv 0$) process.
  The speed measure density is then
  \[
    m'(u) = \frac{\beta}{s'(u)} = e^{-\beta G(u)}\,m_0'(u),
    \qquad
    m_0'(u) = \beta u^{\delta-1}\,e^{-\lambda\beta u^2/2}.
  \]
  The factor $e^{-\beta G(0)}$ in the sojourn condition
  of~\cref{eq-CIR-potential} is chosen so that the atom rescales by the
  same factor as the Lebesgue part: by the remark following
  \cref{diffusion-local-time},
  \[
    m(\{0\}) = \frac{e^{-\beta G(0)}}{\mu} = e^{-\beta G(0)}\,m_0(\{0\}).
  \]
  Hence, $m(du) = e^{-\beta G(u)}\,m_0(du)$ on all of $[0,\infty)$, and
  since $\pi_0 \propto m_0$, we have $m \propto e^{-\beta G}\pi_0$.
  The total mass $Z < \infty$ by hypothesis, so
  $\pi = e^{-\beta G}\pi_0 / Z$ is an invariant probability measure.

  Uniqueness follows as in~\cref{thm-ergodicity}: $G \in C^1([0,\infty))$
  ensures $e^{-\beta G}$ is bounded above and below on compacts, so $m$
  and $m_0$ have the same null sets, irreducibility transfers from the
  base process, and \citet[Theorem~10.0.1]{MeynTweedie1993} gives uniqueness.
\end{proof}

\subsection{MCMC Sampler with potential \texorpdfstring{$G$}{G}}\label{sampler-gpotential}

We develop a Metropolis--Hastings sampler for the invariant measure
$\pi$ of the sticky CIR process with potential $G$, using a proposal that
combines a single Euler step on the ODE $\dot x = -G'(x)$ with an exact draw
from the sticky potential-free  CIR resolvent kernel (as in \cref{sampler-gzero}).
The Metropolis--Hastings correction
accounts for the potential $G$ and the asymmetry introduced by starting the
CIR draw from the Euler-shifted point (defined below) rather than from
$x$ itself.

\paragraph{Proposal kernel} An Euler step from $x$ may overshoot the boundary; we route any such update back to $0$. Write $\mathcal R_h \coloneqq \{x>0 : x \le G'(x)h\}$ for the set of points from which the Euler step would exit the domain, and define the (clamped) Euler map
\[
  \phi_h(x) = \begin{cases}
    x - G'(x)\,h, & x > 0,\ x\notin\mathcal{R}_h,\\
    0,            & x = 0\ \text{or}\ x\in\mathcal{R}_h.
  \end{cases}
\]
We derive the proposal kernel $q(x\to v)$ for the move $x\to v$ by applying the transition distribution of \cref{thm:transition-dist} with starting point $s = \phi_h(x)$.
The proposal  has an interior density on $(0,\infty)$
and an atom at $0$, so
\begin{equation}\label{eq:proposal-q}
  q(x\to v)
  = \begin{cases}
    \alpha\,G_\alpha^{\rm sticky}(s,v)\,m'(v), & v > 0,\\[1pt]
    w_0(s),                                    & v = 0.
  \end{cases}
\end{equation}
At $s = 0$, \cref{eq:proposal-q} reduces to the boundary kernel of \cref{cor:leave-zero}: $\alpha\,G^{\rm sticky}_\alpha(0,v)\,m'(v)\,dv = p_{\rm leave}\,\nu_\alpha(dv)$ on $(0,\infty)$ (using $f_0^{\rm sticky}(0) = p_{\rm leave}$ from \cref{eq:f0sticky-zero} and $Z_\nu = |\mathcal{W}|/\alpha$), with atom $w_0(0) = 1 - p_{\rm leave}$.

In log-form, the speed-measure factors of $q$ and $\pi$ cancel pairwise in the Metropolis--Hastings ratio. Define
\begin{align*}
  \ell_+(s,v) &= \log\alpha + \log G^{\rm sticky}_\alpha(s,v), \qquad v > 0,\\[2pt]
  \ell_0(s)   &= \log w_0(s),
\end{align*}
so that at $s = 0$, $\ell_+(0,v) = \log p_{\rm leave} + \log U(a,b,z_v) - \log Z_\nu$ and $\ell_0(0) = \log(1 - p_{\rm leave})$. The same formulas apply to the reverse move $y\to x$ with $s = \phi_h(y)$. 

Proposals are accepted with probability
$\min\{1,\rho\}$, where
\[
  \rho (x\to y)= \frac{q(y\to x)\,\pi(y)}{q(x\to y)\,\pi(x)}
\]
is the Metropolis--Hastings ratio. We detail $\log \rho$ for the three cases $x\to y$ below.
\begin{description}
  \item[Interior $\to$ interior ($x\to y$, $x,y>0$).]
  \begin{equation}\label{eq:logrho-int}
    \log\rho(x\to y)
    = \beta\bigl(G(x)-G(y)\bigr)
    + \ell_+(\phi_h(y),\,x) - \ell_+(\phi_h(x),\,y).
  \end{equation}

  \item[Interior $\to$ boundary ($x\to 0$, $x>0$).]
  The reverse move $0\to x$ always uses the boundary kernel, giving
  \begin{gather}
    \label{eq:logrho-bdy}
    \begin{split}
      \log\rho(x\to 0)
    &= \beta\bigl(G(x)-G(0)\bigr) - \log\mu\\
&\quad    + \log p_{\rm leave} + \log U(a,b,z_x)
    - \log Z_\nu - \ell_0(\phi_h(x)),
  \end{split}
  \end{gather}
  where $\ell_0(\phi_h(x))$ equals $\log w_0(\phi_h(x))$ if $\phi_h(x) > 0$ (unrouted)
  and $\log(1-p_{\rm leave})$ if $\phi_h(x) = 0$ (routed).

  \item[Boundary $\to$ interior ($0\to y$, $y>0$).]
  By detailed balance, $\log\rho(0\to y) = -\log\rho(y\to 0)$:
  \begin{gather}\label{eq:logrho-from0}
   \begin{split} \log\rho(0\to y)
    &= \beta\bigl(G(0)-G(y)\bigr) + \log\mu\\
&\quad    - \log p_{\rm leave} + \log Z_\nu
    - \log U(a,b,z_y) + \ell_0(\phi_h(y)),
  \end{split}\end{gather}
  where $\ell_0(\phi_h(y))$ equals $\log w_0(\phi_h(y))$ if $\phi_h(y) > 0$ and
  $\log(1-p_{\rm leave})$ if $\phi_h(y) = 0$.
\end{description}

\begin{algorithm}
  \caption{MCMC sampler for the sticky CIR with potential $G$}
  \label[algorithm]{alg:mcmc}
\begin{algorithmic}[1]
\Statex \textsc{Precomputation.} Resolvent tables $I^U_<,I_>$, $c_\mu$ and $p_{\rm leave}$ as in \cref{alg:sticky-cir}; $Z_\nu=|\mathcal{W}|/\alpha$.
\Statex
\State \textbf{Propose.} With $\phi_h(u_k)=\max\bigl(u_k-G'(u_k)\,h,\,0\bigr)$, draw $y\sim P_{\phi_h(u_k)}^{G=0}$ from \cref{alg:sticky-cir} --- Case~1 if $\phi_h(u_k)>0$, else Case~2 (an over-shooting Euler step is routed to the boundary).
\State \textbf{Correct.} Set $u_{k+1}=y$ with probability $\min(1,e^{\log\rho})$, else $u_{k+1}=u_k$, where $\log\rho$ is given by \eqref{eq:logrho-int}, \eqref{eq:logrho-bdy}, or \eqref{eq:logrho-from0} according as $u_k\to y$ is interior$\to$interior, interior$\to$boundary, or boundary$\to$interior (with $\log\rho=0$ for $0\to0$).
\end{algorithmic}
\end{algorithm}
The MCMC algorithm is detailed in \cref{alg:mcmc} and the experiments are reported in \cref{sec:experiments}.

\section{Unadjusted Langevin sampler with mixed times}\label{sampler-ula}

The MCMC sampler of \cref{sampler-gpotential} combines a deterministic
Euler step with an exact draw from the sticky CIR resolvent kernel, followed by
a Metropolis--Hastings correction that enforces exact invariance with
respect to~$\pi$.  Removing the Metropolis--Hastings correction gives an \emph{unadjusted
Langevin algorithm} (ULA) that is cheaper per step but targets a biased
stationary distribution~$\pi_h$.

\paragraph{Splitting interpretation}
Decompose the generator of~\cref{eq-CIR-potential} as
$\mathcal{L} = \mathcal{L}_1 + \mathcal{L}_2$, where
\[
  \mathcal{L}_1 f = \frac{d}{dm_0}\frac{df}{ds_0}
\]
is the generator of the sticky potential-free CIR process (i.e., with $G \equiv 0$), and
\[
  \mathcal{L}_2 f = -G'(x)\,f'(x)
\]
is deterministic transport along $\dot x = -G'(x)$ on $(0,\infty)$.
Naively applying the Lie--Trotter split $e^{h\mathcal{L}} \approx e^{h\mathcal{L}_1} e^{h\mathcal{L}_2}$ is ambiguous here because the transport flow $\Phi_t$ of $\dot x = -G'(x)$ may leave $[0,\infty)$ in finite time: from states $x \in \mathcal{R}_h^{\mathrm{cts}} \coloneqq \{x > 0 : \Phi_t(x) = 0 \text{ for some } t \leq h\}$ the flow reaches $0$ and, where $G'(0)>0$, would continue to negative $x$, outside the domain of $G'$, so $e^{h\mathcal{L}_2}f(x)=f(\Phi_h(x))$ is undefined. We extend $e^{h\mathcal{L}_2}$ to
$[0,\infty)$ by depositing such mass at the boundary:
\begin{equation}\label{eq:transport-ext}
  \bigl(e^{h\mathcal{L}_2} f\bigr)(x) =
  \begin{cases}
    f(\Phi_h(x)), & x \notin \mathcal{R}_h^{\mathrm{cts}}, \\
    f(0), & x \in \{0\} \cup \mathcal{R}_h^{\mathrm{cts}}.
  \end{cases}
\end{equation}
This corresponds to routing states $x \in \mathcal{R}_h$
to the boundary kernel $P_0^{G=0}$ in the MCMC proposal. The subsequent sticky CIR step
$e^{h\mathcal{L}_1}$ then evolves mass at $0$ according to the sticky
boundary behaviour, so the split $e^{h\mathcal{L}} \approx e^{h\mathcal{L}_1} e^{h\mathcal{L}_2}$ is well-defined. 

We will use a first-order \emph{mixed-time splitting} approximation with a fixed
time $h=1/\alpha$ for the ODE part and an independent
$\tau\sim\operatorname{Exp}(\alpha)$ for the CIR part:
\[
  e^{(\mathcal{L}_1+\mathcal{L}_2)h}\varphi
   \approx 
  \mathbb{E}_\tau\!\left[e^{\mathcal{L}_1\tau}\,e^{\mathcal{L}_2h}\varphi\right]
  = \alpha(\alpha-\mathcal{L}_1)^{-1}\,e^{\mathcal{L}_2h}\varphi.
\]
The identity $\mathbb{E}_\tau[e^{\mathcal{L}_1\tau}] = \alpha(\alpha-\mathcal{L}_1)^{-1}$
uses the fact that the Laplace transform of the semigroup at
$\tau\sim\operatorname{Exp}(\alpha)$ is the resolvent.  Both the fixed-time
and the exponential-time approximations agree to first order in $h$:
expanding in $h=1/\alpha$ gives
\[
  \alpha(\alpha-\mathcal{L}_1)^{-1}e^{\mathcal{L}_2h}\varphi
  = \varphi + h(\mathcal{L}_1+\mathcal{L}_2)\varphi + O(h^2),
\]
matching $e^{(\mathcal{L}_1+\mathcal{L}_2)h}\varphi = \varphi +
h(\mathcal{L}_1+\mathcal{L}_2)\varphi + O(h^2)$.  Thus, to first order, we may fix 
the time for the ODE step and randomise the time for the CIR step, or vice versa, 
as long as the times agree in expectation.

\paragraph{The clamped Euler map and the ULA kernel}
The ULA proposal coincides with the Metropolis--Hastings proposal of \cref{alg:mcmc}: apply a \emph{clamped Euler map} $\phi_h$ to obtain the starting point, then draw from the sticky CIR resolvent kernel.  Throughout, a clamped Euler map is any $\phi_h\colon[0,\infty)\to[0,\infty)$ that coincides with the Euler step away from the $O(h)$ boundary layer and holds the atom,
\begin{equation}\label{eq:clamped-general}
  \phi_h(x)=x-hG'(x)\quad (x\ge L_h),\qquad \phi_h(0)=0,\qquad L_h\coloneqq h\,\snorm{G'}_\infty .
\end{equation}
Its behaviour on the layer $[0,L_h]$ is otherwise unconstrained; since two such maps differ only on this $O(h)$ set, the bias analysis is insensitive to the choice (\cref{rem:scope}).  Routing the layer to the boundary as in \eqref{eq:transport-ext} is one choice; the analysis below fixes a $C^1$ representative \eqref{eq:clamped-kernel}, while the experiments use the hard clamp $\phi_h(x)=\max(x-hG'(x),0)$ of \cref{alg:ula}.  Removing the acceptance step gives the ULA transition: starting from $u_k = x \ge 0$, set
\[
  u_{k+1} \sim P_{\phi_h(x)}^{G=0}(u_{T_\alpha} \in \cdot)
\]
(\cref{alg:sticky-cir}, Case~1 if $\phi_h(x) > 0$ and Case~2 if $\phi_h(x) = 0$). By~\cref{eq:trans} or equivalently \cref{eq:proposal-q}, the transition kernel is
\begin{equation}\label{eq:ula-kernel}
  K_h(x, dy)
  = \alpha\,G^{\rm sticky}_\alpha(\phi_h(x),y)\,m'(y)\,dy
  + w_0(\phi_h(x))\,\delta_0(dy), \qquad x \geq 0,
\end{equation}
where $w_0(s) = -c_\mu\,U(a,b,z_s)$ is the probability of landing at $0$ from starting point $s$ (\cref{eq:w0}).

\paragraph{Stationary distribution and bias}
The ULA kernel $K_h$ does not satisfy detailed balance with respect
to~$\pi$, so its stationary distribution $\pi_h \neq \pi$ in general.  We bound the resulting bias in two stages: a total-variation bound $\|\pi_h-\pi\|_{\mathrm{TV}}\le Ch^{\,\delta-1}$ by a consistency--stability argument (\cref{thm:ula-bias}), sharpened to a first-order expansion of the bias with an explicit, observable-independent constant, which gives $\|\pi_h-\pi\|_{\mathrm{TV}}\asymp h\abs{\log h }$, a rate independent of $\delta$ (\cref{thm:sharp-rate}).  
Throughout the analysis, $C$ denotes a positive constant whose value may change from line to line, depending only on the model parameters $\lambda, \beta, \delta, \mu$ and on $G$ (via $\|G^{(j)}\|_\infty$ for the relevant $j$); bounds involving $C$ are understood to hold uniformly for $h \in (0, h_0]$  and over test functions (usually denoted $\varphi$).

We first show uniqueness of $\pi_h$ under the following moment bound.

\begin{assumption}\label[assumption]{ass:moment}
  There exist $M < \infty$ and $h_0 > 0$ such that the ULA chain
  $(u_n)$ satisfies
  \[
    \sup_{h\le h_0} \limsup_{n\to\infty} \mathbb{E}\bigl[u_n^2\bigr]
     \le  M,
  \]
  for all initial conditions $u_0 \ge 0$.
\end{assumption}
This assumption is verified under a one-sided Lipschitz condition on $G$ in \cref{lem:lyapunov}.

Denote by $C^k_{\mathrm b}([0,\infty))$ the space of $k$-times continuously
differentiable functions on $[0,\infty)$ with bounded derivatives:
\(
  C^k_{\mathrm b}([0,\infty)) = \{f\in C^k([0,\infty)) \colon
  \|f^{(j)}\|_\infty < \infty \text{ for } j=0,1,\ldots,k\}
\) and $\norm{f}_{C^k_{\mathrm b}} = \max_{j=0,\ldots,k} \|f^{(j)}\|_\infty$.
\begin{theorem}[existence and regularity of the ULA stationary distribution]\label{thm:ula-exists} Suppose that $G' \in C^1_{\mathrm b}([0,\infty))$ and that \cref{ass:moment} holds.

  \begin{enumerate}[label=\textup{(\roman*)}]
    \item \textup{(existence and uniqueness)}
      For every $h > 0$ sufficiently small, $K_h$ is irreducible and has a unique stationary probability measure $\pi_h$.

    \item \textup{(smooth density on the interior)}
      The absolutely continuous part of $\pi_h$ on $(0,\infty)$ admits a density $\rho_h \in C^1((0,\infty))$ with respect to Lebesgue measure.  Equivalently, the relative density $w_h \coloneqq d\pi_h/d\pi_0$ exists and lies in $C^1((0,\infty))$.
  \end{enumerate}
\end{theorem}

\begin{proof}
\textup{(i)} On $\mathcal{R}_h^c$, $\phi_h$ reduces to the smooth Euler step $x \mapsto x - G'(x)h$, which is a $C^1$ diffeomorphism of $\mathcal{R}_h^c$ onto its image for $h < \|G''\|_\infty^{-1}$. Therefore, $K_h(x,\cdot) = P_{\phi_h(x)}^{G=0}(u_{T_\alpha}\in\cdot)$ inherits the irreducibility of the sticky CIR resolvent. For $x \in \{0\}\cup\mathcal{R}_h$, $\phi_h(x) = 0$ and $K_h(x,\cdot) = P_0^{G=0}(u_{T_\alpha}\in\cdot)$ is the boundary kernel, which is also irreducible.  Under \cref{ass:moment}, the second moment is bounded uniformly in $h$, giving tightness and existence of $\pi_h$ by Prokhorov's theorem; uniqueness follows from irreducibility and the density of the resolvent kernel on $(0,\infty)$.

\textup{(ii)} From the explicit form of the proposal kernel~\eqref{eq:proposal-q}, $K_h(x,dv) = \alpha\,G^{\rm sticky}_\alpha(\phi_h(x), v)\,m'(v)\,dv$ on $v > 0$, with $G^{\rm sticky}_\alpha(s,\cdot)$ smooth on $(0,\infty)$ for each fixed $s \ge 0$ via \cref{thm:transition-dist}.  Stationarity $\pi_h = \pi_h K_h$ then expresses the absolutely continuous part of $\pi_h$ on $(0,\infty)$ as
\[
  \rho_h(v)  =  m'(v)\,\alpha\int_{[0,\infty)} G^{\rm sticky}_\alpha(\phi_h(x), v)\,\pi_h(dx).
\]
The Kummer-function representation of $G^{\rm sticky}_\alpha$ from \cref{thm:transition-dist} gives uniform-in-$s$ local bounds on $\partial_v G^{\rm sticky}_\alpha(s,v)$ for $v$ in compact subsets of $(0,\infty)$, so differentiation under the integral applies and $\rho_h \in C^1((0,\infty))$.  Since $\rho_0 = m'/Z_0$ is $C^\infty$ and strictly positive on $(0,\infty)$, $w_h = \rho_h/\rho_0 \in C^1((0,\infty))$.
\end{proof}

Before the bias bound, we record the finite-time weak convergence of the ULA scheme to the sticky CIR diffusion.

\begin{theorem}[finite-time weak convergence of the ULA scheme]\label{thm:ula-weak-conv}
  Suppose $G\in C^2([0,\infty))$ with $\|G'\|_\infty<\infty$, and \cref{ass:moment} holds.  Then, as $h\to 0$, for every $\varphi\in C_{\mathrm b}([0,\infty))$ and every $t>0$,
  \[
    \mathbb{E}[\varphi(X_{n_h})\mid X_0=x]
      \;\to\;
      \mathbb{E}[\varphi(u_t)\mid u_0=x],
      \qquad n_h \coloneqq  \lfloor t/h\rfloor,
  \]
  uniformly for $x$ in compact subsets of $[0,\infty)$.
\end{theorem}

\begin{proof}
  Write $\mathcal{A}_h^K=h^{-1}(K_h-I)$ for the discrete generator and
  $T(t)\varphi(x)=\mathbb{E}[\varphi(u_t)\mid u_0=x]$ for the sticky CIR
  semigroup, a strongly continuous contraction semigroup on $C_0([0,\infty))$
  (\cref{thm-exist-potential}), with generator $\mathcal{L}$ characterised by the Wentzell
  boundary condition \cref{eq-wentzell}.
  Two boundary facts are used.  First, by \cref{lem:resolvent-regularity}, a
  function $\varphi\in\mathrm{dom}(\mathcal{L})$ has a finite scale-flux
  $\tfrac{d\varphi}{ds}(0^+)$. For the second, observe that the Girsanov tilt by $G$ (\cref{thm-exist-potential}) adds only interior
  drift, leaving the boundary stickiness $m(\{0\})=1/\mu$ intact, while the
  scale density obeys $s_G'(x)=s_0'(x)\,e^{\beta(G(x)-G(0))}$ with factor
  $\to1$ as $x\to0^+$, so the flux at $0$ is unchanged. Thus,  the potential leaves the sticky
  boundary generator unchanged and
  \begin{equation}\label{eq:bdy-generator-match}
    \mathcal{L}\varphi(0)=\mathcal{L}_1\varphi(0),
    \qquad\varphi\in\mathrm{dom}(\mathcal{L}).
  \end{equation}

  \emph{Generator convergence.}  Away from the boundary $\phi_h(x)=x-hG'(x)$ and
  Taylor expansion gives
  $K_h\varphi(x)=(R_h\varphi)(\phi_h(x))=\varphi(x)+h\mathcal{L}\varphi(x)+o(h)$,
  so $\mathcal{A}_h^K\varphi\to\mathcal{L}\varphi$ uniformly on compact subsets
  of $(0,\infty)$.  At the atom, the held step gives
  $\mathcal{A}_h^K\varphi(0)=h^{-1}(R_h\varphi(0)-\varphi(0))$, which converges
  to the $G=0$ boundary generator $\mu\,\tfrac{d\varphi}{ds}(0^+)
  =\mathcal{L}_1\varphi(0)$ and hence to
  $\mathcal{L}\varphi(0)$ by \eqref{eq:bdy-generator-match}.  The only region of
  non-uniformity is the layer $\mathcal{R}_h=\{0<x\le hG'(x)\}$ (nonempty only
  if $G'(0)>0$), where the routed step gives
  $\mathcal{A}_h^K\varphi(x)=h^{-1}((R_h\varphi)(0)-\varphi(x))$ and the cusp
  makes $\varphi(0)-\varphi(x)=\Theta(x^{2-\delta})=\Theta(h^{2-\delta})$, so
  $\mathcal{A}_h^K\varphi=O(h^{1-\delta})$ on $\mathcal{R}_h$.  This excess is
  confined to a set entered with probability $O(h^\delta)$ in a single step,
  uniformly in the starting point (\cref{lem:density-bound}).

  The chain $(X_n)$ solves the martingale problem for
  $\mathcal{A}_h^K$; \cref{ass:moment} gives tightness, hence a convergent
  subsequence.  The additive functional satisfies
  $\sum_{k<n_h}h\,\mathcal{A}_h^K\varphi(X_k)\to\int_0^t\mathcal{L}\varphi(u_s)\,ds$:
  off $\mathcal{R}_h$ this is the Riemann-sum limit just established, while the
  layer contributes at most
  $(\#\text{visits})\times h\,O(h^{1-\delta})
   =O(t\,h^{\delta-1})\times O(h^{2-\delta})=O(th)\to0$, the visit count
  following from the per-step entry probability $O(h^\delta)$
  (\cref{lem:density-bound}).  The discrete generators converge, $\mathcal{A}_h^K\varphi\to\mathcal{L}\varphi$ boundedly and pointwise on the core $\mathrm{dom}(\mathcal{L})$ (uniformly off $\mathcal{R}_h$, the layer contributing $O(th)\to0$ as above); with the tightness from \cref{ass:moment}, the diffusion-approximation theorem \citep[Thm.~4.8.2]{EthierKurtz1986} identifies every limit point as the unique solution of the well-posed
  $\mathcal{L}$-martingale problem \citep[\S4]{ItoMcKean1965}, giving
  $\mathbb{E}[\varphi(X_{n_h})\mid X_0=x]\to\mathbb{E}[\varphi(u_t)\mid u_0=x]$
  for $\varphi\in C_0$, $n_h=\lfloor t/h\rfloor$; \cref{ass:moment} makes the
  laws tight, upgrading to $\varphi\in C_{\mathrm b}$
  \citep[Theorem~25.10]{Billingsley1999}.  The convergence is uniform for $x$ in any
  compact $K\subset[0,\infty)$: the finite-time second moment is bounded uniformly in the
  starting point, $\mathbb{E}_x[X_{n_h}^2]\le x^2+b/\gamma$ by \eqref{eq:lyapunov-geom}, so the
  laws $\{X_{n_h}\mid X_0=x\}_{h\le h_0,\,x\in K}$ are tight; for $x_h\to x_\star$ in $K$, the
  generator and layer estimates above being uniform in the starting point, every weak limit of
  $(X_{n_h}\mid X_0=x_h)$ solves the $\mathcal{L}$-martingale problem from $x_\star$ and so equals
  the law of $u_t$ from $x_\star$, while
  $\mathbb{E}[\varphi(u_t)\mid u_0=x_h]\to\mathbb{E}[\varphi(u_t)\mid u_0=x_\star]$ by the Feller
  property; a subsequence argument gives the uniformity.
\end{proof}

The bias analysis follows the classical \emph{consistency--stability} form: a one-step weak error and a uniform geometric contraction of the chain.  Before proceeding with this analysis, we first make precise the clamping mechanism for the Euler map necessary for the proofs.

\paragraph{The analysis clamp}
The one-step weak error is governed by how the boundary atom is propagated.  For the analysis, we fix the clamped Euler map \eqref{eq:clamped-general} to the $C^1$ form
\begin{equation}\label{eq:clamped-kernel}
  K_h=R_h\circ\phi_h,\qquad
  \phi_h(x)=\chi_h(x)\,\bigl(x-hG'(x)\bigr),
\end{equation}
where $\chi_h\in C^1([0,\infty))$ is a cutoff with $\chi_h\equiv1$ on $[L_h,\infty)$ and $\chi_h(0)=0$ (e.g., $\chi_h(x)=\sin^2(\pi x/2L_h)$ on $[0,L_h]$); then $\phi_h$ is $C^1$, holds the atom ($\phi_h(0)=0$), and equals $x-hG'(x)$ for $x\ge L_h$.  The bias bound below uses only the atom-holding property $\phi_h(0)=0$; the $C^1$ regularity enters solely through the bounded clamp slope used in the contraction argument (\cref{app:contraction-nonconvex}).  Holding the atom at $0$, rather than transporting it into the resolvent kernel's boundary cusp (as the plain Euler step would for $G'(0)<0$), is what keeps the one-step weak error small for \emph{all} $G$.  

We are ready to present consistency \cref{lem:boundary-charge} and stability \cref{ass:contraction}.
Write $\nu_h\coloneqq \pi K_h-\pi$ for the signed one-step defect, with Lebesgue decomposition $\nu_h(dx)=\nu_h(\{0\})\,\delta_0(dx)+r(x)\,dx$ on $[0,\infty)$ and weighted norm $\|\nu_h\|_{V}=\int V\,d|\nu_h|=|\nu_h(\{0\})|+\int_0^\infty(1+x^2)|r(x)|\,dx$.

\begin{lemma}[consistency: one-step weighted defect]\label[lemma]{lem:boundary-charge}
  Assume $G'\in C^2_{\mathrm b}([0,\infty))$ and $\delta \in (1,2)$.  There exist $C_{\mathrm b}<\infty$ and $h_0>0$ such that, for all $h\in(0,h_0]$, the one-step defect of the clamped kernel \eqref{eq:clamped-kernel} satisfies
  \begin{equation}\label{eq:boundary-charge}
    \|\nu_h\|_{V}=|\nu_h(\{0\})|+\int_0^\infty(1+x^2)|r(x)|\,dx\le C_{\mathrm b}\,h^{\delta}.
  \end{equation}
\end{lemma}
\noindent
The proof (Appendix~\ref{app:smooth-bias}) is a boundary-layer estimate: the defect is dominated by the $O(h)$ layer at the atom, where two contributions of the same order $h^{(\delta+1)/2}$ cancel, leaving the net $O(h^\delta)$.

\begin{assumption}[contraction]\label[assumption]{ass:contraction}
  Let $V(x)=1+x^2$ (the Lyapunov function of \cref{ass:moment}) and $\|\mu\|_V\coloneqq \sup_{|f|\le V}|\mu(f)|=\int V\,d|\mu|$.  There exist $c>0$, $C<\infty$, and $h_0>0$ such that, for all $h\in(0,h_0]$, the chain $K_h$ has a unique stationary distribution $\pi_h$ and is $V$-uniformly geometrically ergodic, uniformly in $h$:
  \begin{equation}\label{eq:ass-contraction}
    \bigl\|K_h^n(x,\cdot)-\pi_h\bigr\|_{V}\le C\,V(x)\,(1-ch)^n,\qquad x\in[0,\infty),\ \ n\ge0 .
  \end{equation}
\end{assumption}
\noindent
The contraction rate is bounded away from $1$ by an amount of order $h$, uniformly in $h$: over one diffusion time, or $O(h^{-1})$ steps, the non-degenerate noise spreads the law of the chain onto a common interior set, while the second-moment drift of \cref{ass:moment} controls returns from infinity, a Foster--Lyapunov/Harris mechanism that requires no convexity (\cref{lem:contraction-nonconvex}).  The weight $V$ is needed because the Poisson solution below grows at infinity.

We join the consistency lemma and stability assumption in \cref{thm:ula-bias} using the solution to a Poisson equation.
\begin{lemma}[bound on the discrete Poisson solution]\label[lemma]{lem:poisson-sup}
  Suppose \cref{ass:contraction} holds.  For bounded measurable $\eta\colon[0,\infty)\to\mathbb R$, the Neumann series $g\coloneqq \sum_{n\ge0}K_h^n(\eta-\pi_h(\eta))$ converges and solves $(I-K_h)g=\eta-\pi_h(\eta)$, with
  \begin{equation}\label{eq:poisson-sup}
    |g(x)|\le C\,h^{-1}\,V(x)\,\snorm{\eta}_\infty,\qquad x\in[0,\infty),\ \ h\in(0,h_0].
  \end{equation}
\end{lemma}

\begin{proof}
  Write $\hat\eta\coloneqq \eta-\pi_h(\eta)$.  Since $(K_h^n\hat\eta)(x)=(K_h^n\eta)(x)-\pi_h(\eta)=\int\eta\,\bigl(K_h^n(x,\cdot)-\pi_h\bigr)$ and $|\eta|\le\snorm{\eta}_\infty\le\snorm{\eta}_\infty V$, the ergodicity \eqref{eq:ass-contraction} gives $|(K_h^n\hat\eta)(x)|\le\snorm{\eta}_\infty\,\|K_h^n(x,\cdot)-\pi_h\|_{V}\le C\snorm{\eta}_\infty V(x)(1-ch)^n$.  Summing the geometric series, $g=\sum_{n\ge0}K_h^n\hat\eta$ converges with $(I-K_h)g=\hat\eta$ and
  \[
    |g(x)|\le\sum_{n\ge0}C\snorm{\eta}_\infty V(x)(1-ch)^n=\frac{C}{ch}V(x)\snorm{\eta}_\infty=O(h^{-1})V(x)\snorm{\eta}_\infty . \qedhere
  \]
\end{proof}

\begin{proposition}[a priori total-variation bound]\label[proposition]{thm:ula-bias}
    Suppose $G'\in C^2_{\mathrm b}([0,\infty))$ and \cref{ass:moment,ass:contraction} hold.  For $\delta\in (1,2)$, there are $C,h_1>0$ such that, for all $h\in(0,h_1]$,
  \begin{equation}\label{eq:ula-tv}
    \|\pi_h-\pi\|_{\mathrm{TV}}\le C\,h^{\,\delta-1}.
  \end{equation}
\end{proposition}

\begin{proof}
  By \cref{ass:moment,ass:contraction} the chain $K_h$ has a unique stationary distribution $\pi_h$.  Fix a bounded measurable $\eta$ and let $g$ be the discrete Poisson solution of \cref{lem:poisson-sup}, $(I-K_h)g=\eta-\pi_h(\eta)$.  Since $\pi$ is a probability measure,
  \[
    \nu_h(g)=(\pi K_h-\pi)(g)=-\pi\bigl((I-K_h)g\bigr)
      =-\pi\bigl(\eta-\pi_h(\eta)\bigr)=\pi_h(\eta)-\pi(\eta).
  \]
  By \cref{lem:poisson-sup}, $|g(x)|\le Ch^{-1}V(x)\snorm{\eta}_\infty$, so integrating against the signed measure $\nu_h$ and using \cref{lem:boundary-charge},
  \[
    \begin{aligned}
      |\pi_h(\eta)-\pi(\eta)|=|\nu_h(g)|
      &\le\int|g|\,d|\nu_h|\le Ch^{-1}\snorm{\eta}_\infty\int V\,d|\nu_h|\\
      &=Ch^{-1}\snorm{\eta}_\infty\,\|\nu_h\|_V\le C\,h^{\,\delta-1}\snorm{\eta}_\infty,
    \end{aligned}
  \]
  for every bounded measurable $\eta$.  Taking the supremum over $\snorm{\eta}_\infty\le1$ gives \eqref{eq:ula-tv}.
\end{proof}

\begin{remark}[satisfying the hypotheses]\label[remark]{rem:discharge}
  \cref{ass:moment} holds under a one-sided Lipschitz condition (see \cref{lem:lyapunov}), and \cref{ass:contraction} holds whenever $G$ has bounded gradient, in particular for non-convex potentials (see \cref{lem:contraction-nonconvex}).
\end{remark}
\subsection{Sharp bias rate}
The rate $\delta-1$ of \cref{thm:ula-bias} is the product of the consistency error $\|\nu_h\|_{V}=O(h^\delta)$ (\cref{lem:boundary-charge}; confirmed directly, across $\delta$, by the exact atom-defect computation of \cref{fig:one-step-atom}) and the Poisson bound $|g|_V=O(h^{-1})$ (\cref{lem:poisson-sup}).  The intermediate inequality $|\nu_h(g)|\le|g|_V\,\|\nu_h\|_{V}$ discards the sign correlation between $g$ and $\nu_h$, and it is far from sharp: the true stationary bias is smaller, its rate is independent of $\delta$, and it admits a first-order expansion with an explicit constant (\cref{thm:sharp-rate}). \Cref{thm:ula-bias} is nonetheless not superfluous: its bound \eqref{eq:ula-tv} is the a priori input that the sharp expansion bootstraps to that finer rate.

\begin{theorem}[first-order expansion of the bias: the sharp rate]\label{thm:sharp-rate}
  Let the assumptions of \cref{thm:ula-bias} hold, and fix $\delta\in(1,2)$.  Then, for every bounded measurable $\eta\colon[0,\infty)\to\mathbb R$,
  \begin{equation}\label{eq:sharp-law}
    \pi_h(\eta)-\pi(\eta)\;=\;K_\star\,\hat\eta(0)\,h\abs{\log h } \;+\;O\bigl(h\,\snorm{\eta}_\infty\bigr),
    \qquad \hat\eta(0)\coloneqq \eta(0)-\pi(\eta),
  \end{equation}
  with the observable-independent constant
  \begin{equation}\label{eq:Kstar}
    K_\star\;=\;\tfrac12\,(\delta-1)\,\beta\,G'(0)^2\,\pi(\{0\}) .
  \end{equation}
  Equivalently, $\pi_h=\pi+K_\star\,h\abs{\log h } \,(\delta_0-\pi)+O(h)$ in total variation. 
\end{theorem}

The proof is given in Appendix~\ref{app:sharp-rate}.  Its key step removes the resolvent kernel from the stationarity equation exactly: testing $\pi_hK_h=\pi_h$ against the \emph{preconditioned} function $(I-h\mathcal L_1)g_\star$, where $\mathcal Lg_\star=-\hat\eta$ is the continuous Poisson solution, gives $K_h\bigl[(I-h\mathcal L_1)g_\star\bigr]=g_\star\circ\phi_h$ exactly, so the entire bias is carried by the deterministic Euler map:
\begin{equation}\label{eq:precond-identity}
  \pi_h(\eta)-\pi(\eta)
  \;=\;\pi_h\Bigl(\bigl[\,G'g_\star'-h^{-1}\bigl(g_\star-g_\star\circ\phi_h\bigr)\bigr]\mathbf 1_{(0,\infty)}\Bigr).
\end{equation}

\begin{remark}[structure of the leading error]\label[remark]{rem:sharp}
  \cref{thm:sharp-rate} supersedes the rate of \cref{eq:ula-tv}: $h\abs{\log h } \ll h^{\delta-1}$ for every $\delta\in(1,2)$, and the exponent is independent of $\delta$.  The leading error is \emph{rank one}: a transfer of mass $K_\star h\abs{\log h } $ from the bulk (removed proportionally to $\pi$) onto the sticky atom.  Since $K_\star\ge0$, the ULA \emph{over-weights} the atom at leading order, by an amount set by the boundary drift through $G'(0)^2$.  Set against the classical theory, this is a genuine departure: for the Euler--Maruyama discretisation of a non-degenerate Langevin diffusion with smooth potential on $\mathbb{R}^d$, the stationary bias is first order, $\pi_h(\eta)-\pi(\eta)=C(\eta)h+O(h^2)$ \citep{TalayTubaro1990}, and the resulting ULA has an $O(h)$ bias \citep{RobertsTweedie1996,DurmusMoulines2017}.  The accessible boundary and the atom of $\pi$ inflate this to $h\abs{\log h}$, the extra $\abs{\log h}$ being generated on the $O(h)$ boundary layer, where the truncation error pairs with the cusped invariant density to give a logarithmically divergent integral cut off at the routing scale $x\asymp h$; the classical $O(h)$ rate is recovered when the boundary is inactive ($G'(0)=0$, or $\hat\eta(0)=0$).
\end{remark}

\begin{remark}[scope]\label[remark]{rem:scope}
  The bounds \eqref{eq:ula-tv} and \eqref{eq:sharp-law} are established for the analysis clamp \eqref{eq:clamped-kernel}.  Any other clamped Euler map \eqref{eq:clamped-general}, including the hard clamp used in \cref{sec:experiments}, differs only on the $O(h)$ boundary layer and exhibits the same rates, with the same leading constant $K_\star$ (\cref{rem:sharp}(iii)).
\end{remark}

%
%
%
\begin{algorithm}
  \caption{ULA for the sticky CIR with potential $G$}
  \label[algorithm]{alg:ula}
  \begin{algorithmic}[1]
  \Statex \textsc{Precomputation.} Resolvent tables, $c_\mu$ and $p_{\rm leave}$ as in \cref{alg:sticky-cir}.
  \Statex
  \State Draw $u_{k+1}\sim P_{\phi_h(u_k)}^{G=0}$ from \cref{alg:sticky-cir} (Case~1 if $\phi_h(u_k)>0$, else Case~2): the proposal of \cref{alg:mcmc}, accepted unconditionally.
  \end{algorithmic}
  \end{algorithm}
\section{Numerical experiments}\label{sec:experiments}

We report two experiments at base parameters $\lambda=1$, $\beta=2$, $\delta=1.5$:
\begin{itemize}
  \item \emph{Experiment 1} (\cref{fig:bdy-error,fig:ess,fig:density}): compare MCMC (\cref{alg:mcmc}, \cref{sampler-gpotential}) and ULA (\cref{alg:ula}, \cref{sampler-ula}) on a grid of potentials $G(u) \in \{u^2/2,\,(u-1)^2/2,\,u^3/3\}$, stickiness $\mu \in \{0.5,\,1.0,\,2.0\}$, and step rates $\alpha \in \{2,\,5,\,10\}$ ($200{,}000$ steps, $4$ chains, $10{,}000$-step warmup), with diagnostics for boundary mass error, ESS per second, and interior density.   
  \item \emph{Experiment 2} (\cref{fig:MALA-acceptance,fig:MALA-bdy,fig:ula-bias}): fix $\mu=1$ and consider four potentials: $G=0$ (reference exact sampler), $G(u)=u^2/2$ ($G'(0)=0$), $G(u)=(u-1)^2/2$ ($G'(0)=-1$), and $G(u)=2u$ ($G'(0)>0$). We present MCMC acceptance rates (at $\alpha=5$, $h=0.2$, $30{,}000$ steps) and ULA bias decay across $\alpha$ ($10{,}000$ steps per $\alpha$) separately. Theoretical boundary mass $\pi(\{0\})$ for each potential is given in \cref{tab:theory}.
\end{itemize}
ESS (effective sample size, accounting for autocorrelation within each chain) and ESS per second (ESS normalised by wall-clock time) are computed via ArviZ~\citep{arviz2019} on interior draws ($u>0$), trimmed to equal length across chains.
\paragraph{Boundary mass error}

\cref{fig:bdy-error} shows the signed error $\hat\pi(\{0\})-\pi(\{0\})$,
where $\hat\pi(\{0\})$ denotes the empirical boundary fraction averaged
across $4$ chains, with $\pm 2$ standard error bands.
 The MCMC error (green) is indistinguishable from Monte Carlo error across all settings,
confirming exact invariance. The ULA error (orange) is larger and reduces as $\alpha$ increases ($h=1/\alpha\to0$), consistent with the vanishing atom bias guaranteed by \cref{thm:ula-bias}.
\begin{figure}[ht]
  \centering
  \includegraphics[width=\textwidth]{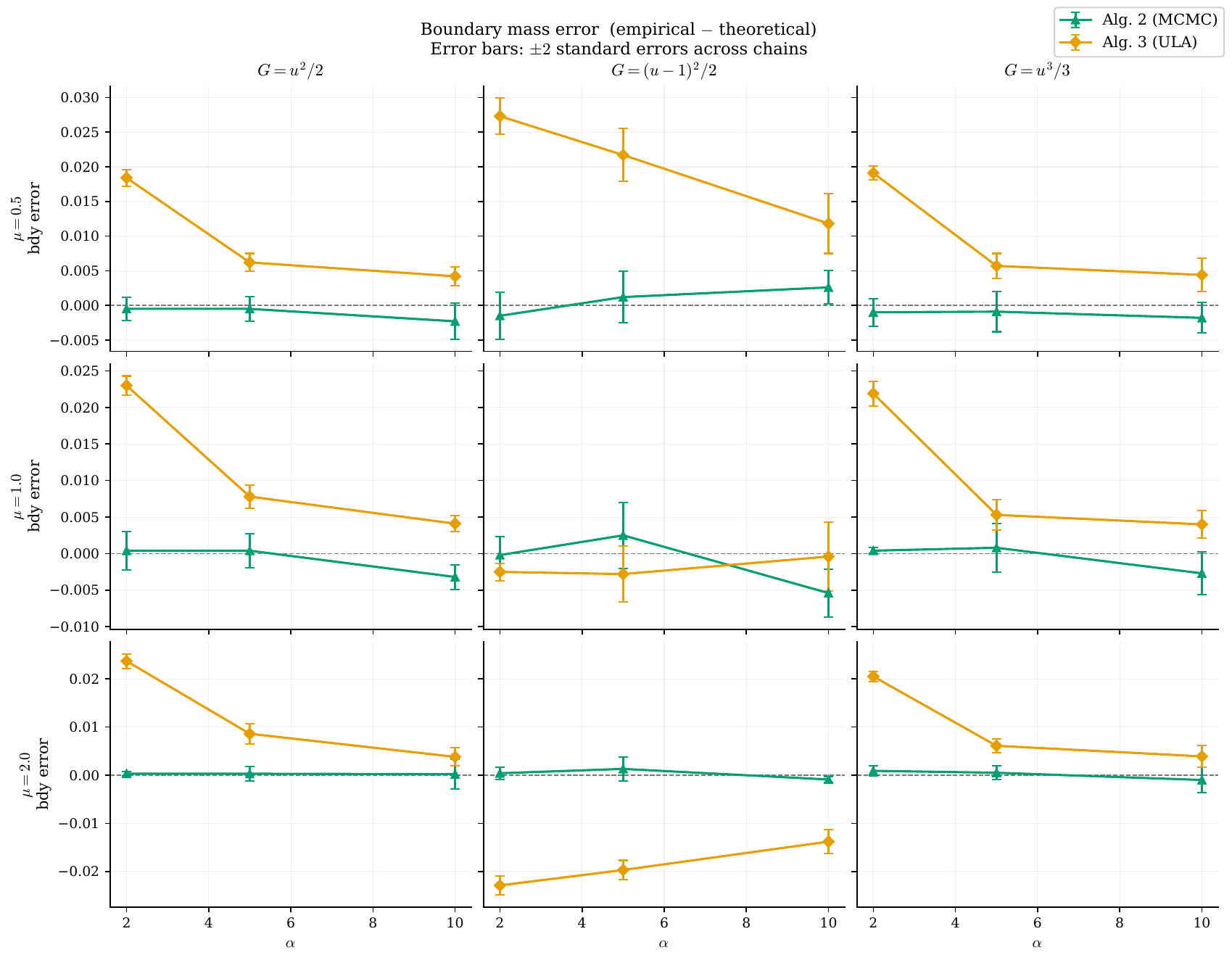}
  \caption{Signed boundary mass error $\hat\pi(\{0\})-\pi(\{0\})$
    for both algorithms across potentials (columns) and stickiness
    parameters (rows), as a function of step rate $\alpha$.
    Error bars are $\pm 2$ standard errors across $4$ chains.
    Small values confirm that MCMC targets the correct invariant measure, with ULA exhibiting  bias that increases for small $\alpha$.}
  \label{fig:bdy-error}
\end{figure}


\paragraph{Efficiency: ESS per second}

\cref{fig:ess} shows bulk ESS/s of wall-clock time.  ULA achieves the highest ESS/s across most settings and  ESS/s decreases sharply with $\alpha$ for both algorithms, with no clear dependence on $\mu$.

\begin{figure}[ht]
  \centering
  \includegraphics[width=\textwidth]{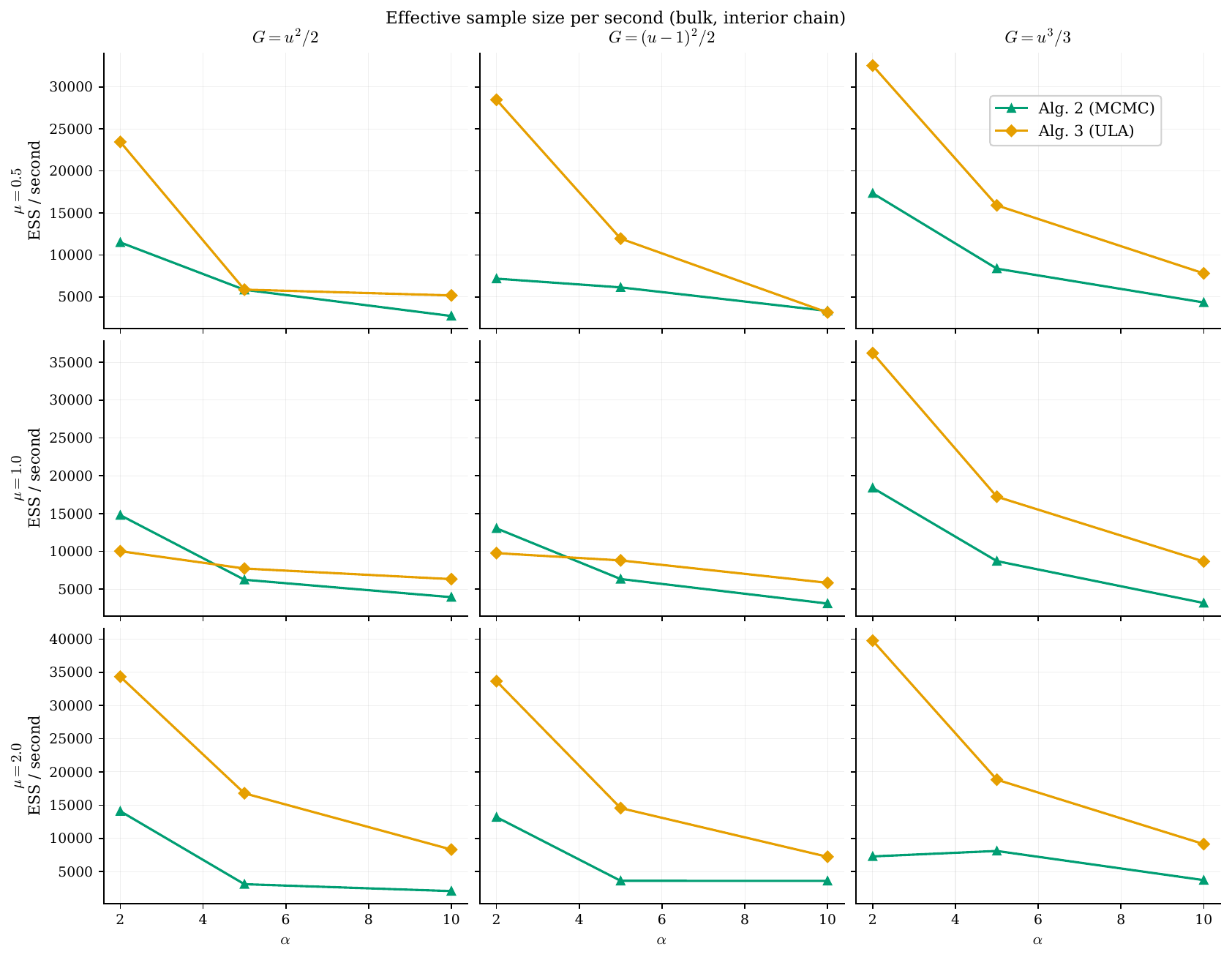}
  \caption{Bulk ESS per second (interior chain) for two algorithms
    across potentials (columns) and stickiness parameters (rows),
    as a function of $\alpha$.  Higher is better.  ESS/s rewards
    both mixing efficiency and computational speed.}
  \label{fig:ess}
\end{figure}

\paragraph{Interior density}
\cref{fig:density} overlays the empirical interior density from both
algorithms on the theoretical $\pi$ at the reference setting
$G=(u-1)^2/2$, $\mu=1$, $\alpha=20$.  Both samplers match the
interior density and boundary mass closely: the MCMC gives
$\hat\pi(\{0\})=0.277$ and the ULA gives $\hat\pi(\{0\})=0.275$,
against the theoretical value $0.275$.

\begin{figure}[ht]
  \centering
  \includegraphics[width=\textwidth]{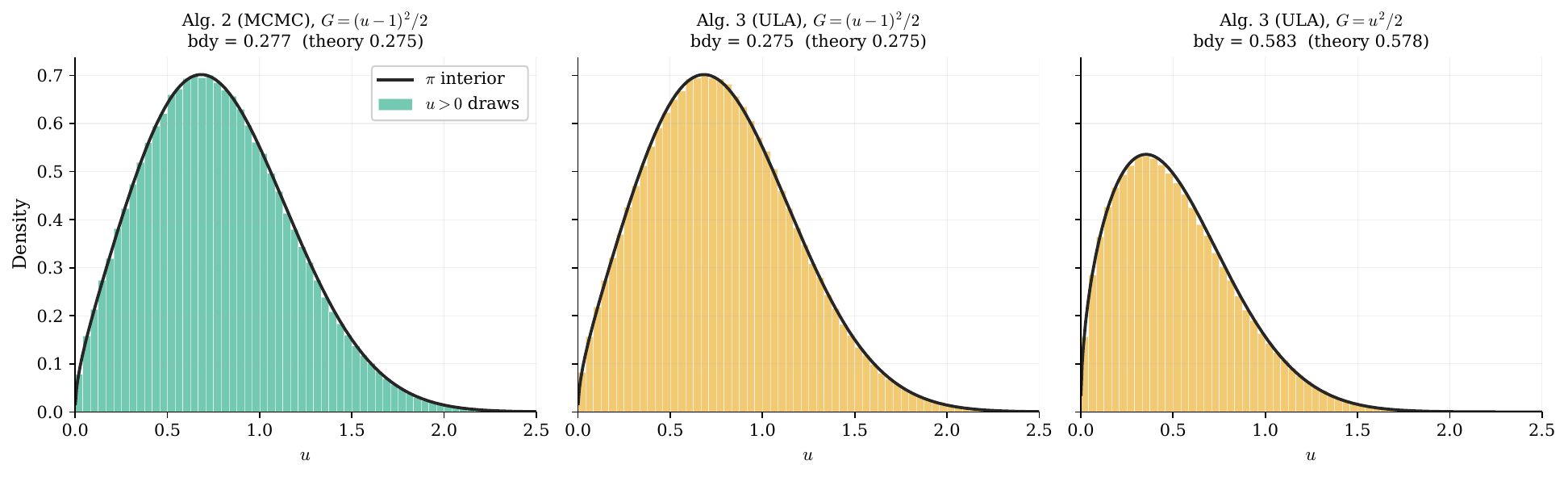}
  \caption{Empirical interior density (histogram) versus theoretical
  $\pi$ (black curve) at $\mu=1$, $\alpha=20$ showing close match in all cases.
  Left: MCMC (Alg.~2) with $G=(u-1)^2/2$.
  Centre: ULA (Alg.~3) with $G=(u-1)^2/2$.
  Right: ULA with $G=u^2/2$, where $G'(0)=0$.}
  \label{fig:density}
\end{figure}

\begin{table}[ht]
\centering
\begin{tabular}{lc}
\toprule
Potential & $\pi(\{0\})$ \\
\midrule
$G=0$           & $0.449$ \\
$G=u^2/2$       & $0.579$ \\
$G=(u-1)^2/2$   & $0.275$ \\
$G=2u$          & $0.844$ \\
\bottomrule
\end{tabular}
\caption{Theoretical boundary mass $\pi(\{0\})$ for each test potential
($\lambda=1$, $\beta=2$, $\delta=1.5$, $\mu=1$).}
\label{tab:theory}
\end{table}

\paragraph{MCMC acceptance rates}

\cref{fig:MALA-acceptance,fig:MALA-bdy} summarise the MCMC performance.
\cref{fig:MALA-acceptance} shows the three acceptance rates:
interior~$\to$~interior, interior~$\to$~boundary (to-0), and
boundary~$\to$~interior (from-0). Acceptance rates remain above 70\% across all potentials and move types. \cref{fig:MALA-bdy} shows good agreement of the 
empirical boundary fraction against the theoretical target.

\begin{figure}[ht]
  \centering
  \includegraphics[width=\textwidth]{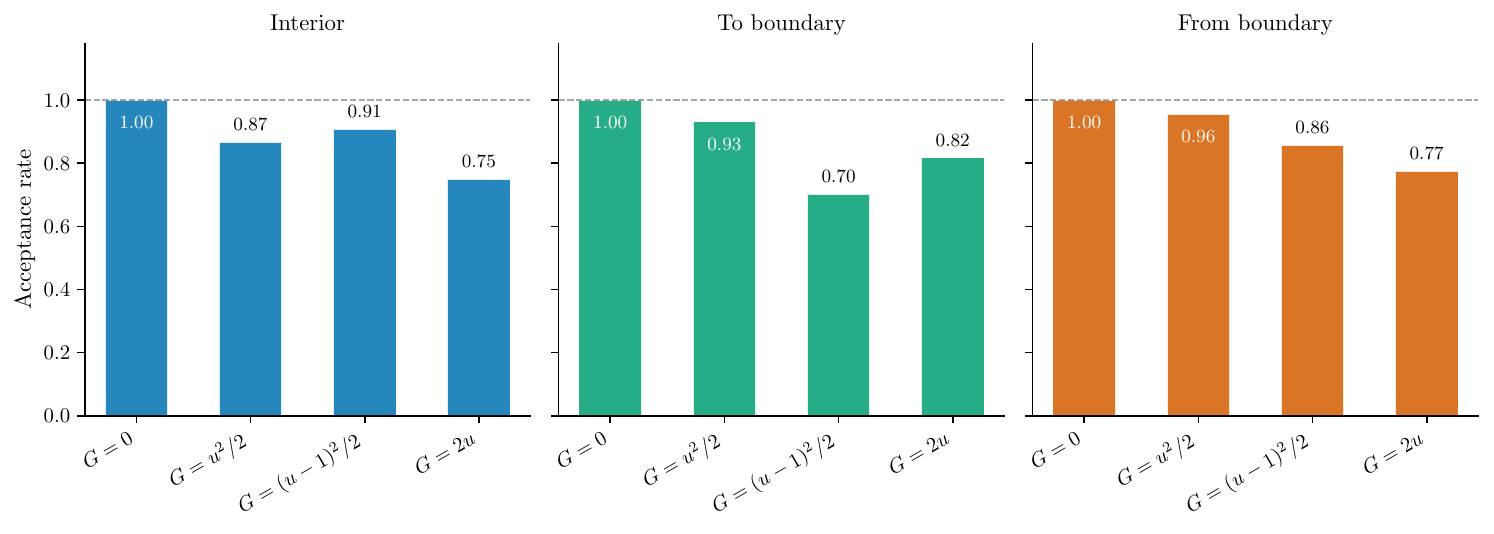}
  \caption{MCMC acceptance rates at $\alpha=5$ ($h=0.2$) for four test
    potentials.  Left: interior--interior rate.  Centre: interior-to-boundary
    rate.  Right: boundary-to-interior rate.  Dashed line marks perfect
    acceptance ($=1$).}
  \label{fig:MALA-acceptance}
\end{figure}

\begin{figure}[ht]
  \centering
  \includegraphics[width=0.55\textwidth]{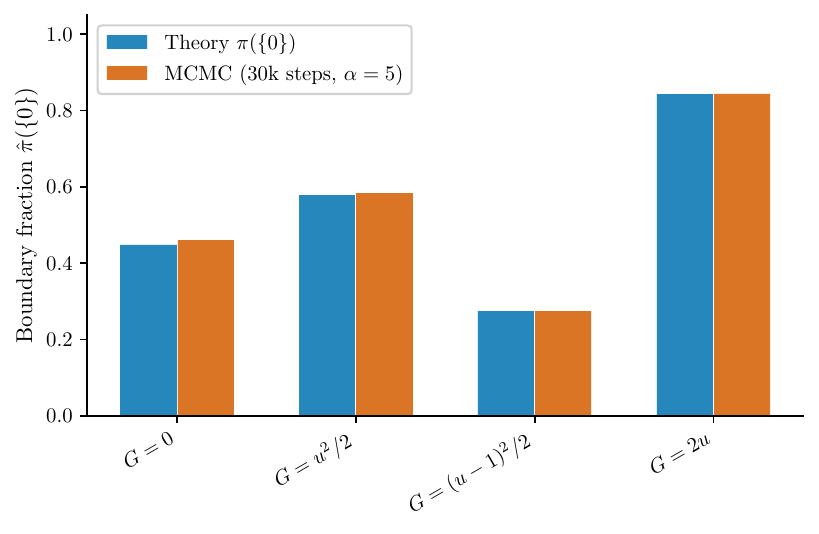}
  \caption{Empirical boundary fraction of the MCMC chain (30k steps,
    $\alpha=5$) versus the theoretical target $\pi(\{0\})$.}
  \label{fig:MALA-bdy}
\end{figure}

\paragraph{ULA boundary fraction vs step size}

The bias is consistent with the bound of
\cref{thm:ula-bias} (at $\delta=1.5$ the guarantee is $O(h^{1/2})$)
across all potentials tested.  The
clearest illustration is $G=2u$, where the bias is largest at large
$h$ and decreases visibly as $h\to 0$.  For the remaining potentials,
the bias is small (below $0.03$) across all step sizes tested, with
the decrease not clearly distinguishable from Monte Carlo
noise at the chain lengths used.

\begin{figure}[ht]
  \centering
  \includegraphics[width=\textwidth]{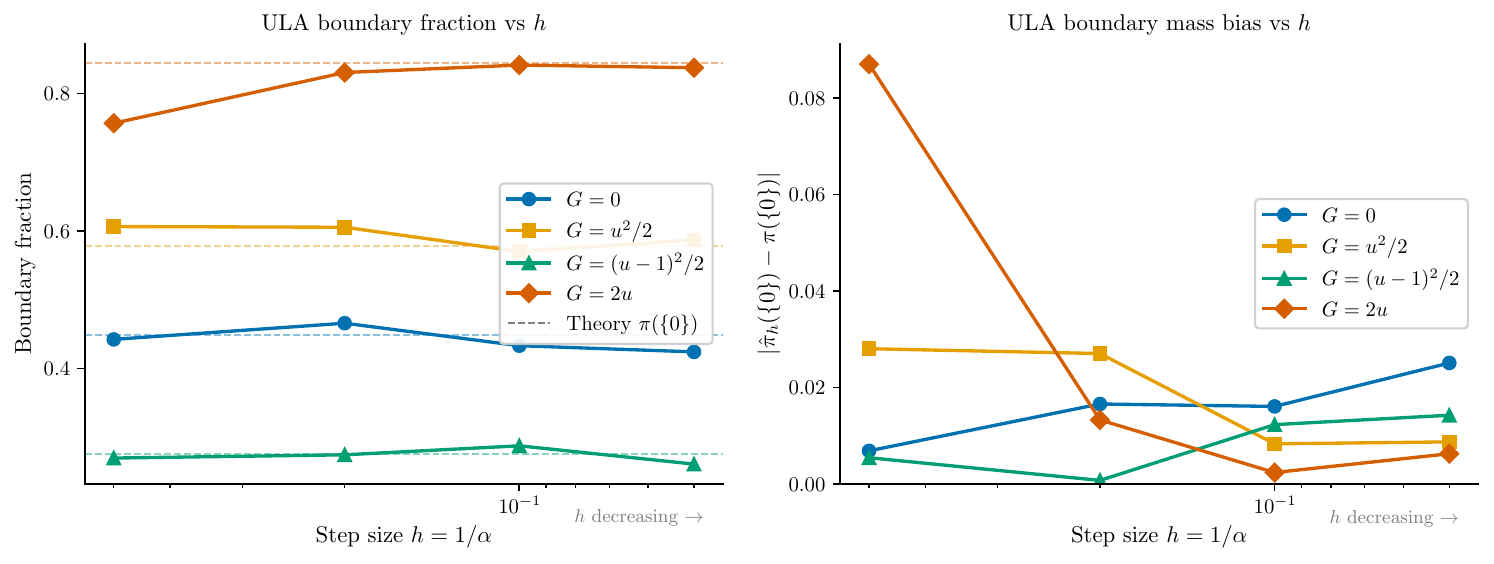}
  \caption{ULA boundary fraction (left) and absolute bias (right) versus
    step size $h=1/\alpha$.  Dashed horizontal lines on the left panel
    mark the theoretical values $\pi(\{0\})$.  The $x$-axis is reversed
    so that larger steps are on the left.  Each data point is computed from a chain of 10,000 steps.}
  \label{fig:ula-bias}
\end{figure}

\paragraph{One-step consistency at the atom}
The stationary bias of \cref{fig:ula-bias} is hard to resolve as $h\to0$, where
it drops below the Monte Carlo noise of a finite chain.  We therefore probe the
rate at its source: the one-step weighted defect
$\nu_h(\{0\})=(\pi K_h)(\{0\})-\pi(\{0\})$, the dominant term of
\cref{lem:boundary-charge}.  Since the one-step land-at-zero probability is
explicit, $K_h(x,\{0\})=w_0(\phi_h(x))$ \eqref{eq:w0}, this defect is obtained by
quadrature against the closed-form invariant measure, free of Monte Carlo error,
\[
  \nu_h(\{0\})=\int_{[0,\infty)} w_0(\phi_h(x))\,\pi(dx)-\pi(\{0\}).
\]
\cref{fig:one-step-atom} reports $|\nu_h(\{0\})|$ for $G(u)=2u$ and
$\delta\in\{1.3,1.5,1.7\}$, down to $h=10^{-3}$.  The defect decays as $h^{\delta}$, matching
the one-step consistency estimate of \cref{lem:boundary-charge}; amplified by the
$O(h^{-1})$ bound on the Poisson solution (\cref{lem:poisson-sup}), this is the
$O(h^{\delta-1})$ total-variation bias of \cref{thm:ula-bias}.

\begin{figure}[ht]
  \centering
  \includegraphics[width=0.62\textwidth]{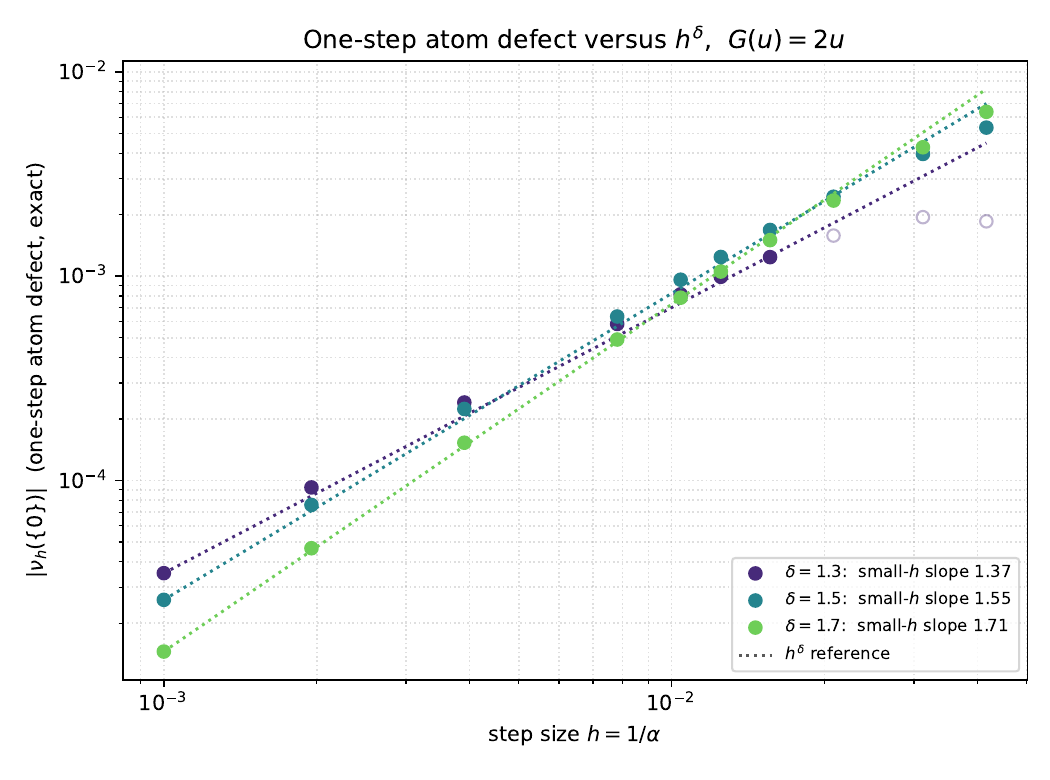}
  \caption{One-step atom defect
    $|\nu_h(\{0\})|=|(\pi K_h)(\{0\})-\pi(\{0\})|$ for the clamped ULA with
    $G(u)=2u$, evaluated exactly by quadrature against $\pi$, versus step size
    $h=1/\alpha$, for three values of $\delta$.  Dotted lines are $h^{\delta}$
    references and the legend reports the fitted small-$h$ slope; the defect
    follows $h^{\delta}$, the consistency rate behind the $O(h^{\delta-1})$ bias.
    Faded points (largest $h$, $\delta=1.3$) are pre-asymptotic and excluded from
    the slope.}
  \label{fig:one-step-atom}
\end{figure}

\paragraph{Sharpness of the stationary rate}
\cref{fig:sharp-rate} verifies the first-order expansion \cref{eq:sharp-law} of
\cref{thm:sharp-rate}: the stationary law
$\pi_h$ of the clamped kernel is obtained by solving $\pi_hK_h=\pi_h$ directly
on a grid (a linear system, free of Monte Carlo error), for $G(u)=(u-1)^2/2$, $\delta\in\{1.3,1.5,1.7\}$,
and $h\in[1/256,1/8]$.  Panel~(a) shows the atom-mass bias
$|\pi_h(\{0\})-\pi(\{0\})|$: its decay is visibly steeper than the
$h^{\delta-1}$ guides of \cref{thm:ula-bias}, follows $C\,h\abs{\log h } $, and its
rate is independent of
$\delta$.  Panel~(b) shows the collapse test: dividing the bias by
$\hat\eta(0)\,h\abs{\log h } $ sends the four observables
$\eta\in\{x,\,x^2,\,(1+x)^2,\,\mathbf 1_{\{0\}}\}$ onto a single constant for
each $\delta$, approaching the explicit $K_\star$ of \cref{eq:Kstar} as
$h\downarrow0$. The spread across
observables likewise shrinks like $1/\abs{\log h } $, largest at the smallest
$\delta$ as in \cref{eq:Kstar},  while the cusp-free control $\eta_c$ (chosen so that
$\hat\eta_c(0)=0$) falls towards zero: the observable enters the leading bias
only through $\hat\eta(0)$, as the rank-one structure of
\cref{thm:sharp-rate} dictates.

\begin{figure}[ht]
  \centering
  \includegraphics[width=\textwidth]{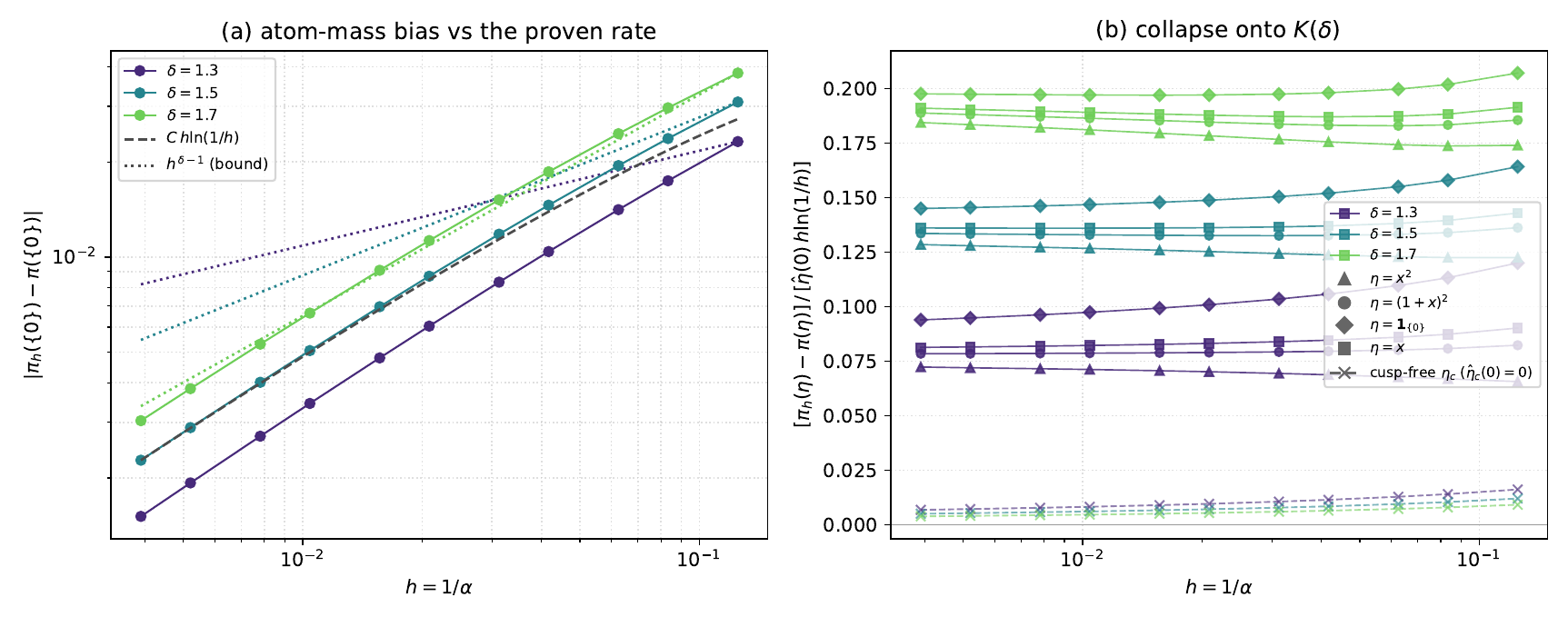}
  \caption{Numerical verification of the sharp-rate expansion \cref{eq:sharp-law}
    ($G(u)=(u-1)^2/2$; stationary law computed on a grid, free of Monte Carlo
    error).  (a)~Atom-mass bias versus $h$: dotted lines are the
    $h^{\delta-1}$ rates of \cref{thm:ula-bias}, the dashed line is
    $C\,h\abs{\log h } $; the measured decay
    is $\delta$-independent and follows the latter.  (b)~The bias divided by
    $\hat\eta(0)\,h\abs{\log h } $ collapses the observables $x$, $x^2$, $(1+x)^2$,
    $\mathbf 1_{\{0\}}$ onto one constant per $\delta$, approaching $K_\star$ of
    \cref{eq:Kstar} as
    $h\downarrow0$; the residual spread is the $O(h)$ remainder, shrinking like
    $1/\abs{\log h } $.  The cusp-free control $\eta_c$ with $\hat\eta_c(0)=0$
    (crosses, divided by $h\abs{\log h } $ only) tends to zero.}
  \label{fig:sharp-rate}
\end{figure}

\section{Conclusion}\label{sec:conclusion}

This paper develops both the theory and the numerics of the one-dimensional sticky CIR process for $\delta \in (1,2)$, the regime in which the boundary is accessible but not absorbing. The invariant measure has the atom-plus-density structure required of sparse priors, and the explicit Green's function of the reflecting CIR resolvent makes exact sampling of the $G \equiv 0$ invariant measure feasible. The Gibbsian reweighting that incorporates a potential $G$ (\cref{thm-inv-potential}) is targeted exactly by the Metropolis--Hastings sampler (\cref{alg:mcmc}), and approximately by an unadjusted Langevin scheme (\cref{alg:ula}) whose total-variation bias admits the first-order expansion $\pi_h-\pi=K_\star\,h\abs{\log h } \,(\delta_0-\pi)+O(h)$ with the explicit constant $K_\star$ (\cref{thm:sharp-rate}). The uniform-in-$h$ geometric ergodicity underpinning these bounds (\cref{ass:contraction}) holds for non-convex potentials with bounded gradient (\cref{lem:contraction-nonconvex}).  Beyond being cheaper per step, the unadjusted scheme is rejection-free, so it never stalls and decorrelates faster, giving the higher effective sample size per second of \cref{fig:ess}; the bias toward $\pi_h$ is the price of this freedom.

The principal direction for future work is the multivariate setting --- the full
Hadamard--Langevin dynamics for the posterior $\pi(d\vec u,d\vec v)\propto
\exp(-\beta G(\vec u\odot \vec v))\,\pi_0(d\vec u,d\vec v)$ of
\cref{sec:motivation} --- and we close by arguing that the one-dimensional
problem solved here is its crux. Under $G\equiv0$, the prior $\pi_0$ is a product,
so the $\vec u$-diffusion is \emph{separable}, a system of $d$ independent sticky
CIR coordinates, and the potential couples them only through the drift $\nabla G$. The
genuinely new analytical difficulty is the accessible sticky boundary: the atom
that delivers sparsity, and the logarithmically enhanced bias it induces, are
codimension-one effects, carried in $d$ dimensions by the faces $\{u_i=0\}$,
where the dynamics normal to a face is exactly the one-dimensional sticky CIR
treated here. We therefore expect the leading multivariate bias to be a sum of
per-face contributions built from the constant $K_\star$ of
\cref{thm:sharp-rate}, with $G'(0)$ replaced by the boundary drift normal to
each face, while the smooth tangential coupling through $G$ is an interior
effect governed by the classical multidimensional Langevin theory. What remains is thus the
\emph{coupling} of the $\vec u$ and $\vec v$ components and the per-face
bookkeeping, not the boundary analysis, which the present paper supplies; and the
applied payoff, genuine coordinate-wise sparsity, is already delivered by the
one-dimensional atom replicated across coordinates.
\appendix
\section{Technical lemmas and proofs for the ULA bias bound}%
\label{app:ula-bias}

\setcounter{theorem}{0}
\renewcommand{\thetheorem}{A.\arabic{theorem}}
We collect several technical lemmas used in our study of the ULA method.

\subsection{Moment bound}
We establish a Foster--Lyapunov drift condition for the ULA kernel, which provides a uniform-in-$h$ moment bound for the ULA chain under a one-sided Lipschitz condition on $G$.

\begin{lemma}[Foster--Lyapunov condition for the ULA kernel]
  \label[lemma]{lem:lyapunov}
  Consider a continuously differentiable function $G\colon [0,\infty)\to\real$ such that  $|G'(x)|\le C(1+x)$ and
  \begin{equation}\label{eq:one-sided}
    x\,G'(x)\ge -c_0 x^2 - c_1, \qquad x\ge 0,
  \end{equation}
  for constants $C,c_1\ge 0$ and $c_0\in[0,\lambda)$.  Set $V(x)=x^2$ and
  $\kappa = c_0+C^2$.   For every $h\le h_0 \coloneqq \lambda/(2\kappa\alpha)$,
  the ULA kernel $K_h$ satisfies
  \[
    K_h V(x)  \le  (1-\gamma)\,V(x) + b, \qquad x\ge 0,
  \]
  with $\gamma = \lambda/(\alpha+2\lambda)>0$ and
  $b = \bigl[\alpha h(2c_1+2C^2)+2\delta/\beta\bigr]/(\alpha+2\lambda)<\infty$,
  both independent of $h\le h_0$ and of the initial condition $u_0$.
  Iterating, for every $n\ge0$,
  \begin{equation}\label{eq:lyapunov-geom}
    K_h^n V(x)\ \le\ (1-\gamma)^n V(x)+\frac b\gamma ;
  \end{equation}
  in particular, \cref{ass:moment} holds with $M = b/\gamma$.
\end{lemma}

\begin{proof}
  For the sticky CIR process with $G=0$ starting at $y\ge 0$, the function
  $t\mapsto e^{2\lambda t}E_y[u_t^2]-\delta/(\lambda\beta)$ is a martingale, giving
  \[
    E_y[u_t^2]  =  e^{-2\lambda t}y^2 + \frac{\delta}{\lambda\beta}(1-e^{-2\lambda t}).
  \]
  Averaging over $T_\alpha\sim\operatorname{Exp}(\alpha)$, and using
  $E[e^{-2\lambda T_\alpha}]=\alpha/(\alpha+2\lambda)$,
  \begin{equation}\label{eq:resolvent-moment}
    E_y[u_{T_\alpha}^2]
     = \frac{\alpha}{\alpha+2\lambda}\,y^2
      +\frac{2\delta}{\beta(\alpha+2\lambda)}.
  \end{equation}

  For $x \in \mathcal{R}_h^c$ (i.e., $x_1 \coloneqq x - G'(x)h > 0$), the ULA uses $y = x_1$ as starting point.
  The assumption~\cref{eq:one-sided} and $|G'(x)|\le C(1+x)$ give
  \[
    x_1^2 = x^2 - 2hxG'(x) + h^2 G'(x)^2
     \le  x^2(1+2c_0 h+2C^2 h^2) + h(2c_1+2C^2 h),
  \]
  where we used $G'(x)^2\le 2C^2(1+x^2)$.  Hence, with
  $\kappa=c_0+C^2$ and $h\le 1$,
  \[
    x_1^2  \le  (1+2\kappa h)\,x^2 + h(2c_1+2C^2).
  \]

  Substituting into~\cref{eq:resolvent-moment} with $y=x_1$:
  \[
    K_h V(x)
     \le \frac{\alpha(1+2\kappa h)}{\alpha+2\lambda}\,x^2 + b,
    \quad
    b = \frac{\alpha h(2c_1+2C^2)+2\delta/\beta}{\alpha+2\lambda}.
  \]
  For $h\le h_0=\lambda/(2\kappa\alpha)$ we have $2\kappa\alpha h\le\lambda$, so
  \[
    \frac{\alpha(1+2\kappa h)}{\alpha+2\lambda}
     \le \frac{\alpha+2\kappa\alpha h}{\alpha+2\lambda}
     \le 1-\frac{\lambda}{\alpha+2\lambda}
     \eqqcolon 1-\gamma.
  \]
  For $x=0$ or $x \in \mathcal{R}_h$, the ULA routes to the boundary kernel and draws directly from the resolvent at $0$, so
  $K_h V(x)=2\delta/[\beta(\alpha+2\lambda)]\le b$.

  \textit{Verification of \cref{ass:moment}.}
  Iterating the drift condition $K_h V \le (1-\gamma)V + b$ gives
  \[
    \mathbb{E}[u_n^2]
    = K_h^n V(u_0)
    \le (1-\gamma)^n u_0^2 + \frac{b}{\gamma},
  \]
  and hence $\limsup_{n\to\infty}\,\mathbb{E}[u_n^2] \le b/\gamma$
  since $(1-\gamma)^n\to 0$.  The bound $b/\gamma$ depends only on
  $\lambda,\beta,\delta,\mu,c_0,c_1$ and is independent of both $h\le h_0$
  and $u_0$.
\end{proof}

\subsection{Regularity} We record the regularity of solutions to the resolvent equation $(I - h\mathcal{L}_1)v = f$.

\begin{lemma}[resolvent regularity]\label[lemma]{lem:resolvent-regularity}
For every $h > 0$ and $f\in L^{\infty}([0,\infty))$,
$v = (I - h\mathcal{L}_1)^{-1} f$
satisfies $\|v\|_\infty \le \|f\|_\infty$ and $v\in C^2((0,\infty))$.  For
$\delta\in(1,2)$, $v$ has a boundary cusp of order $x^{2-\delta}$, and,  as $x\downarrow 0$,
\begin{equation}\label{eq:resolvent-cusp}
  v(x)-v(0)=c_v\,x^{2-\delta}+o(x^{2-\delta}),\qquad
  v'(x)=c_v(2-\delta)\,x^{1-\delta}+o(x^{1-\delta}),
\end{equation}
with $c_v$ generically nonzero. Note that $v$ is in general \emph{not} differentiable
at $0$ in the ordinary sense ($v'(0^+)=+\infty$). The
\emph{scale} derivative $dv/ds$ is finite at the boundary, and
\begin{equation}\label{eq:resolvent-flux}
  \frac{dv}{ds}(0^+)=\tfrac{\alpha}{\mu}\bigl(v(0)-f(0)\bigr)
\end{equation}
 when $f$ is continuous at $0$.
\end{lemma}

\begin{proof}
The $L^\infty$ contraction $\|v\|_\infty \le \|f\|_\infty$ is standard since
$(I-h\mathcal{L}_1)^{-1}$ corresponds to averaging $f$ against the resolvent
kernel of the sticky CIR (a probability kernel for each $x$).  Smoothness of
$v$ on $(0,\infty)$ follows from the smoothness of the kernel away from the
diagonal.

With $z_x=\lambda\beta x^2/2$, $a=1/(2\lambda h)$, and $b=\delta/2\in(\tfrac12,1)$,
the resolvent kernel is built from the fundamental solutions $M(a,b,z_x)$ and
$U(a,b,z_x)$ of \cref{lem:sticky-resolvent}.  The function $M(a,b,z_x)$ is
analytic in $z_x\propto x^2$, which is smooth and even in $x$ with vanishing
$x$-derivative at $0$.  The second solution $U$ has
the connection-formula expansion \cite[(13.2.42)]{Olver2010}
$U(a,b,z)=\tfrac{\Gamma(1-b)}{\Gamma(a-b+1)}M(a,b,z)
+\tfrac{\Gamma(b-1)}{\Gamma(a)}\,z^{1-b}M(a-b+1,2-b,z)$, whose singular part is a
nonzero multiple of $z^{1-b}=(\lambda\beta/2)^{1-\delta/2}x^{2-\delta}$ (as
$\Gamma(\delta/2-1)\neq0$).  Thus, $v$ inherits a term $c_v\,x^{2-\delta}$, giving
\eqref{eq:resolvent-cusp}; the coefficient $c_v$ vanishes only for $f$ in the
codimension-one set annihilating the $U$-projection, so $v'(0^+)=+\infty$
generically.  For the flux, write the resolvent equation as
$\mathcal{L}_1 v=\alpha(v-f)$.  Although the two singular terms of
$\mathcal{L}_1 v=\tfrac1\beta v''+(\tfrac{\delta-1}{\beta x}-\lambda x)v'$ each
blow up like $x^{-\delta}$ on the cusp, they cancel:
$\mathcal{L}_1(x^{2-\delta})=-\lambda(2-\delta)x^{2-\delta}$ is bounded, so
$\mathcal{L}_1 v$ stays bounded as required.  Evaluating the generator at the
sticky point through the speed atom $m(\{0\})=1/\mu$ gives
$\mathcal{L}_1 v(0)=\mu\,\tfrac{dv}{ds}(0^+)$; equating with
$\alpha(v(0)-f(0))$ yields \eqref{eq:resolvent-flux}.
\end{proof}

\subsection{Routing region} The consistency estimate \cref{lem:boundary-charge} is proved directly in Appendix~\ref{app:smooth-bias}.  Here we record a one-step bound on the chance of landing in a near-boundary layer, used in the proofs of \cref{thm:ula-weak-conv,lem:local-tv}; in particular it bounds the mass of the routing region
  \(
    \mathcal{R}_h = \{x > 0\colon  x \leq G'(x)h\},
  \)
  the set of interior points for which the Euler step $x_1 = x - G'(x)h$
  exits the domain.

\begin{lemma}\label[lemma]{lem:density-bound}
Let $\delta \in (1,2)$ and $\alpha=1/h$.  There are $C,h_0>0$, depending only on the model parameters, such that the driftless kernel's one-step probability of landing in a near-boundary layer is bounded uniformly in the starting point,
\begin{equation}\label{eq:layer-onestep}
  \sup_{x\ge0} P_x^{G=0}\bigl(u_{T_\alpha}\in(0,\epsilon]\bigr)\ \le\ C\,\epsilon^\delta,
  \qquad 0<\epsilon\le 2\sqrt h,\ \ h\le h_0 .
\end{equation}
Consequently, if $G'\in C_{\mathrm b}([0,\infty))$, then $\mathcal{R}_h\subseteq(0,\|G'\|_\infty h]$ gives the one-step routing bound $\sup_{x\ge0}K_h(x,\mathcal{R}_h)=O(h^\delta)$; and for any stationary distribution $\pi_h$ of $K_h$ (\cref{thm:ula-exists}), stationarity gives $\pi_h(\mathcal{R}_h)\le\sup_{x\ge0}K_h(x,\mathcal{R}_h)=O(h^\delta)$.
\end{lemma}

\begin{proof}
Fix $0<\epsilon\le2\sqrt h$.  From the resolvent formula \cref{eq:trans} and the sticky resolvent formula (\cref{lem:sticky-resolvent}), for any $x_1\geq 0$,
\[
  P_{x_1}^{G=0}(u_{T_\alpha}\in(0,\epsilon])
  = \int_0^\epsilon \alpha\,G^{\rm sticky}_\alpha(x_1,y)\,m'(y)\,dy,
\]
where \(
  G^{\rm sticky}_\alpha(x_1, y)
   =  {f_0^{\rm sticky}(x_1\wedge y)\,U(a,b,z_{x_1\vee y})}/{|\mathcal{W}|}.
\)
We bound the two factors $U(a,b,z_{x_1\vee y})$ and $f_0^{\rm sticky}(x_1\wedge y)$ separately, uniformly in $h\le h_0$ and $x_1\ge 0$, on the integration range $y\in(0,\epsilon]$.  Since $U(a,b,\cdot)$ is positive and decreasing on $[0,\infty)$ (\cref{rem:exit-sampler}),
\(  U(a,b,z_{x_1\vee y})  \le  U(a,b,0)  =  U_0\).
The argument $x_1\wedge y$ lies in the boundary layer $[0, \epsilon]$, since $x_1\wedge y\le y\le\epsilon$, so we expand $f_0^{\rm sticky}$ around the boundary.  At the boundary, \cref{eq:f0sticky-zero} gives
\[
  f_0^{\rm sticky}(0)  =  p_{\rm leave}  =  \frac{\mu|\mathcal{W}|}{\mu|\mathcal{W}|+\alpha U_0}.
\]
With $\alpha=1/h$, $a = \alpha/(2\lambda)$, $b=\delta/2$, $z_y = \lambda\beta y^2/2$, the expansion parameter on the layer is $a\,z_y = \beta y^2/(4h)\le\beta\epsilon^2/(4h)\le\beta$ (using $\epsilon\le2\sqrt h$), which keeps the Taylor remainder of $M(a,b,z) = 1 + (a/b)z + \cdots$ uniform there (large $a$ balanced by small $z_y$).  The bound $|c_\mu|U_0\le 1$ on the second coefficient of $f_0^{\rm sticky}(x) = M(a,b,z_x) + c_\mu U(a,b,z_x)$ keeps the $U$-contribution under control.  Taylor expansion at $z=0$ then yields
\(  f_0^{\rm sticky}(x_1\wedge y)  =  p_{\rm leave} + O(h)\),
uniformly for $h\le h_0$. Since  $p_{\rm leave} = \mu|\mathcal{W}|h/(\mu|\mathcal{W}|h + U_0) \asymp h^{\delta/2}$ as $h\to 0$, the remainder is absorbed into the leading term:
\(
  f_0^{\rm sticky}(x_1\wedge y)  \le  C\,p_{\rm leave}
\)
for some constant $C$ independent of $h$ and $x_1$.  Substituting the two factor bounds and then the explicit form of $p_{\rm leave}$,
\[
  \alpha\,G^{\rm sticky}_\alpha(x_1,y)
   \le  \frac{C\,\alpha\,p_{\rm leave}\,U_0}{|\mathcal{W}|}
   =  \frac{C\,\alpha\,\mu\,U_0}{\mu|\mathcal{W}|+\alpha U_0}
   \le  C\mu,
\]
where the last step bounds the denominator from below by $\alpha U_0$.
Using $m'(y) = \beta y^{\delta-1}e^{-\lambda\beta y^2/2} \leq \beta y^{\delta-1}$,
\[
  P_{x_1}^{G=0}(u_{T_\alpha}\in(0,\epsilon])
   \leq  C\mu \int_0^\epsilon \beta y^{\delta-1}\,dy
   =  \frac{C\mu\beta}{\delta}\,\epsilon^\delta ,
\]
uniformly in $x_1 \geq 0$, which is \eqref{eq:layer-onestep}.

For the consequences, $K_h(x,\cdot)=P^{G=0}_{\phi_h(x)}(u_{T_\alpha}\in\cdot)$ with $\phi_h(x)\ge0$, so for $h$ small enough that $\|G'\|_\infty h\le2\sqrt h$, taking $\epsilon=\|G'\|_\infty h$ in \eqref{eq:layer-onestep} with $\mathcal{R}_h\subseteq(0,\|G'\|_\infty h]$ gives $\sup_{x}K_h(x,\mathcal{R}_h)\le C(\|G'\|_\infty h)^\delta=O(h^\delta)$; and $\pi_h(\mathcal{R}_h)=\pi_h(K_h(\cdot,\mathcal{R}_h))\le\sup_x K_h(x,\mathcal{R}_h)$ by stationarity.
\end{proof}

\subsection{Consistency assumption}\label{app:smooth-bias}

\begin{proof}[Proof of \cref{lem:boundary-charge}]\label{app:smooth-consistency}
We bound the three parts of $\|\nu_h\|_{V}=|\nu_h(\{0\})|+\int_0^\infty(1+x^2)|r(x)|\,dx$ separately: the atom defect, the far field $(1,\infty)$, and the near-boundary part $(0,1)$.  Throughout, $w\coloneqq d\pi/d\pi_0=e^{-\beta G}/Z$ (\cref{thm-inv-potential}), $\rho_0$ is the density of $\pi_0$, and $A(z)\coloneqq R_h(z,\{0\})=-c_\mu\,U(a,b,z_z)=(1-p_{\rm leave})\,U(a,b,z_z)/U(a,b,0)$ is the $G=0$ atom-hitting probability (\cref{thm:transition-dist}), with  $z_z=\lambda\beta z^2/2$, $a=\alpha/(2\lambda)=1/(2\lambda h)$, $b=\delta/2$.  Recall the  Laplace transform and derivative formula \cite[(13.10.7), (13.3.22)]{Olver2010} for $U$:
\begin{equation}\label{eq:kummer-moment}
 \! \int_0^\infty e^{-z}z^{s-1}U(a,b,z)\,dz=\frac{\Gamma(s)\Gamma(s+1-b)}{\Gamma(a+s+1-b)},
  \;\,
  \frac{d U(a,b,z)}{dz}=-aU(a+1,b+1,z).
\end{equation}

\paragraph{Atom: $|\nu_h(\{0\})|\le Ch^\delta$}
Since $\pi_0$ is $R_h$-invariant (\cref{thm-ergodicity}), its interior and atom masses obey $\int_0^\infty A(y)\,\rho_0(y)\,dy=\pi_0(\{0\})(1-A(0))$.  As $\phi_h(0)=0$, the pushforward-atom mass separates into an atom and an interior part,
\[
  (\pi K_h)(\{0\})=\int_{[0,\infty)}A(\phi_h(y))\,d\pi(y)=A(0)\,\pi(\{0\})+\int_0^\infty A(\phi_h(y))\,w(y)\,\rho_0(y)\,dy.
\]
Subtract $\pi(\{0\})=w(0)\,\pi_0(\{0\})$ and replace $w(0)(1-A(0))\,\pi_0(\{0\})$ via the invariance identity, to find
\begin{equation}\label{eq:atom-closed}
  \nu_h(\{0\})=(\pi K_h)(\{0\})-\pi(\{0\})=\int_0^\infty\bigl[A(\phi_h(y))\,w(y)-A(y)\,w(0)\bigr]\rho_0(y)\,dy.
\end{equation}
This vanishes when $G\equiv0$ ($w$ constant, $\phi_h$ is the identity).  For $G\not\equiv0$, split the integrand into a transport and a weight part,
\[
  A(\phi_h(y))w(y)-A(y)w(0)
   =\underbrace{[A(\phi_h(y))-A(y)]\,w(y)}_{\text{transport}}
   +\underbrace{A(y)\,[w(y)-w(0)]}_{\text{weight}} .
\]
 On the layer $(0,L_h]$, $A(\phi_h)\vee A(y)\le A(0)$ and so $|A(\phi_h)w(y)-A(y)w(0)|\le A(0)(w(y)+w(0))$, which contributes at most $C\,A(0)\,\pi\bigl((0,L_h]\bigr)=O(h^\delta)$.  Off the layer, $\phi_h(y)-y=-hG'(y)$ and $w(y)-w(0)=w'(0)y+O(y^2)$ with $w'(0)=-\beta G'(0)w(0)$. A Taylor expansion in $h$ and in $y$ gives
\begin{gather}
  \begin{split}
  \int_{L_h}^\infty[A(\phi_h(y))-A(y)]\,w(y)\,\rho_0(y)\,dy&=I_{\mathrm{tr}}+O\bigl(h^{1+\delta/2}\bigr),\\
  \int_{L_h}^\infty A(y)\,[w(y)-w(0)]\,\rho_0(y)\,dy&=I_{\mathrm{wt}}+O\bigl(h^{1+\delta/2}\bigr),
\end{split}
\label{eq:tr-wt-expand}
\end{gather}
with leading parts of size $h^{(\delta+1)/2}$
\[
  I_{\mathrm{tr}}=-h\,G'(0)\,w(0)\!\int_0^\infty\! A'(y)\,\rho_0(y)\,dy,
  \qquad
  I_{\mathrm{wt}}=-\beta\,G'(0)\,w(0)\!\int_0^\infty\! A(y)\,y\,\rho_0(y)\,dy.
\]
Here,  the $O(h^{1+\delta/2})$ remainders include the drift variation $h\,(G'(y)-G'(0))\,A'(y)$ and the second-order terms $\tfrac12h^2G'(0)^2A''$, $\tfrac12w''(0)y^2$ (in $(G')^2$ and $G''$).  By \eqref{eq:kummer-moment}, $A'(y)=c_\mu\,a\lambda\beta y\,U(a+1,b+1,z_y)$, and the substitution $z=z_y$ sends $\rho_0(y)\,dy$ to a multiple of $z^{b-1}e^{-z}\,dz$ and $y$ to a multiple of $z^{1/2}$.  Each integral then reduces, via \eqref{eq:kummer-moment} at $s=b+\tfrac12$ --- with kernel $U(a,b,\cdot)$ for $I_{\mathrm{wt}}$ and $U(a+1,b+1,\cdot)$ for $I_{\mathrm{tr}}$ --- to a common factor $F$ times a $\Gamma$-ratio.  Using $\Gamma(\tfrac32)=\tfrac12\Gamma(\tfrac12)$,
\[
  I_{\mathrm{wt}}=-\tfrac12\,\Gamma(\tfrac12)\,F,
  \qquad
  I_{\mathrm{tr}}=h a\lambda\,\Gamma(\tfrac12)\,F,
  \qquad\text{so}\qquad
  I_{\mathrm{tr}}+I_{\mathrm{wt}}=\Gamma(\tfrac12)\,F\,\bigl(h a\lambda-\tfrac12\bigr)=0,
\]
because $ha\lambda=\tfrac12$ ($a=1/(2\lambda h)$). We have shown the leading $h^{(\delta+1)/2}$ cancels exactly.  By \eqref{eq:tr-wt-expand}, $\nu_h(\{0\})$ is then the sum of the layer term $O(h^\delta)$ and the off-layer remainder $O(h^{1+\delta/2})$.  Since $1+\delta/2>\delta$ for $\delta<2$,  $|\nu_h(\{0\})|\le Ch^\delta$.

\paragraph{Far field: $\int_1^\infty(1+x^2)|r(x)|\,dx\le Ch^2$}
This is the classical interior estimate, with the boundary playing no role.  On $[1,\infty)$ the kernel acts as the smooth map $R_h\circ\phi_h$ --- no atom, layer, or cusp.  Since $\pi$ is invariant for the exact sticky-CIR semigroup $P_h\coloneqq e^{h\mathcal L}$ (so $\pi P_h=\pi$), the defect is $\nu_h=\pi K_h-\pi=\pi(K_h-P_h)$.  On the interior, $K_h$ is a first-order splitting of $P_h$: the resolvent satisfies $R_h=e^{h\mathcal L_1}+O(h^2)$ and composes with the transport step to give $K_h-P_h=O(h^2)$, in the density-level sense established for the non-sticky Euler scheme by \citet[Cor.~2.1]{BallyTalay1996b}, whose density-expansion coefficients carry a Gaussian bound $\exp(-c\|x-y\|^2/T)$ that the defect $r$ inherits.  Contributions to $r(x)$ for $x>1$ from the $O(h^\delta)$ boundary region are exponentially small in $1/h$ (the resolvent kernel rarely connects the boundary to $\{x>1\}$ in one step).  Hence, $r(x)=O(h^2)$ on $[1,\infty)$, locally uniformly, and the $(1+x^2)$-weighted integral converges to give the stated bound.  

\paragraph{Near-boundary: $\int_0^1(1+x^2)|r(x)|\,dx\le Ch^\delta$}
As $1+x^2\le2$ on $[0,1]$, it suffices to bound $\int_0^1|r(x)|\,dx=\sup\{\nu_h(\psi):\psi\in C_{\mathrm c}^\infty(0,1),\ \snorm{\psi}_\infty\le1\}$.  Fix such a $\psi$ and set $\chi\coloneqq R_h\psi$; then $\snorm{\chi}_\infty\le1$ ($R_h$ is Markov), and by \cref{lem:resolvent-regularity} the product $\rho_0\chi'$ (equivalently $m'\chi'$) is bounded near $0$ and the scale-flux $\tfrac{d\chi}{ds}(0^+)$ is finite.

Since $\phi_h(0)=0$ and $\psi(0)=0$, the kernel identity $K_h\psi=\chi\circ\phi_h$ gives the exact decomposition (writing $\rho\coloneqq w\rho_0$ for the interior density of $\pi$)
\begin{equation}\label{eq:incr-telescope}
  \nu_h(\psi)=\pi(\{0\})\,\chi(0)+\int_0^\infty\bigl(\chi\circ\phi_h-\psi\bigr)\rho\,dx .
\end{equation}
The atom charge is removed by the invariance $\pi_0R_h=\pi_0$ of the $G=0$ kernel (\cref{thm-ergodicity}): evaluating $\pi_0(R_h\psi)=\pi_0(\psi)$ on $\pi_0=\pi_0(\{0\})\delta_0+\rho_0\,dx$ with $\psi(0)=0$ gives $\pi_0(\{0\})\chi(0)=\int_0^\infty(\psi-\chi)\rho_0\,dx$, hence $\pi(\{0\})\chi(0)=w(0)\int_0^\infty(\psi-\chi)\rho_0\,dx$.  Substituting into \eqref{eq:incr-telescope} and splitting $\chi\circ\phi_h-\psi=(\chi\circ\phi_h-\chi)+(\chi-\psi)$, the $w(0)$ term combines with the $(\chi-\psi)\rho=(\chi-\psi)w\rho_0$ part to leave
\begin{equation}\label{eq:key-defect}
  \nu_h(\psi)=\int_0^\infty(\chi-\psi)\,(w-w(0))\,\rho_0\,dx
    +\int_0^\infty(\chi\circ\phi_h-\chi)\,\rho\,dx .
\end{equation}
\emph{First-order term.}  In the first integral, $\chi-\psi=h\mathcal{L}_1\chi$. In the scale--speed form \eqref{eq:generator-scale-speed}, $\mathcal{L}_1\chi=\frac{d}{dm}\frac{d\chi}{ds}=\frac{1}{m'}\frac{d}{dx}\bigl(\chi'/s'\bigr)$ and $s'm'=\beta$ by \eqref{eq:scale-speed-cir}. Using $\rho_0\propto m'$, this gives $(\mathcal{L}_1\chi)\,\rho_0=\tfrac1\beta(\rho_0\chi')'$.  Integrating by parts,
\[
  \int_0^\infty(\chi-\psi)(w-w(0))\rho_0\,dx
   =\tfrac{h}{\beta}\Bigl[\rho_0\chi'\,(w-w(0))\Bigr]_0^\infty
    -\tfrac{h}{\beta}\!\int_0^\infty\!\rho_0\chi'\,w'\,dx
   =h\!\int_0^\infty\! G'\chi'\,\rho\,dx ,
\]
where the boundary term at $0$ vanishes because $\rho_0\chi'$ is bounded and $(w-w(0))(0)=0$ (and at $\infty$ by decay of $\rho_0$), and $w'=-\beta G'w$ with $\rho=w\rho_0$.  Thus,
\begin{equation}\label{eq:nu-residual}
  \nu_h(\psi)=h\!\int_0^\infty\! G'\chi'\,\rho\,dx+\int_0^\infty(\chi\circ\phi_h-\chi)\rho\,dx .
\end{equation}
On $(L_h,\infty)$, where $\phi_h(y)=y-hG'(y)$, Taylor's theorem gives $\chi\circ\phi_h-\chi=-hG'\chi'+\tfrac12h^2(G')^2\chi''+O(h^3\chi''')$, and the $-hG'\chi'$ term cancels the part of $h\int_0^\infty G'\chi'\rho$ over $(L_h,\infty)$; collecting,
\begin{equation}\label{eq:nu-three}
  \nu_h(\psi)=\underbrace{\tfrac12h^2\!\int_{L_h}^\infty\!(G')^2\chi''\rho\,dx}_{\text{bulk}}
   +\underbrace{\int_0^{L_h}\!\!\bigl[(\chi\circ\phi_h-\chi)+hG'\chi'\bigr]\rho\,dx}_{\text{layer}}
   +O\!\Bigl(h^3\!\int_{L_h}^\infty\!|\chi'''|\rho\,dx\Bigr).
\end{equation}
\emph{Bulk.}  Integrate by parts twice, moving both derivatives off the bounded factor $\chi$:
\[
  \tfrac12h^2\!\int_{L_h}^\infty\!(G')^2\chi''\rho
   =\tfrac12h^2\Bigl[(G')^2\rho\,\chi'-\bigl((G')^2\rho\bigr)'\chi\Bigr]_{L_h}^\infty
   +\tfrac12h^2\!\int_{L_h}^\infty\!\chi\,\bigl((G')^2\rho\bigr)''dx .
\]
As $((G')^2\rho)''=O(x^{\delta-3})$ near $0$ and $L_h=h\snorm{G'}_\infty$, the integral is $O(L_h^{\delta-2})=O(h^{\delta-2})$ and contributes $O(h^2\cdot h^{\delta-2})=O(h^\delta)$ via $\snorm{\chi}_\infty\le1$.  The boundary terms at $L_h$ are $O(h^\delta)$ too: the $\chi'$-term because $\rho\chi'$ is bounded (\cref{lem:resolvent-regularity}), so $(G')^2\rho\chi'|_{L_h}=O(1)$ and $\tfrac12h^2\cdot O(1)=O(h^2)$, and the $\chi$-term because $((G')^2\rho)'(L_h)=O(h^{\delta-2})$. \\\emph{Layer.}  On $(0,L_h]$, $|\chi\circ\phi_h-\chi|\le2\snorm{\chi}_\infty$ integrated against $\pi((0,L_h])=O(h^\delta)$ gives $O(h^\delta)$, while $|hG'\chi'\rho|\le h\snorm{G'}_\infty|\rho\chi'|$ with $\rho\chi'$ bounded, integrated over the length-$L_h$ layer, gives $O(h^2)$.\\  \emph{Remainder.}  $\chi'''=O(x^{-1-\delta})$ near $0$ (third derivative of the $x^{2-\delta}$ cusp), so $|\chi'''|\rho=O(x^{-2})$ and $h^3\int_{L_h}^\infty x^{-2}\,dx=O(h^3 L_h^{-1})=O(h^2)$.  

Hence, $\nu_h(\psi)=O(h^\delta)$ uniformly over $\snorm{\psi}_\infty\le1$, so the near-boundary bound $\int_0^1(1+x^2)|r(x)|\,dx\le 2\int_0^1|r(x)|\,dx\le Ch^\delta$ holds.

Combining the three parts, $\|\nu_h\|_V=|\nu_h(\{0\})|+\int_0^\infty(1+x^2)|r(x)|\,dx\le Ch^\delta$, which is \eqref{eq:boundary-charge}.
\end{proof}

\subsection{Geometric ergodicity}\label{app:contraction-nonconvex}

We show \cref{ass:contraction} holds for any potential with bounded gradient, by the classical route of a
Foster--Lyapunov drift together with a minorisation. 
Throughout this subsection, we assume
\begin{equation}\label{eq:nonconvex-hyp}
  G\in C^2([0,\infty)),\qquad \|G'\|_\infty\eqqcolon B<\infty.
\end{equation} 

\begin{proposition}[geometric ergodicity]\label[proposition]{lem:contraction-nonconvex}
  Under \eqref{eq:nonconvex-hyp}, \cref{ass:contraction} holds: the clamped kernel $K_h$ of
  \eqref{eq:clamped-kernel} is $V$-uniformly geometrically ergodic, uniformly in $h$, at rate
  $1-ch$ for some $c>0$, with $V(x)=1+x^2$.
\end{proposition}

For the proof, we recall the drift condition established in \cref{lem:lyapunov} and prove a minorisation condition for $K_h$ with respect to $[0,M_0]$ in \cref{lem:interior-minorisation}. We first show a minorisation for interior compact sets (i.e., excluding initial data at the boundary point $\{0\}$).

\begin{lemma}[minorisation on interior]\label[lemma]{lem:scheme-density}
  Let $I,[a,b]\subset(0,\infty)$ be compact and let $T_\star>0$; assume \eqref{eq:nonconvex-hyp}. There
  are thresholds $\theta,T_0>0$, depending only on $\delta,\lambda,\beta$ and the position of
  $I\cup[a,b]$, such that if
  \begin{equation}\label{eq:bulk-geom}
    \operatorname{diam}(I\cup[a,b])\ \le\ \theta
    \qquad\text{and}\qquad
    T_\star\ \le\ T_0 ,
  \end{equation}
  then there exist $c_\star,h_2>0$ with
  \begin{equation}\label{eq:scheme-lower}
    K_h^{m}(x',\cdot)\ \ge\ c_\star\,\mathrm{Leb}|_{[a,b]},
    \qquad\text{ for $x'\in I,\ \ h\le h_2,\ \ |m-T_\star/h|\le1$.}
  \end{equation}
\end{lemma}

\begin{proof}
  Set $D\coloneqq \operatorname{diam}(I\cup[a,b])$ and $a_0\coloneqq \operatorname{dist}(I\cup[a,b],0)$, and define
  the bounded interval
  \(
    J\coloneqq \bigl(\tfrac12 a_0,\ \max(I\cup[a,b])+\tfrac12 a_0\bigr),
  \)
  so that $I\cup[a,b]\subset J$ and $d_J\coloneqq \operatorname{dist}(I\cup[a,b],\partial J)=\tfrac12 a_0$. The
  sticky CIR has constant diffusion $\sigma^2=2/\beta$ and drift
  $b(x)=\tfrac{\delta-1}{\beta x}-\lambda x-G'(x)$ (\cref{eq-CIR}); on $J$ the drift is
  bounded (the $1/x$ term is regular away from $0$ and $\|G'\|_\infty=B<\infty$) and $\sigma^2$ is
  constant. Extend the drift off $J$ to a globally bounded function, leaving $\sigma^2$ constant and
  unchanged on $J$; the extended coefficients then satisfy the uniform-ellipticity and
  bounded-drift, H\"older-diffusion hypotheses \textup{(UE)}, \textup{(SB)} of \citet{LemaireMenozzi2010}. Write $q^{m}_h(x',\cdot)$ for the $m$-step
  ($K_h^{m}$) density of the scheme with these extended coefficients.

  \emph{(i) Discrete two-sided Gaussian bound.} By \citet[Thm.~2.1]{LemaireMenozzi2010}, the Euler
  scheme  obeys the following two-sided Gaussian bound: for $0<T_\star\le1$,
  \begin{equation}\label{eq:scheme-aronson}
    C^{-1}\,\mathsf{g}_{1/c}(T_\star;x',v)\ \le\ q^{m}_h(x',v)\ \le\ C\,\mathsf{g}_{c}(T_\star;x',v),
  \end{equation}
  where $\mathsf{g}_\kappa(t;x,y)$ is the Gaussian density $\mathcal{N}(x,\,t/\kappa)$ at $y$,
  with $c\in(0,1]$, $C\ge1$ independent of $T_\star$ and $h$. The result is stated for the Euler
  scheme; our chain $K_h=R_h\circ\phi_h$ integrates the diffusion \emph{exactly} through the resolvent
  $R_h=(I-h\mathcal L_1)^{-1}$ and applies the drift $-G'$ by the Lipschitz map $\phi_h$
  (\cref{eq:clamped-kernel}), so its one-step law from $x$ has mean $x+hb(x)+O(h^2)$ and variance
  $\sigma^2 h+O(h^2)$, matching the Euler step to $O(h^2)$. Now \eqref{eq:scheme-aronson} is proved by a
  discrete parametrix expansion that builds the $m$-step density by chaining the one-step transition
  law, using it only through its mean, variance, and Gaussian tail; this construction is not specific
  to the Euler scheme but holds for general Markov chains converging to the diffusion
  \citep{KonakovMammen2000}. The matched one-step data therefore yield the same bound for $q^{m}_h$
  (for $h\le h_2$), with the exact diffusion step only removing the Euler discretisation error. No smoothness of $G$ beyond
  $\|G'\|_\infty<\infty$ is used.

  \emph{(ii) Interior lower bound.} Set the threshold $\theta\coloneqq\tfrac12 c\,d_J\,(\le d_J)$. For
  $D\le\theta$, $x'\in I$ and $v\in[a,b]$, the lower bound in \eqref{eq:scheme-aronson} with
  $|x'-v|\le D$ gives
  \begin{equation}\label{eq:cts-lower}
    q^{m}_h(x',v)\ \ge\ C^{-1}\sqrt{\tfrac{1}{2\pi c\,T_\star}}\,\exp\!\Bigl(-\tfrac{D^2}{2c\,T_\star}\Bigr)\eqqcolon c_0>0,
    \qquad \forall x'\in I,\ v\in[a,b],\ h\le h_2 .
  \end{equation}

  \emph{(iii) Locality and the small-$T_\star$ choice.} The bound \eqref{eq:scheme-aronson} is for the
  extended coefficients. Restoring the original sticky-CIR coefficients (which agree with the
  extension on $J$) changes them only \emph{outside} $J$, at distance $\ge d_J$ from the segment
  $[x',v]\subset J$. Coupling the extended and original schemes until their first exit from $J$, the
  two $m$-step densities at $v\in[a,b]$ can differ only through paths that leave $J$---a displacement
  $\ge d_J$ from $[x',v]$---so by the off-diagonal upper bound of \eqref{eq:scheme-aronson} the
  original sticky-CIR density differs from $q^{m}_h(x',v)$ by at most
  $C\sqrt{c/(2\pi T_\star)}\,\exp(-c\,d_J^2/(2T_\star))$. Since $D\le\theta=\tfrac12 c\,d_J$ gives
  $D^2/c\le\tfrac14 c\,d_J^2<c\,d_J^2$, the ratio of this error to $c_0$ as $T_\star\downarrow 0$ satisfies
  \[
    C^2 c\,\exp\!\Bigl(-\frac{c\,d_J^2-D^2/c}{2T_\star}\Bigr)\to 0.
  \]
  Hence, there is $T_0>0$ such that for $T_\star\le T_0$ the locality error is $\le\tfrac12 c_0$, and the
  original sticky-CIR density of $K_h^{m}(x',\cdot)$ is $\ge c_0-\tfrac12 c_0=\tfrac12 c_0\eqqcolon
  c_\star$ on $[a,b]$, $x'\in I$, which is \eqref{eq:scheme-lower}.
\end{proof}

We extend this minorisation to include initial data up to and including the boundary.
\begin{lemma}[minorisation]\label[lemma]{lem:interior-minorisation}
  Assume \cref{eq:nonconvex-hyp}, and let $C=[0,M_0]$ be compact. There exists a time $T>0$, an
  interval $[a,b]\subset(0,M_0)$, a constant $\varepsilon>0$ and $h_0>0$ such that, writing
  $m_h=\ceil*{T/h}$,
  \begin{equation}\label{eq:interior-minor}
    K_h^{m_h}(x,\cdot)\ \ge\ \varepsilon\,\mathrm{Leb}|_{[a,b]}\qquad\text{for all }x\in C,\ h\le h_0 .
  \end{equation}
\end{lemma}

\begin{proof}
  Let $P_t=e^{t\mathcal L}$ be the transition semigroup of the sticky CIR process with potential
  $G$. On $(0,\infty)$, this is a non-degenerate diffusion, with generator
  $\mathcal L=\tfrac12\sigma^2\partial_x^2+b(x)\partial_x$, constant diffusion coefficient
  $\sigma^2=2/\beta$, and drift $b(x)=\tfrac{\delta-1}{\beta x}-\lambda x-G'(x)$ (\cref{eq-CIR}); by
  \eqref{eq:nonconvex-hyp}, the drift is smooth on $(0,\infty)$ and, with its derivatives, bounded on
  every interval $[\rho,\rho']\subset(0,\infty)$, while $\sigma^2$ is constant, hence uniformly
  elliptic. Fix a point $x_0\in(0,M_0)$ and
  take intervals $I,[a,b]\subset(0,M_0)$ around $x_0$ of diameter at most the threshold $\theta$ of
  \cref{lem:scheme-density}, and $T$ small enough that $T_\star\coloneqq T/2\le T_0$, so that
  \eqref{eq:bulk-geom} holds. We bound $K_h^{m_h}(x,\cdot)$ below
  by routing through $I$ at the half-way time: by the Chapman--Kolmogorov inequality, with
  $m_h'=\floor*{m_h/2}$ and $m_h''=m_h-m_h'$,
  \begin{equation}\label{eq:ck-split}
    K_h^{m_h}(x,A)\ \ge\ \int_I K_h^{m_h'}(x,dx')\,K_h^{m_h''}(x',A),\qquad A\subset[a,b].
  \end{equation}
  We bound the two factors separately.

  \emph{Reaching the interior set $I$.} By \cref{thm:ula-weak-conv},
  $K_h^{\floor*{T_\star/h}}(x,\cdot)\to P_{T_\star}(x,\cdot)$ weakly, uniformly for $x\in C$. The open set $I^\circ$ satisfies $P_{T_\star}(x,I^\circ)\ge c_1>0$ for
  all $x\in C$: the sticky boundary $0$ is regular, so the process leaves it and reaches every
  interior set, whence $P_{T_\star}(x,I^\circ)>0$ for each $x$. The sticky CIR is Feller (\cref{thm-exist-potential}), so $x\mapsto P_{T_\star}(x,I^\circ)$ is lower
  semicontinuous. A lower semicontinuous, strictly positive function
  attains its positive minimum on the compact $C$. Then, $\liminf_{h\to0}K_h^{m_h'}(x,I^\circ)
  \ge P_{T_\star}(x,I^\circ)\ge c_1$, uniformly on $C$; so for $h\le h_1$,
  \begin{equation}\label{eq:reach-bulk}
    K_h^{m_h'}(x,I)\ \ge\ \tfrac12 c_1\qquad \forall x\in C.
  \end{equation}
  This step asks only that the chain reach an interior \emph{set} with positive probability; no
  density estimate, and no rate, is needed, so the degeneracy and stickiness at $0$ are
  immaterial here.

  \emph{Interior minorisation.} By \cref{lem:scheme-density}, there are $c_\star>0$ and $h_2>0$ with
  \begin{equation}\label{eq:step2-lower}
    K_h^{m_h''}(x',\cdot)\ \ge\ c_\star\,\mathrm{Leb}|_{[a,b]}\qquad \forall x'\in I,\ h\le h_2.
  \end{equation}

  Finally, for $h\le h_0\coloneqq \min(h_1,h_2)$ and any measurable $A$, by \eqref{eq:ck-split},
  \eqref{eq:reach-bulk} and \eqref{eq:step2-lower},
  \[
    K_h^{m_h}(x,A)\ \ge\ \int_I K_h^{m_h'}(x,dx')\,c_\star\,\mathrm{Leb}(A\cap[a,b])
    \ \ge\ \tfrac12 c_1 c_\star\,\mathrm{Leb}(A\cap[a,b]),
  \]
  which is \eqref{eq:interior-minor} with $\varepsilon=\tfrac12 c_1 c_\star$.
\end{proof}

\begin{proof}[Proof of \cref{lem:contraction-nonconvex}]
  Hypothesis \eqref{eq:nonconvex-hyp} implies the conditions of \cref{lem:lyapunov}: $|G'|\le B$
  gives linear growth, and $xG'(x)\ge-Bx\ge-\tfrac\lambda2 x^2-\tfrac{B^2}{2\lambda}$ gives the
  one-sided bound with $c_0=\lambda/2<\lambda$. Let $V(x)=1+x^2$, and let
  $\gamma=\lambda/(\alpha+2\lambda)$ and $b$ be the drift constants of \cref{lem:lyapunov}, for which
  $b/\gamma$ is independent of $h$. Fix the compact sublevel set $C\coloneqq\{V\le R_0\}=[0,M_0]$ for any
  $R_0>1+b/\gamma$.
  Take $T,[a,b],\varepsilon,h_0$ from \cref{lem:interior-minorisation} and set $P\coloneqq
  K_h^{m_h}$, $m_h=\ceil*{T/h}$. 
  
  The iterated drift \eqref{eq:lyapunov-geom} of
  \cref{lem:lyapunov}, with $n=m_h$, gives a geometric drift for $P$,
  \[
    PV\le\kappa V+b/\gamma,\qquad \kappa\coloneqq(1-\gamma)^{m_h},
  \]
  which is $h$-uniform: since $\gamma=\lambda/(\alpha+2\lambda)=\lambda h+O(h^2)$ and $m_h=T/h+O(1)$,
  one has $\gamma\,m_h\to\lambda T$, so $\kappa=e^{-\lambda T}\bigl(1+o(1)\bigr)$ is bounded away from
  $1$ uniformly in $h$, and $b/\gamma$ is $h$-independent. Moreover, \eqref{eq:interior-minor} is a
  minorisation of $P$ on the return set $C$, with constant $\varepsilon$ and reference probability
  measure $\nu=(b-a)^{-1}\mathrm{Leb}|_{[a,b]}$. A geometric drift together with a minorisation on
  the drift's return set is precisely the hypothesis of the quantitative Harris theorem
  \citep[Thm.~1.3]{HairerMattingly2011} (see also \citealp[Thm.~15.0.1]{MeynTweedie1993}): $P$ has
  a unique invariant law $\pi_h$, which is then the unique $K_h$-invariant law, and
  \[
    \|P^k(x,\cdot)-\pi_h\|_V\le \bar C\,V(x)\,\bar\rho^{\,k},
  \]
  with $\bar C<\infty$ and $\bar\rho\in(0,1)$ depending only on $(\gamma,b,\varepsilon,T)$ and uniform in $h$. Unwinding $P=K_h^{m_h}$ (so $m_h=T/h+\order{1}$ steps per $P$-step), and
  using $\|K_h\mu-K_h\mu'\|_V\le \|\mu-\mu'\|_V$ on intermediate steps, gives
  \[
    \|K_h^n(x,\cdot)-\pi_h\|_V\le \bar C'\,V(x)\,\bar\rho^{\,\floor*{n/m_h}}
      =\bar C'\,V(x)\,(1-ch+\order{h^2})^{\,n},\qquad c=\frac{|\ln\bar\rho|}{T}>0 .
  \]
  This is \eqref{eq:ass-contraction}, uniformly in $h$.
\end{proof}

\subsection{Proof of the first-order expansion (\cref{thm:sharp-rate})}\label{app:sharp-rate}

Throughout this subsection $\eta$ is bounded measurable, $\hat\eta\coloneqq \eta-\pi(\eta)$, and all constants $C$ depend only on the model parameters $(\lambda,\beta,\delta,\mu)$ and $\|G'\|_{C^2_{\mathrm b}}$; dependence on $\snorm{\eta}_\infty$ is always displayed.  We use the scale--speed pairs of \cref{sec:background}: for the driftless sticky CIR, $s'(x)=x^{-(\delta-1)}e^{\lambda\beta x^2/2}$, $m'(x)=\beta x^{\delta-1}e^{-\lambda\beta x^2/2}$, $m(\{0\})=1/\mu$, and for the tilted process $s_G'=s'e^{\beta G}$, $m_G'=m'e^{-\beta G}$, $m_G(\{0\})=e^{-\beta G(0)}/\mu$, under which $\pi=m_G/Z$ with $Z=m_G([0,\infty))$.  Note the two identities used repeatedly:
\begin{equation}\label{eq:sm-identities}
  s_G'(x)\,m_G'(x)=\beta,
  \qquad
  \rho(x)\coloneqq \frac{m_G'(x)}{Z}=\rho_\star\,x^{\delta-1}\bigl(1+O(x)\bigr),
\end{equation}
where $\rho_\star\coloneqq \beta e^{-\beta G(0)}/Z=\mu\beta\,\pi(\{0\})$.
The argument has three ingredients: regularity of the continuous Poisson solution (\cref{lem:poisson-cont}); the preconditioned stationarity identity \cref{eq:precond-identity} (\cref{lem:precond}), which eliminates the resolvent kernel  and the Kummer boundary layer from the bias exactly; and a near-boundary total-variation difference bound (\cref{lem:local-tv}).  These combine to give the expansion; the one subtlety is that the difference term is bounded by a vanishing multiple of $\|\pi_h-\pi\|_{\mathrm{TV}}$ itself, so the crude bound of \cref{thm:ula-bias} is sharpened by a single re-substitution before the $O(h)$ remainder emerges.

The continuous Poisson solution $g_\star$ of $\mathcal Lg_\star=-\hat\eta$ is the explicit scale--speed solution.  Since $\mathcal L=\frac{d}{dm_G}\frac{d}{ds_G}$ \eqref{eq:generator-scale-speed}, we integrate twice: against $dm_G$ for the flux $\Phi\coloneqq dg_\star/ds_G$ (the constant fixed by $\Phi(\infty)=0$, forced by admissibility as $s_G'$ grows super-polynomially), then against $ds_G$. This gives
\begin{equation}\label{eq:pois-rep}
  \Phi(y)\coloneqq \int_{(y,\infty)}\hat\eta\,dm_G,
  \qquad
  g_\star(x)\coloneqq \int_0^x s_G'(y)\,\Phi(y)\,dy .
\end{equation}
The next lemma records the regularity and boundary behaviour of $g_\star$.

\begin{lemma}[the continuous Poisson solution]\label[lemma]{lem:poisson-cont}
  The function $g_\star$ of \eqref{eq:pois-rep} lies in $C([0,\infty))\cap C^1((0,\infty))$ with $g_\star'$ locally absolutely continuous, solves $\mathcal Lg_\star=-\hat\eta$ a.e.\ on $(0,\infty)$, and has the boundary flux
  \begin{equation}\label{eq:pois-flux}
    \frac{dg_\star}{ds_G}(0^+)=-\hat\eta(0)\,\frac{e^{-\beta G(0)}}{\mu}
    \qquad\bigl(\text{equivalently }\ \mathcal Lg_\star(0)=-\hat\eta(0)\bigr).
  \end{equation}
  For $0<x\le1$, with $c_F\coloneqq -\hat\eta(0)/\mu$, 
  \begin{equation}\label{eq:pois-cusp}
  \begin{aligned}
    g_\star'(x)&=c_F\,x^{1-\delta}+O\bigl(\snorm{\eta}_\infty x^{2-\delta}\bigr),\\
    g_\star''(x)&=(1-\delta)\,c_F\,x^{-\delta}+E(x),
    \qquad |E(x)|\le C\snorm{\eta}_\infty\bigl(1+x^{1-\delta}\bigr).
  \end{aligned}
  \end{equation}
For $x\ge1$, $|g_\star'(x)|\le C\snorm{\eta}_\infty/x$ and $|g_\star''(x)|\le C\snorm{\eta}_\infty$ a.e.; and $|g_\star(x)|\le C\snorm{\eta}_\infty\bigl(1+\log(1+x)\bigr)$.
\end{lemma}

\begin{proof}
  \emph{Poisson equation and regularity.}  By construction $\Phi=dg_\star/ds_G$ and $d\Phi=-\hat\eta\,dm_G$ on $(0,\infty)$, so reading these through the generator \eqref{eq:generator-scale-speed}, $\mathcal Lg_\star=\frac{d}{dm_G}\Phi=-\hat\eta$ a.e.\ on $(0,\infty)$.  Since $\hat\eta$ is bounded and $m_G$ finite, $\Phi$ is bounded and absolutely continuous, so $g_\star'=s_G'\Phi$ is locally absolutely continuous ($s_G'$ being smooth on $(0,\infty)$).

  \emph{Boundary flux.}  The constant of the first integration is fixed by centering: $\pi(\hat\eta)=0$ makes the full $m_G$-integral of $\hat\eta$ (atom included) vanish, so
  $\Phi(0^+)=\int_{(0,\infty)}\hat\eta\,dm_G=-\hat\eta(0)\,m_G(\{0\})$, which is \cref{eq:pois-flux}. In other words, the sticky boundary value $\mathcal Lg_\star(0)=\Phi(0^+)/m_G(\{0\})=-\hat\eta(0)$ falls out of (decay at $\infty$) and (centering), with no extra condition imposed.  Moreover, $|\Phi(y)-\Phi(0^+)|\le\snorm{\hat\eta}_\infty\,m_G((0,y])\le C\snorm{\eta}_\infty y^\delta$.

  \emph{Asymptotics.}  Near $0$, $s_G'(x)=x^{1-\delta}e^{\beta G(0)}(1+O(x))$, so $g_\star'=s_G'\Phi$ gives the first part of \cref{eq:pois-cusp} (note $e^{\beta G(0)}\Phi(0^+)=c_F$); integrating, $|g_\star(x)-g_\star(0)|\le C\snorm{\eta}_\infty x^{2-\delta}$, and $g_\star\in C([0,\infty))$.
  Differentiating $g_\star'=s_G'\Phi$ a.e.\ and using $s_G'm_G'=\beta$,
  \begin{equation}\label{eq:gpp-ode}
    g_\star''=\frac{s_G''}{s_G'}\,g_\star'-\beta\hat\eta,
    \qquad
    \frac{s_G''(x)}{s_G'(x)}=\frac{1-\delta}{x}+\lambda\beta x+\beta G'(x),
  \end{equation}
  and inserting the expansion of $g_\star'$ gives the second part of \cref{eq:pois-cusp}.
  For $x\ge1$: a Gaussian-tail (Mills-ratio) bound gives $m_G((x,\infty))\le C\,m_G'(x)/x$, hence $|\Phi(x)|\le C\snorm{\eta}_\infty m_G'(x)/x$ and $|g_\star'(x)|=s_G'|\Phi|\le C\beta\snorm{\eta}_\infty/x$ by \cref{eq:sm-identities}; then \cref{eq:gpp-ode} with $|s_G''/s_G'|\le Cx$ gives $|g_\star''|\le C\snorm{\eta}_\infty$.  Integrating the $g_\star'$ bounds gives the logarithmic growth of $g_\star$.
\end{proof}

\begin{lemma}[preconditioned stationarity identity]\label[lemma]{lem:precond}
  Let $K_h=R_h\circ\phi_h$ be the clamped kernel \eqref{eq:clamped-kernel} and define, for $x>0$,
  \begin{equation}\label{eq:Psi-def}
    \Psi_h(x)\coloneqq G'(x)\,g_\star'(x)-h^{-1}\bigl(g_\star(x)-g_\star(\phi_h(x))\bigr).
  \end{equation}
  Then, for all $h\le h_0$,  we have
  $\pi_h(\eta)-\pi(\eta)=\pi_h\bigl(\Psi_h\,\mathbf 1_{(0,\infty)}\bigr)$ (which is \cref{eq:precond-identity}).
\end{lemma}

\begin{proof}
  Recall that $\mathcal L_1=\frac{d}{dm}\frac{d}{ds}$ is the $G=0$ sticky generator \eqref{eq:generator-scale-speed}.  Note that, on $(0,\infty)$,  $\mathcal L_1\varphi=\mathcal L\varphi+G'\varphi'$, and the $m$-derivative taken against the point mass gives $\mathcal L_1\varphi(0)=\mu\,(d\varphi/ds)(0^+)$ for any $\varphi$ whose flux $(d\varphi/ds)(0^+)$ exists (see \cref{eq:wentzell}). By \cref{lem:poisson-cont}, the flux of  $\varphi=g_\star$ is well-defined; moreover, writing the flux identity \cref{eq:pois-flux} in the driftless scale via $d/ds=e^{\beta G}\,d/ds_G$, the factors $e^{\pm\beta G(0)}$ cancel and
  \begin{equation}\label{eq:flux-driftless}
    (dg_\star/ds)(0^+)=-\hat\eta(0)/\mu,
    \qquad\text{(equivalently, $\mathcal L_1g_\star(0)=-\hat\eta(0)=\mathcal Lg_\star(0)$)}.
  \end{equation}
  This means the boundary condition is \emph{potential-independent}.  We define $f\coloneqq (I-h\mathcal L_1)g_\star$ \emph{pointwise} on $[0,\infty)$; in particular $f(0)=g_\star(0)-h\mu\,(dg_\star/ds)(0^+)=g_\star(0)+h\hat\eta(0)$.

  \emph{Step 1: $R_hf=g_\star$.}  By \cref{def:resolvent}, the resolvent $R_h=(I-h\mathcal L_1)^{-1}=\alpha(\alpha-\mathcal L_1)^{-1}$ at rate $\alpha=1/h$ is the one-step $G=0$ transition expectation, $R_hf(x)=\mathbb E^{G=0}_x[f(u_{T_\alpha})]$ (\cref{thm:transition-dist,lem:resolvent-regularity}).  For data $f$ that are $m$-integrable near $0$ and of at most polynomial growth, we will show that  $v\coloneqq R_hf$ is a polynomially bounded solution of the boundary-value problem
  \begin{equation}\label{eq:precond-bvp}
    v-h\mathcal L_1v=f\ \text{  on }(0,\infty),
    \qquad
    v(0)-h\mu\,\tfrac{dv}{ds}(0^+)=f(0).
  \end{equation}
    The candidate $g_\star$ solves \cref{eq:precond-bvp} \emph{by the definition of $f$}.  What requires proof is therefore that $R_hf$ solves \cref{eq:precond-bvp} and that polynomially bounded solutions are unique; together these force $R_hf=g_\star$.

  \emph{(Solvability and uniqueness)} Combining the pointwise splitting $\mathcal L_1=\mathcal L+G'\partial_x$ on $(0,\infty)$ with the Poisson equation and the $g_\star'$-bounds of \cref{lem:poisson-cont}, \begin{equation}
    \mathcal L_1g_\star=-\hat\eta+G'g_\star'\label{eq:l1gstar}
  \end{equation}
  and $\mathcal L_1g_\star$ is $O(x^{1-\delta})$ near $0$ and bounded at $\infty$. Thus, $f$ is $L^1(m)$ near $0$ and grows at most logarithmically. In other words, $g_\star$ and $R_hf=\mathbb E_x[f(u_{T_\alpha})]$ grow at most logarithmically.  The Green representation $R_hf(x)=w_0(x)f(0)+\alpha\int_0^\infty G^{\rm sticky}_\alpha(x,y)\,f(y)\,m'(y)\,dy$ converges and solves \cref{eq:precond-bvp}, the variation-of-parameters computation of \cref{lem:sticky-resolvent} needing only $f\in L^1_{\mathrm{loc}}(m)$.  Two polynomially bounded solutions differ by a $w$ with $(\alpha-\mathcal L_1)w=0$ and homogeneous boundary relation; the interior solutions are $\operatorname{span}\{M(a,b,z_x),U(a,b,z_x)\}$ (\cref{lem:sticky-resolvent}), and the $M$-mode ($\sim e^{z_x}$, Gaussian) is excluded by the polynomial growth, so $w=c\,U(a,b,z_\cdot)$.  The boundary relation then reads
  \[
    c\,\bigl[\,U_0-h\mu\,(dU/ds)(0^+)\,\bigr]=c\,\bigl[\,U_0+h\mu|\mathcal W|\,\bigr]=0,
  \]
  using $(dU/ds)(0^+)=-|\mathcal W|$ and $U(a,b,0)=U_0>0$ (\cref{lem:sticky-resolvent}); the bracket is strictly positive, so $c=0$ and $R_hf=g_\star$.

  \emph{Step 2: pairing with $\pi_h$.}  Since $K_h=R_h\circ\phi_h$, Step 1 gives $K_hf=(R_hf)\circ\phi_h=g_\star\circ\phi_h$. By \cref{ass:moment},  $\pi_h(V)\le1+M$. Further,  the interior density of $\pi_h$ behaves like $m'$ near $0$ (\cref{thm:ula-exists}; the kernel bound in \cref{lem:density-bound}). As $f\in L^1(m)$ and grows at most logarithmically,  $f$ is $\pi_h$-integrable. Stationarity $\pi_h(K_hf)=\pi_h(f)$ therefore gives
  \[
    0=\pi_h\bigl(g_\star\circ\phi_h-g_\star\bigr)+h\,\pi_h(\mathcal L_1g_\star).
  \]
  Substitute $\mathcal L_1g_\star=-\hat\eta+G'g_\star'\,\mathbf 1_{(0,\infty)}$ and $\pi_h(\hat\eta)=\pi_h(\eta)-\pi(\eta)$.  At $x=0$, the first substitution reads $\mathcal L_1g_\star(0)=-\hat\eta(0)$, which is \cref{eq:flux-driftless}.  Since $\phi_h(0)=0$, the atom drops from the first term; dividing by $h$ and rearranging gives
  $\pi_h(\eta)-\pi(\eta)=\pi_h\bigl(\Psi_h\,\mathbf 1_{(0,\infty)}\bigr)$ for $\Psi_h(x)=G'(x)\,g_\star'(x)-h^{-1}\bigl(g_\star(x)-g_\star(\phi_h(x))\bigr)$, as required.
\end{proof}

\begin{lemma}[near-boundary total-variation difference]\label[lemma]{lem:local-tv}
  Under the assumptions of \cref{thm:ula-bias}, there are $C<\infty$ and $h_0>0$ such that, for all $h\le h_0$ and $0<\epsilon\le2\sqrt h$, the total variation of $\pi_h-\pi$ on the near-boundary interval $(0,\epsilon]$ obeys
  \[
    |\pi_h-\pi|\bigl((0,\epsilon]\bigr)\;\le\;C\bigl(h^{\delta}+h^{\delta-1}\epsilon^\delta\,\lvert\log\epsilon\rvert\bigr).
  \]
\end{lemma}

\begin{proof}  Exactly as in the proof of \cref{thm:ula-bias}, $\pi_h(\varphi)-\pi(\varphi)=\nu_h(g_h^\epsilon)$ for $g_h^\epsilon$ defined in \cref{lem:poisson-sup}. By \cref{lem:boundary-charge},
  \(
    |\pi_h(\varphi)-\pi(\varphi)|%
    \le\|\nu_h\|_V\,|g_h^\epsilon|_V,%
  \)
    and $|\pi_h-\pi|\bigl((0,\epsilon]\bigr)=\sup|\pi_h(\varphi)-\pi(\varphi)|$ over measurable $\varphi$ with $|\varphi|\le1$ vanishing off $(0,\epsilon]$. Hence,
   the supremum over $\varphi$ gives the claim if $|g_h^\epsilon|_V\le C\bigl(h^\delta+h^{\delta-1}\epsilon^\delta\,\lvert\log\epsilon\rvert\bigr)$. We prove this in two steps.

  \emph{One-step envelope.}  Since $K_h(x,\cdot)=P^{G=0}_{\phi_h(x)}(u_{T_\alpha}\in\cdot)$, the one-step layer bound \eqref{eq:layer-onestep} of \cref{lem:density-bound} (valid for $\epsilon\le2\sqrt h$) gives
  \begin{equation}\label{eq:onestep-envelope}
    \sup_{x\ge0}K_h\bigl(x,(0,\epsilon]\bigr)\le C_1\,\epsilon^\delta,
    \qquad \epsilon\le2\sqrt h,\ h\le h_0 .
  \end{equation}
  Hence, by stationarity, $\pi_h((0,\epsilon])\le C_1\epsilon^\delta$ and $|\pi_h(\varphi)|\le C_1\epsilon^\delta$.

  \emph{Occupation bound on the discrete Poisson solution.}  We have $g_h^\epsilon= \sum_{n\ge0}K_h^n\bigl(\varphi-\pi_h(\varphi)\bigr)$.  For $n\ge1$, conditioning on the last step and using \cref{eq:onestep-envelope}, $|K_h^n\varphi(x)|\le\sup_yK_h(y,(0,\epsilon])\le C_1\epsilon^\delta$. Together with $|\pi_h(\varphi)|\le C_1\epsilon^\delta$, we see that $|K_h^n\varphi(x)-\pi_h(\varphi)|\le 2C_1\epsilon^\delta$. From geometric erogidicity \cref{eq:ass-contraction}, $|K_h^n\varphi(x)-\pi_h(\varphi)|\le CV(x)(1-ch)^n$.  Splitting the series at $N\coloneqq \lceil(ch)^{-1}\delta\,\lvert\log\epsilon\rvert\rceil$ and using the smaller bound on each side,
  \[
    |g_h^\epsilon(x)|\le 2+2C_1\epsilon^\delta N+\frac{C}{ch}V(x)(1-ch)^N
    \le C\,V(x)\bigl(1+h^{-1}\epsilon^\delta\,\lvert\log\epsilon\rvert\bigr),
  \]
as required. 
\end{proof}

\begin{proof}[Proof of \cref{thm:sharp-rate}]
  Write $\Psi\coloneqq \Psi_h$ of \cref{eq:Psi-def} and set $c_0\coloneqq 1+2\|G'\|_\infty$, so that $\phi_h(x)=x-hG'(x)$ and $\phi_h(x)\ge x/2$ for $x\ge c_0h$.  All error terms below carry the factor $\snorm{\eta}_\infty$, which we suppress.  By \cref{lem:precond},
  \begin{equation}\label{eq:bias-split}
    \pi_h(\eta)-\pi(\eta)=\pi(\Psi\mathbf 1_{(0,\infty)})+(\pi_h-\pi)(\Psi\mathbf 1_{(0,\infty)}) .
  \end{equation}

  \emph{Step 1: envelopes for $\Psi$.}  For $x\ge c_0h$, since $g_\star'$ is locally absolutely continuous (\cref{lem:poisson-cont}),
  \begin{equation}\label{eq:Psi-int-rem}
    \Psi(x)=h^{-1}\int_x^{\phi_h(x)}\bigl(\phi_h(x)-y\bigr)\,g_\star''(y)\,dy.
  \end{equation}
  The bounds of \cref{lem:poisson-cont} and $|\phi_h(x)-x|\le h\|G'\|_\infty$ give
  \begin{equation}\label{eq:Psi-envelopes}
    |\Psi(x)|\le
    \begin{cases}
          C\,h\,x^{-\delta},&  x\in[c_0 h,1],\\
       C\,h,&  x\ge1.
    \end{cases} 
  \end{equation}
  On $(0,c_0h]$, the clamp satisfies $\phi_h(x)\in[0,x+L_h]$, so $|g_\star(x)-g_\star(\phi_h(x))|\le\int_0^{(1+c_0)h}|g_\star'|\le Ch^{2-\delta}$ and
  \begin{equation}\label{eq:Psi-layer}
    |\Psi(x)|\le C\bigl(x^{1-\delta}+h^{1-\delta}\bigr)\le Ch^{1-\delta},\qquad 0<x\le c_0h .
  \end{equation}

  \emph{Step 2: the main term in \cref{eq:bias-split}.}  We claim
  \begin{equation}\label{eq:main-term}
    \pi(\Psi\mathbf 1_{(0,\infty)})=K_\star\,\hat\eta(0)\,h\abs{\log h } +O(h).
  \end{equation}
  On $(0,c_0h]$, \cref{eq:Psi-layer,eq:sm-identities} give $\pi((0,c_0h])\le Ch^\delta$ and an $O(h)$ contribution; the same holds on $[1,\infty)$ by \cref{eq:Psi-envelopes} and $\pi([1,\infty))\le1$.  
  
  On $(c_0h,1]$, insert $g_\star''(y)=(1-\delta)c_F\,y^{-\delta}+E(y)$ from \cref{eq:pois-cusp} into \cref{eq:Psi-int-rem}.  The $E$-part of $\Psi$ is bounded by $Ch(1+x^{1-\delta})$, which integrates against $\rho$ to $O(h)$.  For the cusp part, expanding $y^{-\delta}$ about $x$ gives $\int_x^{\phi_h(x)}(\phi_h(x)-y)y^{-\delta}dy=\tfrac12h^2G'(x)^2x^{-\delta}(1+O(h/x))$. The $O(h/x)$ correction integrates against $\rho\le C x^{\delta-1}$ over $(c_0h,1]$ to $O(h^2\cdot h^{-1})=O(h)$ after the prefactor $h$. Putting this together,
  \[
    \pi\bigl(\Psi\mathbf 1_{(c_0h,1]}\bigr)
    =\frac{(1-\delta)c_F}{2}\,h\int_{c_0h}^1G'(x)^2\,x^{-\delta}\rho(x)\,dx+O(h).
  \]
  By \cref{eq:sm-identities}, $G'(x)^2x^{-\delta}\rho(x)=G'(0)^2\rho_\star\,x^{-1}+O(x^{-1}\cdot x)$ on $(0,1]$, where the remainder integrates to $O(1)$ and $\int_{c_0h}^1x^{-1}dx=\abs{\log h } +O(1)$.  Since $(1-\delta)c_F=(\delta-1)\hat\eta(0)/\mu$ and $\rho_\star=\mu\beta\pi(\{0\})$, the coefficient is
  $\tfrac12(\delta-1)\beta\,G'(0)^2\pi(\{0\})\,\hat\eta(0)=K_\star\hat\eta(0)$, proving \cref{eq:main-term}.

  \emph{Step 3: the difference term in \cref{eq:bias-split}.}  Decompose $(0,\infty)=(0,c_0h]\cup(c_0h,\sqrt h\,]\cup(\sqrt h,1]\cup(1,\infty)$ and bound each piece by $|(\pi_h-\pi)(\Psi\mathbf 1_A)|\le(\sup_A|\Psi|)\,|\pi_h-\pi|(A)$, with the envelopes \cref{eq:Psi-envelopes,eq:Psi-layer} for $\sup_A|\Psi|$ and \cref{lem:local-tv} for $|\pi_h-\pi|(A)$ near the boundary.
  By \cref{eq:Psi-layer} and \cref{lem:local-tv} at $\epsilon=c_0h$,
  $|(\pi_h-\pi)(\Psi\mathbf 1_{(0,c_0h]})|\le Ch^{1-\delta}\bigl(h^\delta+h^{\delta-1}h^\delta\abs{\log h } \bigr)=O(h)$.
  On the intervals $A_j=(\epsilon_j,\epsilon_{j+1}]$, $\epsilon_j=2^jc_0h\le\sqrt h$, \cref{eq:Psi-envelopes} and \cref{lem:local-tv} at $\epsilon_{j+1}$ give
  \[
    \bigl|(\pi_h-\pi)(\Psi\mathbf 1_{A_j})\bigr|
    \le Ch\epsilon_j^{-\delta}\bigl(h^\delta+h^{\delta-1}\epsilon_{j+1}^\delta\,\lvert\log\epsilon_{j+1}\rvert\bigr)
    \le C\bigl(h^{1+\delta}\epsilon_j^{-\delta}+h^\delta\abs{\log h } \bigr).
  \]
  Summing over the $O(\abs{\log h } )$ intervals, the first parts form a geometric series dominated by $O(h)$ and the second parts total $O(h^\delta\log^2(h))=O(h)$ (as $\delta>1$).
  On $(\sqrt h,1]$, $|\Psi|\le Ch^{1-\delta/2}$ by \cref{eq:Psi-envelopes}, so by \cref{eq:ula-tv},
  $|(\pi_h-\pi)(\Psi\mathbf 1_{(\sqrt h,1]})|\le Ch^{1-\delta/2}\cdot h^{\delta-1}=Ch^{\delta/2}$.
  On $(1,\infty)$, $|(\pi_h-\pi)(\Psi)|\le Ch\cdot h^{\delta-1}=O(h^\delta)$.
  Combining with \cref{eq:bias-split,eq:main-term}, we see $|\pi_h(\eta)-\pi(\eta)|\le C\bigl(h\abs{\log h } +h^{\delta/2}\bigr)\snorm{\eta}_\infty\le Ch^{\delta/2}\snorm{\eta}_\infty$. Then, taking the supremum over $\snorm{\eta}_\infty\le1$,
  \begin{equation}\label{eq:tv-boot}
    \|\pi_h-\pi\|_{\mathrm{TV}}\le C\,h^{\delta/2}.
  \end{equation}

  \emph{Step 4: re-substitution.}  Re-estimate the only piece that used the crude bound \cref{eq:ula-tv}, now with the sharpened \cref{eq:tv-boot}:
  $|(\pi_h-\pi)(\Psi\mathbf 1_{(\sqrt h,1]})|\le Ch^{1-\delta/2}\cdot h^{\delta/2}=O(h)$ (and the far-field piece improves likewise).  All parts of the difference term are now $O(h)$, so \cref{eq:bias-split,eq:main-term} give \cref{eq:sharp-law}.
\end{proof}

\bibliography{ref}
\end{document}